\documentclass[11pt,a4paper,reqno]{amsart}
\usepackage{amsfonts,amssymb,amsmath,amsthm,graphicx,color,amscd,xspace,verbatim}
\usepackage{scalerel}
\usepackage{constants}
\makeatletter

\makeatletter
\def\@tocline#1#2#3#4#5#6#7{\relax
  \ifnum #1>\c@tocdepth % then omit
  \else
    \par \addpenalty\@secpenalty\addvspace{#2}%
    \begingroup \hyphenpenalty\@M
    \@ifempty{#4}{%
      \@tempdima\csname r@tocindent\number#1\endcsname\relax
    }{%
      \@tempdima#4\relax
    }%
    \parindent\z@ \leftskip#3\relax \advance\leftskip\@tempdima\relax
    \rightskip\@pnumwidth plus4em \parfillskip-\@pnumwidth
    #5\leavevmode\hskip-\@tempdima
      \ifcase #1
       \or\or \hskip 1em \or \hskip 2em \else \hskip 3em \fi%
      #6\nobreak\relax
      \dotfill
      \hbox to\@pnumwidth{\@tocpagenum{#7}}
    \par
    \nobreak
    \endgroup
  \fi}
\makeatother

\newconstantfamily{c}{symbol=c}

\newcommand{\conl}{\Cl[c]}
\newcommand{\conr}{\Cr}

\newtheorem{theorem}{Theorem}[section]
\newtheorem{lemma}[theorem]{Lemma}
\newtheorem{proposition}[theorem]{Proposition}

\newtheorem{corollary}[theorem]{Corollary}

\theoremstyle{definition}
\newtheorem{definition}[theorem]{Definition}

\theoremstyle{remark}

\newcommand{\N}{{\mathbb N}}
\newcommand{\R}{{\mathbb R}}

\newcommand{\tr}{\mathrm{tr}^*}

\newcommand{\beqn}{\begin{eqnarray}}
\newcommand{\eeqn}{\end{eqnarray}}   % equation with number
\newcommand{\beq}{\begin{eqnarray*}}
\newcommand{\eeq}{\end{eqnarray*}}

\newcommand{\be}{\begin{equation}}
\newcommand{\bel}[1]{\begin{equation}\label{#1}}
\newcommand{\ee}{\end{equation}}%% This macro does not work with amstex.
\newcommand{\BA}{\begin{array}}
\newcommand{\EA}{\end{array}}
\newcommand{\BAN}{\renewcommand{\arraystretch}{1.2}
\setlength{\arraycolsep}{2pt}\begin{array}}
\newcommand{\BAV}[2]{\renewcommand{\arraystretch}{#1}
\setlength{\arraycolsep}{#2}\begin{array}}
\newcommand{\BSA}{\begin{subarray}}
\newcommand{\ESA}{\end{subarray}}
\newcommand{\BAL}{\begin{aligned}}
\newcommand{\EAL}{\end{aligned}}

\newcommand{\forevery}{\quad \forall}

\newcommand{\norm}[1]{\left \|#1\right \|}%% adjustable norm
\newcommand{\dist}{\mathrm{dist}}
\newcommand{\sign}{\mathrm{sign}}

\newcommand{\prt}{\partial}

\newcommand{\sbs}{\subset}

\def\bcom{\begin{comment}}
\def\ecom{\end{comment}}

   \def\CB{{\mathcal B}}   
      
   \def\CH{{\mathcal H}}   
\def\CJ{{\mathcal J}}      \def\CL{{\mathcal L}}

\def\BBG {\mathbb G}   \def\BBH {\mathbb H}    
   \def\BBK {\mathbb K}    
   \def\BBN {\mathbb N}    
\def\BBP {\mathbb P}   \def\BBR {\mathbb R}

\def\GTM {\mathfrak M}

\def\tr{\mathrm{tr}_{\mu}}

\def\div{\mathrm{div}}

\newcommand{\ei}{\phi_{\xm}}
\newcommand{\xa}{\alpha}
\newcommand{\xb}{\beta}
\newcommand{\xg}{\gamma}
\newcommand{\xG}{\Gamma}
\newcommand{\xd}{\delta}
\newcommand{\xD}{\Delta}
\newcommand{\xe}{\varepsilon}
\newcommand{\xz}{\zeta}

\newcommand{\xl}{\lambda}
\newcommand{\xL}{\Lambda}
\newcommand{\xm}{\mu}
\newcommand{\xn}{\nu}

\newcommand{\xr}{\rho}
\newcommand{\xs}{\sigma}

\newcommand{\xf}{\phi}
\newcommand{\xF}{\Phi}

\newcommand{\xo}{\omega}
\newcommand{\xO}{\Omega}

\newcommand{\myfrac}[2]{{\displaystyle \frac{#1}{#2} }}
\newcommand{\myint}[2]{{\displaystyle \int_{#1}^{#2}}}
\newcommand{\ap}{{\xa_{\scaleto{+}{3pt}}}}

\newcommand{\xgp}{{\xg_{+}}}
\newcommand{\xgn}{{\xg_{-}}}

\def\bal#1\eal{\small\begin{align*}#1\end{align*}\normalsize}
\def\ba#1\ea{\small\begin{align}#1\end{align}\normalsize}
\numberwithin{equation}{section}

\newcommand{\minsk}{\vspace{0.2cm}}
\newcommand{\medsk}{\vspace{0.4cm}}

\newcommand{\inprod}[2]{{\langle{#1},{#2}\rangle}}

\newcommand{\trti}{ \widetilde{ \rm tr}}

 \newcommand{\tarr}[6]{{\left\{\begin{array}{lll}
   {#1}&{#2}\\
   {#3}&{#4}\\
   {#5}&{#6}
\end{array}\right.}}

\newcommand{\ia}{({\rm i})}
\newcommand{\ib}{({\rm ii})}
\newcommand{\ic}{({\rm iii})}

\begin{document}

\makeatletter
\def\author@andify{%
  \nxandlist {\unskip ,\penalty-1 \space\ignorespaces}%
    {\unskip {} \@@and~}%
    {\unskip \penalty-2 \space \@@and~}%
}
\makeatother
\title[]{Heat and Martin kernel estimates for Schr\"{o}dinger operators
with critical Hardy potentials}

\author[G. Barbatis]{G. Barbatis}
\address{G. Barbatis, Department of Mathematics, National and Kapodistrian University of Athens, 15784 Athens, Greece}
\email{gbarbatis@math.uoa.gr}

\author[K. T. Gkikas]{K.T. Gkikas}
\address{K.T. Gkikas, Department of Mathematics, National and Kapodistrian University of Athens, 15784 Athens, Greece}
\email{kugkikas@math.uoa.gr}

\author[A. Tertikas]{A. Tertikas}
\address{A. Tertikas, Department of Mathematics \& Applied Mathematics,
University of Crete, 70013 Heraklion, Greece and
Institute of Applied and Computational Mathematics, FORTH, 71110 Heraklion, Greece}
\email{tertikas@uoc.gr}

\date{\today}

\begin{abstract}
Let $\Omega$ be a bounded domain in $\R^N$ with $C^2$ boundary
and let $K\subset\partial\Omega$ be either a $C^2$ submanifold of
the boundary of codimension $k<N$ or a point.
In this article we study various problems related to the Schr\"odinger
operator $L_{\mu} =-\Delta - \mu d_K^{-2}$
where $d_K$ denotes the distance to $K$ and $\mu\leq k^2/4$.
We establish parabolic boundary Harnack inequalities as well as
related two-sided heat kernel and Green function estimates.
We construct the associated Martin kernel and prove existence and
uniqueness for the corresponding boundary value problem
with data given by measures. Next we apply the results to the
study of $L_\xm u+g(u) = 0$ and establish existence
and uniqueness under suitable assumptions on the function $g$.
To prove our results we introduce among other things
a suitable notion of boundary trace. This trace is different
from the one used by Marcus  and Nguyen \cite{MT} thus
allowing us to cover the whole range $\mu\leq k^2/4$.

\medskip



\end{abstract}

\maketitle

\tableofcontents

\section{Introduction}

The study of linear Schr\"odinger operators with singular potentials is
central in the  theory of parabolic and elliptic partial differential equations.
In recent years in particular there has been an intense study of operators with Hardy potentials, see e.g.
\cite{BM,hardy-marcus,BMS,FT,BFT,DD2,GM,C,FF,BEL,BFT3}.

Throughout this work we assume that $\Omega$ is a bounded $C^2$ domain;
we note however that some of the results presented in this introduction
are valid under weaker regularity assumptions.

Consider the problem
\be
\left\{
\begin{array}{ll}
 u_t =\Delta u + V(x) u, &  x\in\Omega , \; t>0, \\
u=0 ,  & x \in \partial\xO , \; t>0, \\
 u(0,x)=u_0(x) , &  x\in \xO \, ,
\end{array}
\right.
\label{111}
\ee
where $V\in L^1_{\rm loc}(\Omega)$ and set
\[
\xl^* =\inf_{ C^{\infty}_c(\Omega) }\frac{\int_\xO|\nabla w|^2dx-\int_\xO
V w^2 dx}{\int_\xO w^2 dx}.
\]
Cabr\'e and Martel \cite{CM} have established that if $\lambda^*>-\infty$
then for regular enough initial data there exists a global in time weak solution of \eqref{111} which in addition satisfies an exponential in time
bound. Conversely, the existence of a weak solution which satisfies an
exponential bound implies that $\lambda^*> -\infty$.
In the prototype case of the Hardy potential $V(x)=c|x|^{-2}$ this
has already been studied by Baras and Goldstein \cite{BG}.

Given the existence of a weak solution one natural question is the existence and asymptotic behaviour of the heat kernel and Green function.
If the potential is not too singular then the asymptotic behaviour of the
heat kernel for small time is the same as that of the Laplacian, namely
\begin{eqnarray*}
&& C^{-1}\Big( \frac{d(x)d(y)}{ (d(x)+\sqrt{t})(d(y)+\sqrt{t})}\Big)
t^{-\frac{N}{2}}\exp\Big(-C\frac{|x-y|^2}{t}\Big)\\
& &\hspace{-1.2cm}\leq h(t,x,y)\leq
 C\Big(\frac{d(x)d(y)}{(d(x)+\sqrt{t})(d(y)+\sqrt{t})}\Big)
t^{-\frac{N}{2}}\exp\Big(-C^{-1}\frac{|x-y|^2}{t}\Big),
\end{eqnarray*}
where $d(x)={\rm dist}(x,\partial\Omega)$ denotes the distance to the boundary, see e.g. \cite{Zhang}.

In the case of a more singular potential such as a Hardy
potential, the problem has been studied in
\cite{DS1,D2,VZ,MS,LS,BFT2,MoT,MoT2,FMT2,FMoT}.

A distinction that plays an important role in this context is whether
the singularity of the Hardy potential occurs in the interior or
on the boundary of the domain.
For the potential $\mu |x|^{-2}$, $0\leq \mu\leq (\frac{N-2}{2})^2$,
where $0\in\Omega$, for small time we have
\begin{eqnarray*}
&& C^{-1}\Big( \frac{d(x)d(y)}{ (d(x)+\sqrt{t})(d(y)+\sqrt{t})}\Big)
\Big( \frac{ |x| \; |y| }{ (|x|+\sqrt{t})(|y|+\sqrt{t})}\Big)^{\theta_+}
t^{-\frac{N}{2}}\exp\Big(-C\frac{|x-y|^2}{t}\Big)\\
&\leq  & h(t,x,y) \\
&\leq & C\Big(\frac{d(x)d(y)}{(d(x)+\sqrt{t})(d(y)+\sqrt{t})}\Big)
 \Big( \frac{ |x| \; |y| }{ (|x|+\sqrt{t})(|y|+\sqrt{t})}\Big)^{\theta_+}
t^{-\frac{N}{2}}\exp\Big(-C^{-1}\frac{|x-y|^2}{t}\Big) ,
\end{eqnarray*}
where $\theta_+$ is the largest solution to the equation
$\theta^2 +(N-2)\theta + \mu=0$; see \cite{FMT2}.
This estimate was generalized in \cite{gktai} in case where
the distance is taken
from a closed surface $\Sigma\subset\Omega$
of codimension $k$, $2\leq k\leq N$; see also
\cite{GkT1,GkT2} for more results within this framework.

On the other hand, when the distance is taken
from the boundary $\partial\Omega$ the following
small time estimate is valid for the
heat kernel of the operator $-\Delta -\mu d(x)^{-2}$,
$0\leq\mu\leq \frac{1}{4}$,
\begin{eqnarray*}
&& C^{-1}\Big( \frac{d(x)d(y)}{ (d(x)+\sqrt{t})(d(y)+
\sqrt{t})}\Big)^{1+\theta_+}
t^{-\frac{N}{2}}\exp\Big(-C\frac{|x-y|^2}{t}\Big)\\
&\leq  & h(t,x,y) \;
\leq  C\Big(\frac{d(x)d(y)}{(d(x)+\sqrt{t})(d(y)+\sqrt{t})}\Big)^{1+
\theta_+}
t^{-\frac{N}{2}}\exp\Big(-C^{-1}\frac{|x-y|^2}{t}\Big) ,
\end{eqnarray*}
where $\theta_+$ is the largest solution to the equation
$\theta^2 +\theta +\mu=0$, see  \cite{FMT2,FMoT}.

Another function that is important in the study of this type of problems
is the Martin kernel \cite{An,hunt,Mar}.
Ancona proved the existence of the Martin kernel $ K_{\mu,\partial\xO}(x,y)$ of $L_\mu^{\partial\xO}=-\Delta -\frac{\mu}{d^2}$, $\mu<\frac{1}{4}$,
with pole at $y$,  which is unique up to a normalization (see \cite[Theorem 3]{An}). He showed that for any positive solution $u$ of $L_\mu^{\partial\xO} u = 0$ there exists a unique nonnegative Radon measure $\xn$ on $\partial\xO$ such that
\ba\label{repr1}
u(x)=\int_{\partial \Omega} K_{\mu,\partial\xO}(x,y)d\xn(y).
\ea
The case $\xm=\frac{1}{4}$ was treated by Gkikas and V\'eron in \cite{GkV}. In particular, they showed that the representation formula \eqref{repr1} holds true provided the bottom of the spectrum of  $L_\mu^{\partial\xO}$
is positive.

When $K\subset \xO$ is a closed smooth surface of codimension
$k\in \{ 3 ,\ldots, N\}$, analogous results where obtained in
 \cite{gktai} for the operator
 $L_\mu^{K}=-\Delta -\frac{\mu}{d_K^2}$, $\mu\leq\frac{(k-2)^2 }{4}$,
under the assumption that  the bottom of the spectrum of
$L_\mu^{K}$ is positive.

Our aim in this article is to study such problems in the case where the
Hardy potential involves the distance to a smooth submanifold of the boundary, including the case of a boundary point.
In this direction:

$\bullet$ We establish parabolic boundary Harnack inequalities as well as
related two-sided
heat kernel estimates. For small time, our approach is based on the
ideas of Grigoryan and Saloff–Coste \cite{GS} (see also \cite{SC1}), while for
large time, we exploit the work of Davies in \cite{D1,D2} to obtain sharp-
two sided heat kernel estimates; see also \cite{FMT2,FMoT}.

$\bullet$ In the spirit of \cite{caffa,hunt} (see also \cite{gktai,GkV}), we
construct the Martin kernel of $L_\xm$ in $\xO$ and we prove the
uniqueness also in the critical case. Using the heat kernel estimates, we
obtain sharp pointwise estimates for the Green function as well as the
Martin kernel. We also  show that every nonnegative $L_\xm$-harmonic
function (i.e. solution of $L_\xm u = 0$ in $\xO$ in the sense of
distributions) can be represented as the integral
of the Martin kernel with respect to a finite measure on $\partial\xO$.

$\bullet$ Using the properties of the Green function
and Martin kernel we study the boundary value problem with data given by
measures. Following Marcus-V\'eron \cite{MVbook}
we prove existence, uniqueness as well as a representation formula for any
solution of this problem. As an application, we establish existence
and uniqueness results for
solutions of the equation $L_\xm u+g(u) = 0$ under suitable
assumptions on the function $g$.

\section{Main results}

Throughout this article we consider a bounded $C^2$ domain
$\xO\subset\mathbb{R}^N$, $N\geq3$
and a $C^2$ compact submanifold without boundary
$K \subset \partial\Omega$  of codimension $k$, $1\leq k \leq N$.
For the extreme cases $k=N$ and $k=1$ we
assume that  $K = \{0\}$  and $K=\partial\xO$ respectively.
We set $d_K(x)=\text{dist}(x,K)$ and define the operator
\[
L_\xm=-\xD-\frac{\xm}{d_K^2} \; , \qquad  \mbox{ in }\xO,
\]
where   $\xm$ is a parameter; we shall always assume that
$\mu\leq \frac{k^2}{4}$ so that $L_\xm$ is bounded from below.
The study of the parabolic equation $u_t+L_\xm u=0$
with Dirichlet boundary conditions is strongly related with the
minimization problem,
\bal
C_{\xO,K}=\inf_{u\in H_0^1(\xO)\setminus \{0\}}\frac{\int_\xO|\nabla u|^2dx}{\int_\xO\frac{|u|^2}{d_K^2}dx}.
\eal
It is well known that $0<C_{\xO,K}\leq\frac{k^2}{4}$ (see, e.g., \cite{FF}).

For $\mu\leq \frac{k^2}{4}$ the infimum
\ba
\label{eigenvalue}
\xl_\xm:=\inf_{u\in H_0^1(\xO)\setminus \{0\}}\frac{\int_\xO|\nabla u|^2dx-\xm\int_\xO\frac{u^2}{d_K^2}dx}{\int_\xO u^2 dx}
\ea
is finite. Moreover, if  $\xm< \frac{k^2}{4},$ then there exists a minimizer $\ei\in H_0^1(\xO)$ of \eqref{eigenvalue};
see \cite{FF} for more details.
In addition, by \cite[Lemma 2.2]{MT} the eigenfunction $\ei$ satisfies
\be
\xf_\xm(x) \asymp d(x)d_K^{\xg_+}(x),  \quad\quad \mbox{ in }\Omega \, ,
\label{eigenest}
\ee
provided $\xm<C_{\xO,K}$. Here and below we denote by $\xg_+$ (resp. $\xg_-$) the largest (resp.
the smallest) solution of the equation
$\gamma^2+k\gamma+\mu=0$.

On the other hand, if $\xm=\frac{k^2}{4}$ then there is no $H_0^1(\xO)$
minimizer.  However, there exists a function $\xf_{\mu}\in H^1_{loc}(\xO)$ such that $L_{\mu}
\xf_{\mu}=\xl_{\mu}\xf_{\mu}$ in $\xO$ in the sense of distributions.
In the Appendix we follow ideas of \cite{BMS,DD1,DD2,FMoT} and prove
that estimate \eqref{eigenest} is valid for any $\xm\leq \frac{k^2}{4}$.

\subsection{Heat kernel and boundary Harnack inequality}

Let $u\in C^1((0,\infty): C^2(\xO)),$ then by setting $u=e^{-\xl_\xm t}\ei v,$ we can easily see that
\begin{equation}
\label{st}
\frac{u_t+L_\xm u}{\ei}=v_t- \ei^{-2} \, \text{div}\left(\ei^2\nabla  v\right)=:v_t+\CL_{\mu}v.
\end{equation}
Hence, instead of studying the properties of the operator $L_\xm,$ it is more convenient to study the operator $\frac{\partial}{\partial t}+\CL_{\mu}.$ In this direction, we introduce the weighted Sobolev space $H^1(\xO;\ei^2)$.

\begin{definition}
Let $D\subset\xO$ be an open set. We denote by $H^1(D;\ei^2)$ the weighted Sobolev space
\bal
H^1(D;\ei^2):=\{u\in H^1_{loc}(D):\;|u|\ei+|\nabla u|\ei\in L^2(D)\}
\eal
endowed with the norm
\bal
\norm{u}^2_{H^1(D;\ei^2)}=\int_D u^2\ei^2 dx+\int_D|\nabla u|^2\ei^2 dx.
\eal
\end{definition}
\noindent
We also denote by $H^1_0(D;\ei^2)$ the closure of $C_c^\infty(D)$  in the norm $\norm{\cdot}_{H^1(D;\ei^2)}$.
It is worth mentioning here that $H^1_0(\xO;\ei^2)=H^1(\xO;\ei^2)$ (see Theorem \ref{density}).

Next, we normalize $\ei$ so that $\int _\xO\ei^2 dx=1$. We define the bilinear form
$Q: H_0^1(\xO ; \ei^2)\times H_0^1(\xO ;\ei^2)\to \R$  by
\bal
Q(u,v)=\int_\xO\nabla u \cdot \nabla v \, \ei^2 dx.
\eal
The associated operator is the operator $\CL_{\mu}$ defined
in \eqref{st} and generates a contraction semigroup
$T(t): L^2(\xO ; \ei^2)\to L^2(\xO ; \ei^2)$, $t\geq 0$, denoted also by
$e^{-\CL_{\mu}t}$. This semigroup is positivity preserving and by
\cite[Lemma 1.3.4]{D2} we can easily show that satisfies the conditions
of \cite[Theorems 1.3.2 and 1.3.3]{D2}. Using the logarithmic Sobolev
inequality (Theorem \ref{logth}) and some ideas of Davies
\cite{D1,D2}, we shall
show that $e^{-\CL_{\mu}t }$ is ultracontractive and therefore has
a kernel $k(t,x,y).$ More precisely, we prove
the following large time estimates:
\begin{theorem}
\label{weimaintheoremintro1}
Let $\xm\leq \frac{k^2}{4}$ and $T>0.$ Then there holds
\bal
k(t,x,y)\asymp1
\eal
for any $t\geq T$ and $x,y\in\xO$. The implicit constants
depend only on $\xO,K,\xm$ and $T.$
\end{theorem}

For small time the two-sided heat kernel estimate is different. A pivotal ingredient in the proof of this estimate is the boundary Harnack inequality. However, in order to state the boundary Harnack inequality, we first need to give the following definition of weak solution.

\begin{definition}
Let $D\subset\xO$ be an open set. We say that $v\in C^1((0,T):H^1(D  ; \ei^2  ))$ is a weak solution of $v_t+\CL_{\mu}v=0$ in $(0,T)\times D$ if for each $\xF
\in C^1_c((0,T):C_c^\infty(D)),$ we have
\bal
\int_{0}^{T}\int_{D}(v_t\xF+\nabla v\cdot
\nabla\xF ) \ei^2 \, dy \, dt=0.
\eal
\label{def1intro}
\end{definition}

\begin{theorem}[Boundary Harnack inequality]\label{Harnack}
Let $\mu\leq k^2/4$ and $v$ be a non-negative solution of $v_t+\CL_{\mu}v$ in $(0,r^2)\times \mathcal{B}(x,r)\cap\xO$. There exists $\beta_1>0$
 and a positive constant $C=C(\xO,K,\xb_1,\mu)$ such that
for all $r<\beta_1$ there holds
\be
\sup_{(\frac{r^2}{4},\frac{r^2}{2})\times\mathcal{B}(x,\frac{r}{2})\cap\xO}v\leq C
 \inf_{(\frac{3r^2}{4},r^2)\times\mathcal{B}(x,\frac{r}{2})\cap\xO} v.
\label{harnackpar}
\ee
\end{theorem}
\noindent Here $\mathcal{B}(x,r)$ are suitably defined ``balls"
(see Definition \ref{beta1}).
Let us briefly explain the proof of the above theorem. We first
prove the doubling property for the ``balls" $\mathcal{B}(x,r)$
(Lemma \ref{doubling}), the Poincar\'e inequality (Theorem \ref{thm:3.9}) and the Moser inequality (Theorem \ref{mosersx}). The last three results along with the density Theorem \ref{density} allow us to apply a Moser iteration argument similar to the one in \cite{GS,SC1} so that we reach the desired result. Due to the fact that $K\subset \partial\xO,$ the proof of the above theorem is more complicated than the one in \cite{FMT2,FMoT} and new essential difficulties arise which should be handled in a very delicate way.

Proceeding as in the proof of \cite[Theorem 5.4.12]{SC1}, we may deduce that the boundary Harnack inequality \eqref{harnackpar} implies the following sharp two-sided heat kernel estimate for small time.
\begin{theorem}
\label{weimaintheoremintro2}
Let $\xm\leq \frac{k^2}{4}$ and $T>0$. Then there exists $C>1$
which depends on $\Omega,K, \mu$ and $T$ such that
\begin{align*}
&C^{-1} \left((d(x)+\sqrt{t})(d(y)+\sqrt{t})\right)^{-1}\left((d_K(x)+\sqrt{t})(d_K(y)+\sqrt{t})\right)^{-\xg_+}
t^{-\frac{N}{2}}\exp\Big(-C\frac{|x-y|^2}{t}\Big) \\
&  \leq \,  k(t,x,y)  \\
& \leq C \!\left((d(x)+\sqrt{t})(d(y)+\sqrt{t})\right)^{-1}\!\!
\left((d_K(x)+\sqrt{t})(d_K(y)+\sqrt{t})\right)^{-\xg_+}\!\!
t^{-\frac{N}{2}}\exp\!\Big(\!-C^{-1}\frac{|x-y|^2}{t}\Big)
\end{align*}
for any $0<t\leq T$ and $x,y\in\xO$.
\end{theorem}

Let $h(t,x,y)$ denote the Dirichlet heat kernel of $L_{\mu}$.
It is then immediate that $h(t,x,y)=(\ei(x)\ei(y))e^{-\xl_\xm t}k(t,x,y)$.
Hence, by Theorems \ref{weimaintheoremintro1} and \ref{weimaintheoremintro2}, we obtain the following theorem.
\begin{theorem}\label{maintheorem1}
Let $\xm\leq \frac{k^2}{4}$. There exist $T=T(\xO,K,\xm)>0,$ $C_1=C_1(\xO,K,\xm,T,\xl_\xm)>1$ and $C_2=C(\xO,K,\xm,T)>1$ such that
 \bal
 {\rm (i)} & \\
&C^{-1}_1\Big(\frac{d(x)}{d(x)+\sqrt{t}}\Big)\Big(\frac{d(y)}{d(y)+\sqrt{t}}\Big)
\Big(\frac{d_K(x)}{d_K(x)+\sqrt{t}}\Big)^{\xg_+}\Big(\frac{d_K(y)}{d_K(y)+\sqrt{t}}\Big)^{\xg_+}
t^{-\frac{N}{2}}\exp\Big(-C_1\frac{|x-y|^2}{t}\Big)\\
&\leq h(t,x,y)\leq\\
& C_1\Big(\frac{d(x)}{d(x)+\sqrt{t}}\Big)\Big(\frac{d(y)}{d(y)+\sqrt{t}}\Big)
\Big(\frac{d_K(x)}{d_K(x)+\sqrt{t}}\Big)^{\xg_+}
\Big(\frac{d_K(y)}{d_K(y)+\sqrt{t}}\Big)^{\xg_+}
t^{-\frac{N}{2}}\exp\Big(-C^{-1}_1\frac{|x-y|^2}{t}\Big),
\eal
for any $ 0<t<T$ and $x,y\in\xO.$
\begin{align*}
\hspace{-3cm}{\rm (ii)}  \hspace{2cm} C^{-1}_2\xf_\xm(x)\xf_\xm(y)e^{-\xl_\xm t}\leq h(t,x,y)\leq C_2\xf_\xm(x)\xf_\xm(y)e^{-\xl_\xm t},
\end{align*}
for any $\ t>T$ and $x,y\in\xO.$
\end{theorem}
\noindent If $\xl_\xm>0,$ then by the above theorem we can obtain the existence of a minimal Green function
$G_{\xm}(x,y)$ of $L_{\mu}$ as well as precise asymptotic for $G_{\xm}(x,y)$ (see Subsection \ref{sub:greek} for more details).

\subsection{Martin Kernels and boundary value problems}

If $\xm<C_{\xO,K}$ then
the operator  $L_\xm=-\xD-\frac{\xm}{d_K^2}$ is coercive in $H_0^1(\xO)$.
Hence, taking into account the discussion on the first eigenfunction $\ei$
of \eqref{eigenvalue}, we may apply Ancona's results in \cite{An} to deduce
that any positive solution $u$ of $L_\mu u = 0$ in $\Omega$ can be
represented like \eqref{repr1}.
If $\xm=C_{\xO,K}<\frac{k^2}{4}$ then there exists an $H^1_0$
minimiser of the Hardy quotient and therefore there is no Green function
and the operator is not coercive.
In the remaining case $\xm=C_{\xO,K}=\frac{k^2}{4}$, the operator $L_\xm$
clearly is not coercive and this case is not
covered by Ancona's results in \cite{An}.
One of the main goals of this work is to prove that the
assumption $\xl_{\mu}>0$ suffices to have a respective representation formula, also in the case $\mu=\frac{k^2}{4}$.

In order to state the main results we first need to give some notations and definitions.
For $\xb>0$ we set
\[
K_\xb=\{x\in\mathbb{R}^N\setminus K:\;d_K(x)<\xb\} ,
\qquad
\Omega_\xb=\{x\in \Omega  :\;d(x)<\xb\} .
\]
We assume that $\xb$ is small enough so that for any $x\in \xO_{\xb}$
there exists a unique $\xi_x\in\partial\xO,$ which satisfies $d(x)=|x-\xi_x|.$
Now set
\be
\tilde{d}_{K}(x)=\sqrt{|\dist^{\partial\xO}(\xi_x,K)|^2+|x-\xi_x|^2} \, ,
\qquad x\in K_{\beta},
\label{dtilda}
\ee
where $\dist^{\partial\xO}(\xi_x,K)$ denotes the distance of $\xi_x$
to $K$ measured on $\partial\Omega$.

Let $\xb_0>0$ (this will be determined in Lemma \ref{subsup}).
We consider a smooth cut-off function $0\leq\eta_{\xb_0}\leq1$
with compact support in $K_{\frac{\xb_0}{2}}$
such that $\eta_{\xb_0}=1$ in $\overline{K}_{\frac{\xb_0}{4}}$.
We define
\[
W(x)=\left\{ \BAL &(d+\tilde d^2_K)\tilde d_{K}^\xgn,\qquad&&\text{if}\;
\mu <\frac{k^2}{4}, \\
& (d+\tilde d^2_K)\tilde{d}_{K}^{-\frac{k}{2}}(x)|\ln \tilde d_{K}(x)|,\qquad&&\text{if}\;\mu =\frac{k^2}{4},
\EAL \right. \quad x \in \Omega \cap K_{\xb_0} ,
\]
and
\[
\tilde W(x):=(1-\eta_{\beta_0}(x))+\eta_{\beta_0}(x)W(x), \quad x \in \Omega.
\]
Let  $h \in C(\partial \Omega)$ and
$u \in H^1_{loc}(\xO)\cap C(\xO)$. We write $\trti(u)=h$ whenever
\bel{bdrcond3}
\lim_{x\in\Omega,\;x\rightarrow y\in\partial\xO}
\frac{u(x)}{ \tilde{W}(x)}=h(y)\qquad\text{uniformly for } y\in\partial\xO.
\ee
In Section \ref{section5}
we prove that for any $h \in C(\partial \Omega)$ the problem
\bal
\left\{ \BAL
L_{\xm }v&=0 , \qquad\mathrm{in}\;\;\xO,\\
\trti(v)&=h , \qquad\mathrm{on}\;\;\partial\xO,
\EAL \right.
\eal
has a unique solution $v =v_h\in H^1_{loc}(\xO)\cap C(\xO)$.
From this and the accompanying estimate follows that for any $x_0\in\xO$
the mapping $h\mapsto v_h(x_0)$ is a linear positive functional on $C(\partial\xO)$. Thus there exists a unique Borel measure on $\partial\Omega$, called {\it $L_{\xm }$-harmonic measure} in $\xO,$ denoted by $\xo^{x_0}$, such that
\bal
v_{h}(x_0)=\int_{\partial\xO}h(y) d\xo^{x_0}(y).
\eal
Thanks to the Harnack inequality the measures $\xo^x$ and $\xo^{x_0},$ $x_0,\;x\in \xO$, are mutually absolutely continuous. Therefore, the Radon-Nikodyn derivative exists and we set
\bal
K_{\xm}(x,y):=\frac{dw^x}{dw^{x_0}}(y)\qquad\mathrm{for}\;\xo^{x_0}\text{- almost all }y\in\partial\xO.
\eal
\begin{definition}
Fix $\xi\in\partial\xO.$ A function $\mathcal{K}$ defined in $\xO$ is called a kernel function for $L_\mu$ with pole at $\xi$ and basis at $x_0\in\xO$ if
\begin{eqnarray*}
& {\rm (i)} & \mbox{$\mathcal{K}(\cdot,\xi)$ is $L_{\xm }$-harmonic in $\Omega$,} \\
& {\rm (ii)} & \mbox{
 $ \frac{\mathcal{K}(\cdot,\xi)}{\tilde{W}(\cdot)}\in C(\overline{\xO}\setminus\{\xi\})$ and for any $\eta \in\prt\Omega\setminus\{\xi\}$
 we have }
\lim_{x \in \Omega,\; x \to \eta}
\frac{\mathcal{K}(x,\xi)}{\tilde W(x)}=0,\\
& {\rm (iii)} & \mbox{$\mathcal{K}(x,\xi)>0$ for each $x\in\xO$ and $\mathcal{K}(x_0,\xi)=1.$}
\end{eqnarray*}
\end{definition}
\noindent
Using the ideas in \cite{caffa}, we show the existence and uniqueness of a kernel function with pole at $\xi$ and basis at $x_0$ (see Proposition \ref{uniq}). As a result we obtain the existence of the Martin kernel
and moreover
\bal
K_\xm(x,\xi)=\lim_{y\in\xO,\;y\to\xi}\frac{G_{\mu}(x,y)}{G_{\mu}(x_0,y)},\quad\forall \xi\in\partial\xO.
\eal
In addition, by the estimates on Green function $G_\xm(x,y)$ of $L_\xm$ (see Proposition \ref{green}) we obtain the following result.
\begin{theorem}
\label{poisson}
Assume that $\xm \leq \frac{k^2}{4}$ and $\lambda_{\mu}>0$.
We then have: \newline
$\ia$ If $\xm <\frac{k^2}{4}$ or $\xm=\frac{k^2}{4}$ and $k<N$ then
\begin{equation}
\label{Martinest1}
K_{\xm}(x,\xi) \asymp\frac{d(x)}{|x-\xi|^N}\left(\frac{d_K(x)}{\left(d_K(x)+|x-\xi|\right)^2}\right)^{\xg_+},\quad  \mbox{ in } \Omega\times\partial\xO
\end{equation}
$\ib$ If  $\xm=\frac{N^2}{4}$ (so $k=N$), then
\begin{equation}
\label{Martinest2}
K_{\mu}(x,\xi) \asymp \frac{d(x)}{|x-\xi|^N}\left(\frac{|x|}{\left(|x|+|x-\xi|\right)^2}\right)^{-\frac{N}{2}}
+\frac{d(x)}{|x|^{\frac{N}{2}}}\big|\ln |x-\xi| \big|,\quad
\mbox{ in } \Omega\times\partial\xO
\end{equation}
\end{theorem}

When $K=\partial\xO$,  Filippas, Moschini and Tertikas \cite{FMT2}
derived sharp two-sided estimate on the associated heat kernel. These
estimates where then used
in order to obtain sharp estimates on $G_\xm(x,y).$  Chen and
V\'{e}ron \cite{CV} studied the operator $L_\xm$ with $K=\{0\}\subset\partial
\Omega$
and they constructed the corresponding Martin kernel. The case $K\subset\xO$ was thoroughly studied by Gkikas and Nguyen in \cite{gktai}. Estimates on the Green kernel of $L_{\xm V}=-\xD-\xm V,$ where $V$ is a singular potential such that $|V(x)|\leq c d^{-2}(x)$ in $\xO$, have been given by Marcus \cite{M1,M2}. Marcus and Nguyen \cite{MT} used Ancona's result to show that the Martin kernel $K_\xm(x,y)$ is well defined and they applied the results in \cite{M2} to the model case $L_\xm$ in order to obtain estimates on the Green kernel $G_\xm(x,y)$ and the Martin kernel $K_\xm(x,y)$. However, their results do not cover the critical case $\xm=\frac{k^2}{4}.$

In this work, we follow a different approach which does not use Ancona's result \cite{An} and allows us to study the critical case. In particular our work is inspired by the articles \cite{FMT2,GkV,gktai}. The main difference here is that $K\subset\partial\xO,$ which has an effect on the value of the optimal Hardy constant $C_{\xO,K}$ as well as on the behaviour of the eigenfunction $\ei.$ As a result, this fact yields substantial difficulties and reveals new aspects of the study of $L_\xm.$

We are now ready to state the representation formula.

\begin{theorem}
\label{Lemm11intro}
Assume that $\xm \leq \frac{k^2}{4}$ and $\lambda_{\mu}>0$.
Let $u$ be a positive $L_{\xm }$-harmonic function
in $\Omega.$ Then $u\in L^1(\Omega;\ei)$ and there exists a unique Radon measure $\nu$ on $\partial\Omega$ such that
\bal
u(x)=\int_{\partial\Omega }K_{\xm }(x,\xi)d\nu(\xi) =: \BBK_\mu[\nu] .
\eal
\end{theorem}

In order to study the corresponding boundary value problem, we should first
introduce the notion of the boundary trace. We will define it in a dynamic
way. In this direction, let $\{\xO_n\}$ be
a smooth exhaustion of $\Omega$, that is
an increasing sequence of bounded open smooth domains such that
$\overline{\xO_n}\subset \xO_{n+1}$, $\cup_n\xO_n=\xO$ and $\mathcal{H}^{N-1}(\partial \Omega_n)\to \mathcal{H}^{N-1}(\partial \Omega)$.
The operator $L_{\xm }^{\xO_n}$ defined by
\[
L_{\xm }^{\xO_n}u=-\xD u-\frac{\xm }{d^2_K}u
\]
is uniformly elliptic and coercive in $H^1_0(\xO_n)$ and its first eigenvalue $\lambda_{\xm }^{\xO_n}$ is larger than $\lambda_{\xm }$.
For $h\in C(\prt \xO_n)$ the problem
\[
\left\{
\begin{array}{ll}
 L_{\xm }^{\xO_n}v=0 , &    \qquad  \text{in } \xO_n\\
 v=h , &   \qquad  \text{on } \prt \xO_n,
\end{array}
\right.
\]
admits a unique solution which allows to define the $L_{\xm }^{\xO_n}$-harmonic measure on $\prt \xO_n$
by
\[
v(x_0)=\myint{\prt \xO_n}{}h(y)d\omega^{x_0}_{\xO_n}(y).
\]

\begin{definition}[$L_{\mu}$-boundary trace]
	A function $u\in W^{1,p}_{loc}(\xO),\;p>1$, possesses an $L_{\mu}$-\emph{boundary trace}  if there exists a measure $\nu\in\GTM(\prt\Omega)$ such that for any smooth exhaustion
$\{ \xO_n \}$ of $\xO$, there  holds
\[
\lim_{n\rightarrow\infty}\int_{ \partial \xO_n}\phi u\, d\xo_{\xO_n}^{x_0}=\int_{\partial \xO} \phi \,d\xn , \qquad\forall \phi \in C(\overline{\xO}).
\]
\end{definition}
The $L_{\mu}$-boundary trace of $u$ will be denoted by $\tr(u)$.

Let $\GTM(\partial\xO)$ denote the space of bounded Borel measures on $\partial\xO$ and $\GTM(\xO;\ei)$ the space of Borel measures $\tau$ on $\xO$ such that
$$
\int_\xO\ei d|\tau|<\infty.
$$
Arguing as in \cite{MV} we obtain in Lemma \ref{tracemartin}
that  for any $\nu \in \GTM(\partial \Omega)$ we have
$\tr(\BBK_\mu[\nu])=\nu$.

Assume now that $\tau\in\mathfrak{M}(\xO;\ei)$  and let
\[	
u=\mathbb{G}_{\mu}[\tau]:=
\int_{\Omega }G_{\xm }(x,y)d\tau(y).
\]	
Then $u\in W^{1,p}_{loc}(\xO)$ for every $1<p<\frac{N}{N-1}$ and $\tr(u)=0$ (see Lemma \ref{tracegreen}).

Next, we give the definition of weak solutions of the following boundary value problem.
\begin{definition}
Let $\tau\in\mathfrak{M}(\xO;\ei)$ and $\nu \in \mathfrak{M}(\partial\xO)$. We say that
$u\in L^1(\Omega;\ei)$ is a weak solution of
\begin{equation}\label{NHLPintro} \left\{
\begin{array}{ll}
 L_\mu u=\tau ,  & \qquad \text{in }\;\Omega,\\
\tr(u)=\xn,  &
\end{array}
\right. \end{equation}
if
\[
	 \int_{\Omega} u \, L_{\xm }\zeta \, dx=\int_{\Omega} \zeta \, d\tau + \int_{\Omega} \mathbb{K}_{\xm}[\xn]L_{\xm }\zeta \, dx \; ,
	\qquad\forall \zeta \in\mathbf{X}_\xm(\xO ,K),
\]
where
\begin{equation}
{\bf X}_\mu(\Omega ,K)=\Big\{ \zeta \in H_{loc}^1(\Omega): \phi_\mu^{-1} \zeta \in H^1(\Omega;\phi_\mu^{2}), \, \phi_\mu^{-1}L_\mu \zeta \in L^\infty(\Omega)  \Big\}.
\label{xkom}
\end{equation}
\end{definition}
Let us state our main result for problem \eqref{NHLPintro}.

\begin{theorem}
\label{thm:val}
Let $\tau\in\mathfrak{M}(\xO;\ei)$ and $\xn \in \mathfrak{M}(\partial\xO)$. There exists a unique weak solution $u\in L^1(\Omega;\ei)$ of \eqref{NHLPintro},
\be \label{reprweaksol1}
u=\mathbb{G}_{\mu}[\tau]+\mathbb{K}_{\xm}[\xn]
\ee
Furthermore there exists a positive constant $C=C(\Omega,K,\mu)$ such that
\begin{align} \label{esti2}
\|u\|_{L^1(\Omega;\ei)} \leq \frac{1}{\lambda_\mu}
\| \tau \|_{\GTM(\xO;\ei)} + C \| \nu \|_{\GTM(\partial\Omega)} .
\end{align}
If in addition $d\tau=fdx+d\rho$ where  $f\in L^1(\Omega;\ei)$ and
$\rho \in\mathfrak{M}(\xO;\ei)$,
then for any $\zeta \in \mathbf{X}_\xm(\xO,K)$ with $\zeta \geq 0$, there hold
\be\label{poi4}
	\int_{\Omega}|u|L_{\mu }\zeta \, dx\leq \int_{\Omega}\sign(u)f\zeta\, dx +\int_{\xO}\zeta d|\rho| +
	\int_{\Omega}\mathbb{K}_{\xm}[|\xn|] L_{\xm }\zeta \, dx,
	\ee
	\be\label{poi5}
	\int_{\Omega}u_+L_{\mu }\zeta \, dx\leq \int_{\Omega} \sign_+(u)f\zeta\, dx +\int_{\xO}\zeta\,d\rho_+ +
	\int_{\Omega}\mathbb{K}_{\xm}[\nu_+]L_{\xm }\zeta \,dx .
	\ee
\end{theorem}

It is worth mentioning here that Marcus and Nguyen \cite{MT} studied
problem \eqref{NHLPintro} by introducing an alternative normalized
boundary trace $\text{tr}^*(u)$
(see \cite[Definition 1.2]{MT}). However this normalized boundary trace is
well defined only  if $\xm<\min(C_{\xO,K},\frac{2k-1}{4}).$ As a consequence they
showed that the boundary value problem
\[
\left\{
\begin{array}{ll}
 L_\mu u =\tau , & \qquad \text{in }\;\Omega , \\[0.1cm]
\text{tr}^*(u)=\xn , &
\end{array}
\right.
\]
admits a unique solution provided $\xm<\min(C_{\xO,K},\frac{2k-1}{4}).$

Finally, we study a semilinear problem involving the operator $L_\xm$ and a subcritical absorption (for relevant works see \cite{CV,CheVer,DD1,DN,GkV,MT,MVbook,MarNgu,MarMor,GkT1,GkT2}). In particular we are concerned with the  problem
\begin{equation}
\label{NLinintro} \left\{
\begin{array}{ll}
 L_\mu u+g(u)=0 , & \qquad \text{in }\;\Omega,\\[0.1cm]
\tr(u)=\nu, &
\end{array}
\right.
\end{equation}
where $g: \mathbb{R}\rightarrow\mathbb{R}$ is a nondecreasing continuous function such that $g(0)=0$. The above problem was treated by
Marcus and Nguyen \cite{MT}, where they consider the normalized
boundary trace $\text{tr}^*(u)$ instead of $\tr(u).$ In the case where the
boundary trace is defined in a dynamic way, we have the following result.

\begin{theorem} \label{exist-subGKintro}
Let $\mu\leq k^2/4$ and let
$g: \mathbb{R}\rightarrow\mathbb{R}$ be a nondecreasing continuous function such that $g(0)=0$.  Assume that for some $p>1$ there holds
\begin{equation}
\label{subcd0intro1}
\int_1^\infty t^{-1-p}(g(t)-g(-t))dt < +\infty.
\end{equation}
Let $\nu \in \GTM(\partial \Omega)$. Then

(a) If  (\ref{subcd0intro1}) holds true with
$p=\min\left(\frac{N+1}{N-1},\frac{N+\xgp+1}{N+\xgp-1}\right)$ then  there exists a unique weak solution $u$ of \eqref{NLinintro}.

(b) Assume that either $k<N$ or $k=N$ and $\mu<N^2/4$. If $\nu$ has support in $K$ and  (\ref{subcd0intro1}) holds true with
$p=\frac{N+\xgp+1}{N+\xgp-1}$ then there exists a unique weak solution $u$ of \eqref{NLinintro}.

(c)  If $\nu$ has compact support in $\partial\Omega\setminus K$ and  (\ref{subcd0intro1}) holds true with $p=\frac{N+1}{N-1}$ then there exists a unique weak solution $u$ of \eqref{NLinintro}.

Moreover in all three cases the weak solution $u$ satisfies
\[
u+\mathbb{G}_{\mu}[g(u)]=\mathbb{K}_{\xm}[\xn],\quad a.e. \;\text{in}\;\xO .
\]
\end{theorem}

When $k=N$ and $\mu=N^2/4$, the subcritical condition \eqref{subcd0intro1} must be modified by a  logarithmic factor
in order to obtain the desired result, that is
\begin{theorem} \label{exist-subGK4}
Let $K=\{0\}$ and $\mu=\frac{N^2}{4}$.
Assume that $g$ satisfies
\bal
\int_1^\infty  s^{-\frac{N+2}{N-2}-1}(\ln s)^{\frac{N+2}{N-2}} \big(g(s)-g(-s)\big)ds<\infty
\eal
and let $\alpha> 0$. Then there exists a unique solution $u$ of \eqref{NLinintro} with $\xn=\alpha\xd_0$. Moreover the function $u$ satisfies
\bal
u+\mathbb{G}_{\mu}[g(u)]=
\alpha\mathbb{K}_{\xm}[\xd_0],\quad a.e. \;\text{in}\;\xO.
\eal
\end{theorem}

\section{Hardy-Sobolev type inequalities}

In this section we shall prove various Hardy-Sobolev type inequalities
that will be essential for our analysis.
We start by recalling the following result:

\begin{proposition}\cite[Lemma 2.1]{FF}
There exists $\xb_0=\xb_0(K,\xO)$ small enough such that, for any $x\in \xO\cap K_{\xb_0},$ the following estimates hold:
\begin{eqnarray*}
\mbox{a)}  && \tilde{d}_{K}^2(x)=d^2_{K}(x)(1+g(x)) \\
\mbox{b)}  && \nabla d(x) \cdot \nabla \tilde{d}_{K}(x)=\frac{d(x)}{\tilde{d}_{K}(x)} \\
\mbox{c)}  &&|\nabla\tilde{d}_{K}(x)|^2=1+h(x)  \\[0.1cm]
\mbox{d)}  && \tilde{d}_{K}(x)\xD \tilde{d}_{K}(x)=k-1+f(x) ,
\end{eqnarray*}
where the functions $g,\;h$ and $f$ satisfy
\be
|g(x)|+|h(x)|+|f(x)|\leq C_1(\xb_0,N) \tilde{d}_{K}(x),\quad \forall
x\in\xO\cap K_{\xb_0}. \label{frag}
\ee

\label{propdK}
\end{proposition}

\begin{lemma}
\label{lem1}
Assume that $\xa \neq0$ and $\xg +\xa+k-1\neq 0$.
There exists $\xb_0>0$ and $C=C(\xg,\xa,k,\xb_0,N)$
such that for any open $V\subset K_{\xb_0}\cap \xO$
and for any $u\in C^\infty_c(V)$ there holds
$$
\int_Vd^\xa\tilde{d}_K^{\xg-1} |u|dx+\int_Vd^{\xa-1}\tilde{d}_K^{\xg} |u|dx\leq C \int_Vd^\xa\tilde{d}_K^{\xg} |\nabla u|dx.
$$
\end{lemma}
\begin{proof}
By Proposition \ref{propdK} we have
\begin{align*}
\xg\int_V&d^\xa\tilde{d}_K^{\xg-1} |u|dx+\xg\int_Vd^\xa\tilde{d}_K^{\xg-1} h|u|dx=\int_Vd^\xa\nabla\tilde{d}_K^{\xg}\cdot \nabla\tilde{d}_K |u|dx\\
&=-\xa \int_Vd^{\xa-1} \tilde{d}_K^{\xg}
\nabla d\cdot \nabla\tilde{d}_K |u|dx -\int_Vd^{\xa}\tilde{d}_K^\xg\xD\tilde{d}_K  |u|dx-\int_Vd^\xa\tilde{d}_K^{\xg}\nabla\tilde{d}_K
\cdot \nabla|u|dx\\
&=- \xa\int_Vd^{\xa}\tilde{d}_K^{\xg-1} |u|dx-\int_Vd^{\xa}\tilde{d}_K^{\xg-1}(k-1+f) |u|dx-\int_Vd^\xa\tilde{d}_K^{\xg}\nabla\tilde{d}_K \cdot \nabla|u|dx \, ,
\end{align*}
that is
\[
(\xg +\xa+k-1)\int_Vd^\xa\tilde{d}_K^{\xg-1} |u|dx
= -\int_Vd^{\xa}\tilde{d}_K^{\xg-1}(f+\xg h) |u|dx-\int_Vd^\xa\tilde{d}_K^{\xg}\nabla\tilde{d}_K \cdot \nabla|u|dx \, .
\]
By the above equality, Proposition \ref{propdK} and \eqref{frag}, we can easily prove that
\begin{align*}
\big(|\xg +\xa+k-1|-C(C_1,\xg)\xb_0 \big)
\int_Vd^\xa\tilde{d}_K^{\xg-1} |u|dx\leq (1+C_1\sqrt{\xb_0})\int_Vd^\xa\tilde{d}_K^{\xg} |\nabla u|dx,
\end{align*}
where $C_1=C_1(\xb_0,N)$ is the constant in inequality \eqref{frag}. Choosing $\xb_0$ small enough, we obtain
\be
\label{anis1}
\int_Vd^\xa\tilde{d}_K^{\xg-1} |u|dx\leq C \int_Vd^\xa\tilde{d}_K^{\xg}|\nabla u|dx .
\ee
By \eqref{anis1} and Proposition \ref{propdK} we have
\begin{align*}
\left|\xa\int_Vd^{\xa-1}\tilde{d}_K^{\xg} |u|dx\right|&=\left|\int_V(\nabla d^{\xa}
\cdot \nabla d) \tilde{d}_K^{\xg} |u|dx\right|\\
&\leq C  \int_V d^{\xa}\tilde{d}_K^{\xg -1} |u|dx+
\int_Vd^\xa\tilde{d}_K^{\xg}|\nabla u|dx,
\end{align*}
provided $\xb_0$ is small enough. The result now follows.
\end{proof}

\begin{lemma}
Assume that $a \neq 0$ and $c +a+k-1\neq 0$.
Let $1\leq q \leq \frac{N}{N-1}$ and $b=a-1+N \frac{q-1}{q}$.
If $\beta_0$ is small enough then there exists $C=C(a,c,k,\xb_0,q,N)$
such that for any open $V\subset \xO\cap K_{\xb_0}$
and for any $u\in C^\infty_c(V)$ the following inequality is valid
\be
\Big(\int_Vd^{qb}\tilde d_K^{qc}|u|^qdx\Big)^\frac{1}{q}\leq
C \int_Vd^a\tilde{d}_K^{c} |\nabla u|dx .
\label{eq:new}
\ee
\end{lemma}
\begin{proof}
Let $0\leq\theta_i\leq 1, i=1,2$, be such that $\theta_1+\theta_2=1$ and $\frac{N-1}{N}\theta_1+\theta_2=\frac{1}{q}$.
By H\"{o}lder inequality we have
\begin{align*}
\int_Vd^{qb}\tilde d_K^{qc}|u|^qdx
&=\int_V\Big(d^{qa \theta_1}\tilde d_K^{qc\theta_1}
|u|^{\theta_1q}\Big)
\left(d^{q(a-1)\theta_2}\tilde d_K^{qc\theta_2}|u|^{\theta_2q}\right)dx\\
&\leq \|d^a \tilde d^c_Ku\|_{L^{\frac{N}{N-1}}(V)}^{\theta_1q}
\|d^{a-1} \tilde d^c_Ku\|_{L^{1}(V)}^{\theta_2q},
\end{align*}
and therefore
\be
\|d^b \tilde d_K^c u\|_{L^q(V)}\leq
\|d^a \tilde d^c_Ku\|_{L^{\frac{N}{N-1}}(V)}+
\| d^{a-1} \tilde d^c_Ku\|_{L^{1}(V)}.
\label{holder}
\ee
By the $L^1$ Sobolev inequality and Lemma \ref{lem1} we have
\begin{align*}
\|d^a \tilde d^c_Ku\|_{L^{\frac{N}{N-1}}(V)}&\leq C
\left(|c|\int_Vd^a\tilde{d}_K^{c-1} |u|dx+|a|\int_Vd^{a-1}\tilde{d}_K^{c} |u|dx+\int_Vd^a\tilde{d}_K^{c} |\nabla u|dx\right)\\
&\leq C \int_Vd^a\tilde{d}_K^{c} |\nabla u|dx .
\end{align*}
Combining this with Lemma \ref{lem1} and  \eqref{holder}
concludes the proof. $\hfill\Box$

\begin{lemma}
\label{lem3}
Assume that $a\neq0$ and $c +a+k-1\neq0$.
Let $2<Q\leq\frac{2N}{N-2}$ and $b=a-1+N\frac{Q-2}{2Q}$.
If $\beta_0$ is small enough then there exists $C=C(c,a,k,\xb_0,Q,N)$
such that for any open $V\subset \xO\cap K_{\xb_0}$
and for any $v\in C^\infty_c(V)$ there holds
\[
\bigg(\int_V (d^b\tilde d_K^c)^{\frac{2Q}{Q+2}}|v|^Q dx
\bigg)^\frac{2}{Q}\leq  C \int_Vd^{2a-\frac{2Qb}{Q+2}}
\tilde d_K^{\frac{4c}{Q+2}} |\nabla v|^2dx .
\]
\end{lemma}
\begin{proof}
Let $s=\frac{Q}{2}+1$ and write $Q=qs$. Applying \eqref{eq:new}
to the function $u=|v|^s$ we obtain
\begin{align}
\left(\int_V\big(d^{b}\tilde d_K^{c}\big)^{\frac{2Q}{Q+2}}|v|^Qdx\right)^\frac{Q+2}{2Q}&\leq C \int_Vd^a\tilde{d}_K^{c} |v|^{\frac{Q}{2}}|\nabla v|  dx.
\label{5}
\end{align}
Now, by Schwarz inequality, we have
\begin{align*}
\int_Vd^a\tilde{d}_K^{c} |v|^{\frac{Q}{2}}|\nabla v| dx&=\int_V
d^{b{\frac{Q}{Q+2}}}\tilde{d}_K^{c{\frac{Q}{Q+2}}}|v|^{\frac{Q}{2}} \,
d^{a-b{\frac{Q}{Q+2}}}
\tilde{d}_K^{c(1-{\frac{Q}{Q+2}})} |\nabla v|dx\\
&\leq \left(\int_V\big(d^{b}\tilde d_K^{c}\big)^{\frac{2Q}{Q+2}}|v|^Qdx\right)^\frac{1}{2}\left(\int_Vd^{2a-\frac{2Qb}{Q+2}}\tilde d_K^{c(2-\frac{2Q}{Q+2})} |\nabla v|^2dx\right)^\frac{1}{2}.
\end{align*}
The result follows by \eqref{5} and the last inequality.
\end{proof}

\begin{corollary}
Let  $\xa \neq0$ and assume that $(\xa +\xg)
\frac{N-1}{N-2} +k-1\neq0$.
There exists $\xb_0$ small enough and $C>0$ such that for any open
$V\subset \xO\cap K_{\xb_0}$ and for all
$u\in C^\infty_c(V)$ there holds
\[
\left(\int_V \big(d^{\frac{\xa}{2}} \tilde d_K^{\frac{\xg}{2}}|u|\big)^{\frac{2N}{N-2}}dx\right)^\frac{N-2}{N}\leq C\int_Vd^{\xa}\tilde d_K^{\xg} |\nabla u|^2dx .
\]
\end{corollary}
\begin{proof}
We apply  Lemma \ref{lem3} with
$Q=\frac{2N}{N-2},$ $a=\xa\frac{N-1}{N-2},$ $c=\xg(\frac{N-1}{N-2})$.
\end{proof}

\begin{corollary}
Let  $\xa>0$ and $\xg \geq 0$. There exists
$\xb_0>0$ and $C>0$ such that for any open
$V\subset \xO\cap K_{\xb_0}$
and all $u\in C^\infty_c(V),$ the following inequality is valid
\[
\left(\int_V d^{\xa} \tilde d_K^{\xg}|u|^{\frac{2(N+\xa+\xg)}{N+\xa+\xg-2}}dx\right)^\frac{N+\xa+\xg-2}{N+\xa+\xg}\leq C\int_Vd^{\xa+\frac{2\xg}{N+a+\xg}}\tilde d_K^{\xg-\frac{2\xg}{N+a+\xg}} |\nabla u|^2dx.
\]
\label{hardysobolev2}
\end{corollary}

\begin{proof}
This follows by Lemma \ref{lem3} with
$Q=\frac{2(N+\xa+\xg)}{N+\xa+\xg-2},$ $c=\frac{\xg}{q}$,
$b=\frac{\xa}{q},$ where $q=\frac{2Q}{Q+2}$.
\end{proof}

\begin{corollary}\label{hardysobolev}
Let  $\xa>0$, $\xg <0$ and assume that
$\xa+\xg\frac{ N+\alpha -1}{N+\alpha}+ k-1\neq 0$.
There exists $\xb_0>0$ and $C>0$ such that for any open
$V\subset \xO\cap K_{\xb_0}$ and all $u\in C^\infty_c(V)$
there holds
\[
\left(\int_V d^{\xa} \tilde d_K^{\xg}|u|^{\frac{2(N+\xa)}{N+\xa-2}}dx\right)^\frac{N+\xa-2}{N+\xa}\leq C\int_Vd^{\xa}\tilde d_K^{\xg\frac{N+\xa-2}{N+\xa}} |\nabla u|^2dx \, .
\]
\end{corollary}
\begin{proof}
The proof follows from Lemma \ref{lem3}, with $Q=\frac{2(N+\xa)}{N+\xa-2},$
$c=\xg\frac{ N+\alpha -1}{N+\alpha} $ and $b=\xa\frac{ N+\alpha -1}{N+\alpha}$.
\end{proof}

\section{Heat Kernel Estimates for small time}

We are now going to introduce some notation and tools that will be useful for our local analysis near $K$ and $\partial\xO$; see e.g. \cite{Kuf}.

Let $x =(x',x'')\in \R^N$, $x'=(x_1,..,x_k) \in \R^k$,  $x''=(x_{k+1},...,x_N) \in \R^{N-k}$. For $\beta>0$, we denote by $B_{\beta}^k(x')$ the ball  in $\R^k$ with center $x'$ and radius $\beta$.
For any $\xi\in K$ we also set
\[
V_K(\xi,\xb)=
\Big\{x=(x',x''): |x''-\xi''|<\beta,\; |x_i-\Gamma_{i,K}^\xi(x'')|<\xb,\;\forall i=1,...,k\Big\},
\]
for some functions $\Gamma_{i,K}^\xi: \R^{N-k} \to \R$, $i=1,...,k$.

Since $K$ is a $C^2$ compact submanifold in $\mathbb{R}^N$ without boundary,
there exists $\xb_0>0$ such that

\begin{itemize}
\item For any $x\in K_{6\beta_0}$, there is a unique $\xi \in K$  satisfying
$|x-\xi|=d_K(x)$.

\item $d_K \in C^2(K_{4\beta_0})$, $|\nabla d_K|=1$ in $K_{4\beta_0}$ and there exists $g\in L^\infty(K_{4\beta_0})$ such that
\[
\Delta d_K(x)=\frac{k-1}{d_K(x)}+g(x) , \qquad \text{in } K_{4\beta_0} .
\]
(See \cite[Lemma 2.2]{Vbook} and \cite[Lemma 6.2]{DN}.)

\item For any $\xi \in K$, there exist $C^2$ functions $\Gamma_{i,K}^\xi \in C^2(\R^{N-k};\R)$, $i=1,...,k$, such that defining
\[
V_K(\xi,\xb):=
\Big\{x=(x',x''): |x''-\xi''|<\beta,\; |x_i-\Gamma_{i,K}^\xi(x'')|<\xb,\;
 i=1,...,k\Big\},
\]
we have (upon relabelling and reorienting the coordinate axes if necessary)
\[
V_K(\xi,\beta) \cap K=
\Big\{x=(x',x''): |x''-\xi''|<\beta,\;
x_i=\Gamma_{i,K}^\xi (x''), \;  i=1,...,k\Big\}.
\]

\item There exist $\xi^{j}$, $j=1,...,m_0$, ($ m_0 \in \N$) and $\beta_1 \in (0, \beta_0)$ such that
\be \label{cover}
K_{2\xb_1}\subset \bigcup_{i=1}^{m_0} V_K(\xi^i,\beta_0).
\ee
\end{itemize}

Now set
\[
\xd_K^\xi(x):=\Big(\sum_{i=1}^k|x_i-\Gamma_{i,K}^\xi(x'')|^2\Big)^{\frac{1}{2}}, \qquad x=(x',x'')\in V_K(\xi,4\beta_0).
\]
Then there exists a constant $C=C(N,K)$ such that
\be\label{propdist}
d_K(x)\leq	\xd_K^{\xi}(x)\leq C \| K \|_{C^2} d_K(x),\quad \forall x\in V_K(\xi,2\beta_0),
\ee
where $\xi^j=((\xi^j)', (\xi^j)'') \in K$, $j=1,...,m_0$, are the points in \eqref{cover} and
\[
\| K \|_{C^2}:=\sup\{  \| \Gamma_{i,K}^{\xi^j} \|_{C^2(B_{5\beta_0}^{N-k}((\xi^j)''))}: \; i=1,...,k, \;j=1,...,m_0 \} < \infty.
\]
For simplicity we shall write $\xd_K$ instead of $\xd_K^{\xi}$.
Moreover, $\beta_1$ can be chosen small enough so that for any $x \in K_{\beta_1}$,
\[
B(x,\beta_1) \subset V_K(\xi,\beta_0),
\]
where $\xi \in K$ satisfies $|x-\xi|=d_K(x)$.

When $K=\partial\xO$ we assume that

\[
V_{\partial\xO}(\xi,\beta) \cap \xO=\Big\{x: \sum_{i=2}^N|x_i-\xi_i|^2<\beta^2,\;
0< x_1 -\Gamma_{1,\partial\xO}^\xi (x_2,...,x_N) <\beta \Big\}.
\]
Thus, when $x\in K\subset\partial\xO$ is a $C^2$ compact submanifold in $\mathbb{R}^N$ without boundary, of co-dimension $k$, $1< k \leq N,$ we have that

\be
\xG_{1,K}^\xi(x'')=\xG_{1,\partial\xO}^\xi(\xG_{2,K}^\xi(x''),...,\xG_{k,K}^\xi(x''),x'').\label{dist3}
\ee

Let $\xi\in K$. For any $x\in V_K(\xi,\xb_0)\cap \xO,$ we define
\[
\xd(x)=x_1-\Gamma_{1,\partial\xO}^\xi (x_2,...,x_N),
\]
and
$$\xd_{2,K}(x)=\Big(\sum_{i=2}^k|x_i-\Gamma_{i,K}^\xi(x'')|^2 \Big)^{\frac{1}{2}} .
$$
Then by \eqref{dist3}, there exists a constant $A>1$ which depends only on
$\xO $, $K$ and $\beta_0$ such that

\be
\frac{1}{A}(\xd_{2,K}(x)+\xd(x))\leq\xd_K(x)\leq A(\xd_{2,K}(x)+\xd(x)),\label{propdist2}
\ee
Thus by \eqref{propdist} and \eqref{propdist2} there exists a constant $C=C(\xO,K,\xg)>1$ which depends on $k, N,\Gamma_{i,K}^\xi,\Gamma_{1,\partial\xO}^\xi,\xg $ such that

\be
C^{-1}\xd^2(x)(\xd_{2,K}(x)+\xd(x))^\xg\leq d^2(x)d_K^\xg(x)\leq C\xd^2(x)(\xd_{2,K}(x)+\xd(x))^\xg.\label{ddK}
\ee
We set
\begin{align} \nonumber
\mathcal{V}_K(\xi,\xb_0)=\left\{(x',x''): |x''-\xi''|<\beta_0,\; |\xd(x)|<\xb_0,\;
|\xd_{2,K}(x)|<\xb_0\right\}.
\end{align}
We may then assume that
\begin{align*}
\mathcal{V}_K(\xi,\xb_0)\cap\xO=\left\{(x',x''): |x''-\xi''|<\beta_0, 0<\xd(x)<\xb_0,\; |\xd_{2,K}(x)|<\xb_0\right\},
\end{align*}

\begin{align*}
\mathcal{V}_K(\xi,\xb_0)\cap\partial\xO=\left\{(x',x''): |x''-\xi''|<\beta_0, \xd(x)=0,\; |\xd_{2,K}(x)|<\xb_0\right\},
\end{align*}
and

\begin{align*}
\mathcal{V}_K(\xi,\xb_0)\cap K=\left\{(x',x''): |x''-\xi''|<\beta_0, \xd(x)=0,\; \xd_{2,K}=0\right\}.
\end{align*}

Let $\xb_1>0$,  $1<\xg<2,$ and $0<r<\xb_1.$ For any $x\in
V_{\partial\Omega}(\xi,\frac{\xb_0}{16})$ with $d(x)\leq \xg r,$.
Taking $\beta_1$ small enough we have
$$
\mathcal{D}(x,r):=\Big\{y: \sum_{i=2}^N|y_i-x_i|^2<r^2,\; |\xd(y)|<r+d(x)
\Big\}\subset\subset V_{\partial\xO}(\xi,\frac{\xb_0}{16}).
$$
and there exist $C_{\xi}=C(\xG^\xi,\xO)>1,$ such that
\be
\mathcal{D}(x,r)\subset B(x, C_\xi r).\label{conddk}
\ee
Also,
$$
\mathcal{D}(x,r)\cap\xO=\{y: \sum_{i=2}^N|y_i-x_i|^2<r^2,\; 0<\xd(y)<r+d(x)\}.
$$

\begin{definition}\label{beta1}
Let $\xb_1>0$ be small enough, $r\in(0,\xb_1),$ $b\in (1,2),$ $\xi\in K$ and $x\in V(\xi,\frac{\xb_0}{16}).$ We define
\begin{eqnarray*}
&\mbox{ {\rm (i)}} & \mathcal{B}(x,r)= B(x,r), \mbox{ if } d(x)>b r \\
&\mbox{ {\rm (ii)}} & \mathcal{B}(x,r)= \mathcal{D}(x,r), \mbox{ if $d(x)\leq b r$ and } d_K(x)>b C_\xi r \\
&\mbox{ {\rm (iii)}} &  \mathcal{B}(x,r)=\{y=(y',y''): |y''-x''|< r,\; |\xd_{2,K}(y)|<r+d_K(x),\;|\xd(y)|<r+d(x)\}, \\
&& \mbox{if $d(x)\leq b r$ and $d_K(x)\leq b C_\xi r$.}
\end{eqnarray*}

Finally, we set
$$\overline{\mathcal{M}}(x,r)=\int_{\mathcal{B}(x,r)\cap\xO}d^2(y)d^\xg_K(y)dy.$$
\end{definition}

\subsection{Doubling Property}

\begin{lemma}
Let $\xg\geq -k$. Let $\xi\in\partial\Omega$
and $x\in V(\xi,\frac{\xb_0}{16}).$ Then, there exist $\xb_1>0$ and $C=C(\xO,K,\gamma,\xb_0)>1$ such that
\begin{equation}
\frac{1}{C}(r+d(x))^2(r+d_K(x))^\xg r^N\leq\overline{\mathcal{ M}}(x,r)\leq C(r+d(x))^2(r+d_K(x))^\xg r^N,
\label{doublingsxes}
\end{equation}
for any $0<r<\xb_1.$
\label{doubling}
\end{lemma}
\begin{proof}
We will consider three cases.

\minsk

\noindent
\textbf{Case 1. $d(x)>b r$}  Since $d_K(x)\geq d(x),$ we can easily show that for any $y\in B(x,r)$ we have
$\frac{b-1}{b}d(x)\leq d(y)\leq \frac{b+1}{b}d(x)$ and
$\frac{b-1}{b}d_K(x)\leq d_K(y)\leq \frac{b+1}{b}d_K(x)$.
 Thus the proof of \eqref{doublingsxes} follows easily in this case.

\minsk

\noindent
\textbf{Case 2. $d(x)\leq b r$ and $d_K(x)>b C_\xi r.$} By \eqref{conddk}, we again have that $\frac{b-1}{b} d_K(x)\leq d_K(y)\leq \frac{b+1}{b} d_K(x).$ Using the last inequality and proceeding as the proof of \cite[Lemma 2.2]{FMT2}, we obtain the desired result.

\minsk

\noindent
\textbf{Case 3. $d(x)\leq b r$ and $d_K(x)\leq b C_\xi r.$}

Let $\overline{y}=(y_2,...,y_k)\in \mathbb{R}^{k-1}.$ By \eqref{ddK} and the definition of $\mathcal{B}(x,r)$ , we have
\begin{align}\nonumber
\overline{\mathcal{M}}(x,r)&=\int_{\mathcal{B}(x,r)\cap\xO}d^2(y)d^\xg_K(y)dy\leq \int_{\mathcal{B}(x,r)\cap\xO}C\xd^2(y)(\xd_{2,K}(y)+\xd(y))^\xg dy\\ \nonumber
&\leq C\int_{B^{N-k}(x'',r)}\int_{0}^{d(x)+r}\int_{|\overline{y}|<d_K(x)+r}(|\overline{y}|+y_1)^\xg y_1^2d\overline{y} \, dy_1 \, dy''\\
&=C C(k,N) r^{N-k}\int_{0}^{d(x)+r}\int_{0}^{d_K(x)+r}s^{k-2}(s+y_1)^\xg y_1^2ds \,
dy_1.\label{1}
\end{align}

Now, if $\xg>0$ then
\begin{align*}
\int_{0}^{d(x)+r}&\int_{0}^{d_K(x)+r}s^{k-2}(s+y_1)^\xg y_1^2ds\, dy_1\\
&\leq \frac{1}{k-1}(2r +d(x)+d_K(x))^\xg (d_K(x)+r)^{k-1}(d(x)+r)^3\\
&\leq \frac{(b+2)^\xg (bC_\xi+1)^{k-1}(b+1)}{k-1}(r +d_K(x))^\xg (d(x)+r)^2r^k.
\end{align*}

If $-k\leq\xg\leq0,$ then
\begin{align*}
\int_{0}^{d(x)+r}&\int_{0}^{d_K(x)+r}s^{k-2}(s+y_1)^\xg y_1^2ds\, dy_1\\
&\leq \int_{0}^{d(x)+r}\int_{0}^{d_K(x)+r}s^{k-2}(s+y_1)^{\xg+2}ds\, dy_1\\
&\leq \int_{0}^{d(x)+r}\int_{0}^{d_K(x)+r} (s+y_1)^{\xg+k}ds\, dy_1\\
&\leq (d_K(x)+r)(d(x)+r) (2r +d(x)+d_K(x))^{\xg+k} \\
&\leq(2C_\xi+2) (d(x)+r)^2 (d_K(x)+r)^\xg (2r +d(x)+d_K(x))^{k}\\
&\leq (2C_\xi+2)(2+b+bC_\xi)^{k} (d(x)+r)^2 (d_K(x)+r)^\xg r^k.
\end{align*}

Similarly, for the opposite inequality, we have

\begin{align}\nonumber
\int_{0}^{d(x)+r}&\int_{0}^{d_K(x)+r}s^{k-2}(s+y_1)^\xg y_1^2ds\, dy_1\\ \nonumber
&\geq \int_{\frac{d(x)+r}{2}}^{d(x)+r}\int_{\frac{d_K(x)+r}{2}}^{d_K(x)+r}s^{k-2}(s+y_1)^\xg y_1^2dsdy_1\\
&\geq C(b,C_\xi,k,\xg)(d(x)+r)^2 (d_K(x)+r)^\xg r^k.\label{4}
\end{align}
The desired result follows by \eqref{1}-\eqref{4}.
\end{proof}

From \eqref{eigenest} and Lemma \ref{doubling}, we have the following corollary.

\begin{corollary}
Let  $x\in V(\xi,\frac{\xb_0}{16})$ and
$$
\mathcal{M}(x,r)=\int_{\mathcal{B}(x,r)\cap\xO}\ei^2(y)dy.
$$
Then, there exist $\xb_1>0$ and $C=C(\xO,K,\xb_0)>1$ such that
\[
\frac{1}{C}(r+d(x))^2(r+d_K(x))^{2\xg_+} r^N\leq\mathcal{ M}(x,r)\leq C(r+d(x))^2(r+d_K(x))^{2\xg_+} r^N,
\]
for any $0<r<\xb_1.$\label{doublingcoro}
\end{corollary}

\subsection{Density of $C^\infty_c(\xO)$ functions}

\begin{lemma}\label{densityflat}
Let $k\leq N$, $\xg\geq -k,$ $x=(x_1,x_2,...,x_k,x_{k+1},...,x_N)=(x_1,\overline{x},x'')$.
Let
$$
O=(0,1)\times B^{\mathbb{R}^{k-1}}(0,1)\times B^{\mathbb{R}^{N-k}}(0,1)
$$
and $u\in H^1(O ; x_1^2(x_1+|\overline{x}|)^\xg )$.
Assume that there exists $0<\xe_0<1$ such that  $u(x)=0$ if either
$x_1>\xe_0$ or $|\overline{x}|^2+|x''|^2>\xe_0^2$. Then there exists a sequence $\{u_n\}_{n=1}^\infty\subset C^\infty_c(O)$ such that
\[
u_n\rightarrow u,\quad\text{in}\;H^1(O ; x_1^2(x_1+|\overline{x}|)^\xg )
\]
\end{lemma}
\begin{proof}
Let $m\in\mathbb{N}$.
Set $$v_m(x)=\left\{ \BAL &m, \qquad&& \text{if }u(x)>m, \\
&u(x), &&\text{if } -m\leq u(x)\leq m,\\
&-m&&\text{if }  u(x)<-m.
\EAL \right.$$
Then we can easily prove that $v_m\rightarrow u$ in
$H^1(O; x_1^2(x_1+|\overline{x}|)^\xg )$.

Let $\xe>0$. There exists $m_0\in \mathbb{N},$ such that
\be
\norm{v_{m_0}-u}_{H^1(O;x_1^2(x_1+|\overline{x}|)^\xg )}=\left(\int_{O}x_1^2(x_1+|\overline{x}|)^\xg(|v_{m_0}-u|^2+|\nabla v_{m_0}-\nabla u|^2)dx\right)^{\frac{1}{2}}<\frac{\xe}{3}.\label{densityeq1}
\ee
For any $0<h<1,$ we consider the function
\[
\eta_h(x_1)=\left\{ \BAL &1 \qquad&& \text{if }x_1>h, \\
&1-(\ln h)^{-1}\ln\left(\frac{x_1}{h}\right) &&\text{if } h^2\leq x_1\leq h,\\
&0&&\text{if }  x_1<h^2,
\EAL \right.
\]
We will show that $z_{h}:=\eta_hv_{m_0}\rightarrow v_{m_0}$ in $H^1(O;x_1^2(x_1+|\overline{x}|)^\xg )$, as $h\rightarrow0^+.$ We can easily show that $z_{h}\rightarrow v_{m_0}$ in $L^2(O ; x_1^2(x_1+|\overline{x}|)^\xg )$. Also,
\bal
\BAL
\int_{O}x_1^2(x_1+|\overline{x}|)^\xg|\nabla (v_{m_0}(1-\eta_h))|^2dx \leq & \; 2\int_{O}x_1^2(x_1+|\overline{x}|)^\xg| \nabla v_{m_0}|^2|(1-\eta_h)|^2dx\\
& +2\int_{O}x_1^2(x_1+|\overline{x}|)^\xg| v_{m_0}|^2|\nabla\eta_h|^2dx\\
\leq & \; 2\int_{O}x_1^2(x_1+|\overline{x}|)^\xg| \nabla v_{m_0}|^2|(1-\eta_h)|^2dx\\
& + C(N,k)m_0^2(\ln h)^{-2}\int_{h^2}^h\int_0^1(x_1+r)^\xg r^{k-2}drdx_1\rightarrow0,
\EAL
\eal
since $\xg\geq -k.$ Thus there exists $h_0\in (0,1)$ such that
\be
\norm{v_{m_0}-z_{h_0}}_{H^1(O ; x_1^2(x_1+|\overline{x}|)^\xg )}<\frac{\xe}{3}.\label{densityeq2}
\ee
Note that $z_{h_0}$ vanishes outside $\tilde{O}_{\xs}=(\xs,1)\times B^{\mathbb{R}^{k-1}}(0,1)\times B^{\mathbb{R}^{N-k}}(0,1),$  for some $\xs=\xs(h_0)\in(0,1).$ Thus $z_{h_0}\in H^1_0(\tilde{O}_{\xs}),$ which implies the existence of a sequence
$\{u_n\}\subset C_c^\infty(\tilde{O}_\xs)$ such that
$u_n\rightarrow z_{h_0}$ in $H^1_0(\tilde{O}_{\xs})$. Hence, there exists $n_0\in \mathbb{N}$ such that
\be
\norm{z_{h_0}-u_{n}}_{H^1(O;x_1^2(x_1+|\overline{x}|)^\xg )}<\frac{\xe}{3},\quad \forall n\geq n_0.\label{densityeq3}
\ee
The desired result follows by \eqref{densityeq1}, \eqref{densityeq2} and \eqref{densityeq3}.
\hfill$\Box$

We write a point $x\in\R^N$ as $x=(x_1,x_2,...,x_k,x_{k+1},...,x_N)=(x_1,\overline{x},x'')$.
Given $r_1,r_2,r_3>0$ we denote
$$
O_{r_1,r_2,r_3}=(0,r_1)\times B^{\mathbb{R}^{k-1}}(0,r_2)\times B^{\mathbb{R}^{N-k}}(0,r_3).
$$
\begin{theorem}\label{density}
Assume that $\xg\geq -k$. Then $C^\infty_c(\xO)$ is dense in
$H^1(\xO; d^2d_K^\xg ).$
\end{theorem}
\begin{proof}
Let $u\in H^1(\xO; d^2d_K^\xg )$ and $\xb_0>0$ be the constant in Lemma \ref{doubling}. Let $\xi\in K$ and $0\leq\xf_\xi\leq1$ be a smooth function with $\text{supp}(\xf_\xi)\subset\mathcal{V}_K(\xi,\frac{\xb_0}{8}),$ and $\xf=1$ in $\mathcal{V}_K(\xi,\frac{\xb_0}{16})$. Then the function
$v= u\xf_\xi$ belongs in $H^1(\xO; d^2d_K^\xg )$.

By \eqref{ddK} we have
\begin{align*}
\int_{\xO}d^2(x)d_K^\xg(x)(|v|^2+|\nabla v|^2)dx&\asymp C(\xO,K)\int_{\mathcal{V}_K(\xi,\frac{\xb_0}{8})}\xd^2(x)(\xd_{2,K}(x)+\xd(x))^\xg (|v|^2+|\nabla v|^2)dx\\
&\asymp C(\xO,K)\int_{O_{1, \frac{\xb_0}{8}, \frac{\xb_0}{8}}}y_1^2(y_1+|\overline{y}|)^\xg(|\overline{v}|^2+|\nabla_y \overline{v}|^2)dy,
\end{align*}
where $\overline{y}=(y_2,...,y_k)$ and
\bal
\overline{v}(y)=
v\left(y_1+\xG^\xi_{1,\partial\xO}\left(y_2+\xG^\xi_{2,K}(y''),...,y_k+\xG^\xi_{k,K}(y''),y'' \right),y_2+\xG^\xi_{2,K}(y''),...,y_k+\xG^\xi_{k,K}(y''),y''\right).
\eal
The desired result follows by Lemma \ref{densityflat} and a partition of unity argument.
\end{proof}

By Corollaries \ref{hardysobolev2} and \ref{hardysobolev}, Theorem \ref{density} and using a partition of unity argument, we obtain the following two results.
\begin{corollary}
\label{hardysobolev3}
Let $\xg\geq0$. There exists $C=C(\xO,K,\xg)$ such that
\[
\left(\int_\xO d^{2} d_K^{\xg}|u|^{\frac{2(N+2+\xg)}{N+\xg}}dx\right)^\frac{N+\xg}{N+2+\xg}\leq C\left(\int_\xO d^{2}d_K^{\xg} |\nabla u|^2dx+\int_\xO d^{2}d_K^{\xg} u^2 dx\right),
\]
for any $u\in H^1(\xO ;d^{2}d_K^{\xg} )$.
\end{corollary}

\begin{corollary}\label{hardysobolev4}
  Let $-k\leq \xg<0$. There exists $C=C(\xO,K,\xg)$ such that
\[
\left(\int_\xO d^{2} d_K^{\xg}|u|^{\frac{2(N+2)}{N}}dx\right)^\frac{N}{N+2}\leq C\left(\int_\xO d^{2}d_K^{\xg\frac{N}{N+2}} |\nabla u|^2dx+\int_\xO d^{2}d_K^{\xg} u^2 dx\right),
\]
for any $u\in H^1(\xO; d^{2}d_K^{\xg} )$.
\end{corollary}

\subsection{Poincar\'e inequality}

\begin{lemma}\label{poincareflat}
Let $1\leq k\leq N$ and $\xg\geq -k$.
Assume that $0<c_0r_2<r_3<r_1<r_2,$ for some constant $0<c_0<1.$
Then there exists a positive constant $C=C(c_0,N,K,\xg)$ such that
\[
\inf_{\xz\in\mathbb{R}}\int_{O_{r_1,r_2,r_3}}|f(x)-\xz|^2x_1^2(x_1+|\overline{x}|)^\xg dx\leq Cr^2_2\int_{O_{r_1,r_2,r_3}}|\nabla f(x)|^2x_1^2(x_1+|\overline{x}|)^\xg dx,
\]
for any $f\in C^1(\overline{O}_{r_1,r_2,r_3}).$
\end{lemma}
\begin{proof}
Let $\xz\in \mathbb{R}$ and $y_1=\frac{x_1}{2r_1},$ $\overline{y}=\frac{\overline{x}}{2r_2}$
and $y''=\frac{x''}{2r_3}.$ Set $\overline{f}(y)=f(2r_1y_1,2r_2\overline{y},2r_3y'').$ Then

\begin{align}\nonumber
\int_{O_{r_1,r_2,r_3}}&|f(x)-\xz|^2x_1^2(x_1+|\overline{x}|)^\xg dx\\
&\asymp C(c_0,N,k,\xg) r_2^{N+\xg+2}\int_{O_{\frac{1}{2},\frac{1}{2},\frac{1}{2}}}|\overline{f}(y)-\xz|^2y_1^2(y_1+|\overline{y}|)^\xg dy.\label{poiflat1}
\end{align}
Let $$\xz_{\overline{f}}=\Big(\int_{O_{\frac{1}{2},\frac{1}{2},\frac{1}{2}}}y_1^2(y_1+|\overline{y}|)^\xg dy\Big)^{-1}\int_{O_{\frac{1}{2},\frac{1}{2},\frac{1}{2}}}\overline{f}(y)y_1^2(y_1+|\overline{y}|)^\xg dy.$$
We assert that there exists a positive constant $C>0$ such that
\begin{align}
\int_{O_{\frac{1}{2},\frac{1}{2},\frac{1}{2}}}|\overline{f}(y)-\xz_{\overline{f}}|^2y_1^2(y_1+|\overline{y}|)^\xg dy\leq C\int_{O_{\frac{1}{2},\frac{1}{2},\frac{1}{2}}}|\nabla\overline{f}(y)|^2y_1^2(y_1+|\overline{y}|)^\xg dy,\label{poiflat2}
\end{align}
for any $\overline{f}\in C^1(\overline{O}_{\frac{1}{2},\frac{1}{2},\frac{1}{2}}).$

We will prove this by contradiction. Let $\{\overline{f}_n\}\subset C^1(\overline{O}_{\frac{1}{2},\frac{1}{2},\frac{1}{2}})$ be a sequence such that
\begin{align}
\int_{O_{\frac{1}{2},\frac{1}{2},\frac{1}{2}}}|\overline{f}_n(y)-\xz_{\overline{f}_n}|^2y_1^2(y_1+|\overline{y}|)^\xg dy> n\int_{O_{\frac{1}{2},\frac{1}{2},\frac{1}{2}}}|\nabla\overline{f}_n(y)|^2y_1^2(y_1+|\overline{y}|)^\xg dy.\label{poiflat3}
\end{align}
Setting
$$
g_n(y)=(\overline{f}_n(y)-\xz_{\overline{f}_n})\Big(\int_{O_{\frac{1}{2},\frac{1}{2},\frac{1}{2}}}|\overline{f}_n(y)-\xz_{\overline{f}_n}|^2y_1^2(y_1+|\overline{y}|)^\xg dy\Big)^{-1},
$$
\eqref{poiflat3} becomes
\begin{align*}
1=\int_{O_{\frac{1}{2},\frac{1}{2},\frac{1}{2}}}|g_n(y)|^2y_1^2(y_1+|\overline{y}|)^\xg dy> n\int_{O_{\frac{1}{2},\frac{1}{2},\frac{1}{2}}}|\nabla g_n(y)|^2y_1^2(y_1+|\overline{y}|)^\xg dy
\end{align*}
and we also have $\xz_{g_n}=0.$

Let $\xe>0$. There exists an extension $\overline{g}_n$ of $g_n$ such that
$\overline{g}_n=g_n$ in $\overline{O}_{\frac{1}{2},\frac{1}{2},\frac{1}{2}},$ $\overline{g}_n\in C^1(\overline{O}_{1,1,1}),$  $\overline{g}_n=0$ if
$y_1>\frac{2}{3}$ or $|\overline{y}|>\frac{2}{3}$ or $|y''|>\frac{2}{3}$ and there exists a positive constant $C_1=C_1(N,k,q)$ such that
\begin{align*}
\int_{O_{1,1,1}}|\overline{g}_n(y)|^qy_1^2(y_1+|\overline{y}|)^\xg dy \leq \; & C_1 \int_{O_{\frac{1}{2},\frac{1}{2},\frac{1}{2}}}|g_n(y)|^qy_1^2(y_1+|\overline{y}|)^\xg dy\\
\int_{O_{1,1,1}}|\nabla\overline{g}_n(y)|^qy_1^2(y_1+|\overline{y}|)^\xg dy \leq \; & C_1
\bigg( \int_{O_{\frac{1}{2},\frac{1}{2},\frac{1}{2}}}|\nabla g_n(y)|^qy_1^2(y_1+|\overline{y}|)^\xg dy\\
& +\int_{O_{\frac{1}{2},\frac{1}{2},\frac{1}{2}}}|g_n(y)|^qy_1^2(y_1+|\overline{y}|)^\xg dy\bigg),
\end{align*}
for any $q>1.$
Assume first that $-k\leq\xg<0$. Given
$\sigma\in (0,1/2)$, by Corollary \ref{hardysobolev} we have
that for some $C=C(\gamma,N,k)$,
\begin{align}\nonumber
&\int_{O_{\xs,\frac{1}{2},\frac{1}{2}}}|g_n(y)|^2y_1^2(y_1+|\overline{y}|)^\xg dy\\ \nonumber
&\leq C\xs^{\frac{6}{N+2}}\Big(\int_{O_{\frac{1}{2},\frac{1}{2},\frac{1}{2}}}|\overline{g}_n(y)|^\frac{2(N+2)}{N}y_1^2(y_1+|\overline{y}|)^\xg dy\Big)^{\frac{N}{N+2}}\\ \nonumber
&\leq C \xs^{\frac{6}{N+2}}\Big(\int_{O_{1,1,1}}|\overline{g}_n(y)|^\frac{2(N+2)}{N}y_1^2(y_1+|\overline{y}|)^\xg dy\Big)^{\frac{N}{N+2}}\\ \nonumber
&\leq C \xs^{\frac{6}{N+2}}\int_{O_{1,1,1}}|\nabla\overline{g}_n(y)|^2y_1^2(y_1+|\overline{y}|)^\xg dy\\ \nonumber
&\leq C \xs^{\frac{6}{N+2}}\Big(\int_{O_{\frac{1}{2},\frac{1}{2},\frac{1}{2}}} \hspace{-.6cm}
|\nabla g_n(y)|^2 y_1^2(y_1+|\overline{y}|)^\xg dy+\int_{O_{\frac{1}{2},\frac{1}{2},\frac{1}{2}}}|g_n(y)|^2y_1^2(y_1+|\overline{y}|)^\xg dy\Big)\\
& \leq C \xs^{\frac{6}{N+2}}(1+\frac{1}{n})\label{poiflat5}.
\end{align}
Similarly in case $\xg\geq 0$, by Corollary \ref{hardysobolev2} we can show
\begin{align}
\int_{O_{\xs,\frac{1}{2},\frac{1}{2}}}|g_n(y)|^2y_1^2(y_1+|\overline{y}|)^\xg dy\leq C(\xg,N,k)\xs^{\frac{2(3+\gamma)}{N+2+\xg}}(1+\frac{1}{n})\label{poiflat5b}
\end{align}
Since $(g_n)$ is bounded in
$H^1((\xs,\frac{1}{2})\times B^{\mathbb{R}^{k-1}}(0,\frac{1}{2})\times
B^{\mathbb{R}^{N-k}}(0,\frac{1}{2}))$
uniformly in $\xs\in (0,\frac{1}{2})$,
by \eqref{poiflat5} and \eqref{poiflat5b}, we can easily show that
there exists a subsequence $(g_{n_k})$ such that $g_{n_k}\rightarrow g$ in
$L^2(O_{\frac{1}{2},\frac{1}{2},\frac{1}{2}} ; y_1^2(y_1+|\overline{y}|)^\xg )$.

But
$$\lim_{n\rightarrow\infty}\int_{O_{\frac{1}{2},\frac{1}{2},\frac{1}{2}}}|\nabla g_n(y)|^2y_1^2(y_1+|\overline{y}|)^\xg dy=0,$$
which implies that $\nabla g=0$ a.e. in $O_{\frac{1}{2},\frac{1}{2},\frac{1}{2}}$.
Hence there exists constant $c$ such that $g=c$ a.e. in $O_{\frac{1}{2},\frac{1}{2},\frac{1}{2}}.$ But $\xz_{g_{n_k}}=0$ and $g_{n_k}\rightarrow g$ in $L^2(O_{\frac{1}{2},\frac{1}{2},\frac{1}{2}}),$ thus
$c=0,$ which is clearly a contradiction since $$\int_{O_{\frac{1}{2},\frac{1}{2},\frac{1}{2}}}| g(y)|^2y_1^2(y_1+|\overline{y}|)^\xg dy=1.$$
Since
\begin{align}
\int_{O_{\frac{1}{2},\frac{1}{2},\frac{1}{2}}}|\nabla\overline{f}(y)|^2y_1^2(y_1+|\overline{y}|)^\xg dy
\asymp C(N,k,\xg)\int_{O_{r_1,r_2,r_3}}r^{-N-\xg}|\nabla f(x)|^2x_1^2(x_1+|\overline{x}|)^\xg dx,\label{poiflat6}
\end{align}
the result follows by \eqref{poiflat1}, \eqref{poiflat2} and \eqref{poiflat6}.
\end{proof}

\begin{theorem}
Assume that $\xg\geq -k$.
Let $\xi\in K$, $x\in V(\xi,\frac{\xb_0}{16})$ and let
$\xb_1$ be the constant in Lemma \ref{doubling}.
Then there exists a positive constant
$C=C(C_\xi,\xO,K,\xg,b)>0$ such that
\be
\inf_{\xz\in\mathbb{R}}\int_{\mathcal{B}(x,r)\cap\xO}|f(y)-\xz|^2d^2(y)d^\xg_K(y)dy\leq Cr^2\int_{\mathcal{B}(x,r)\cap\xO}|\nabla f(y)|^2d^2(y)d^\xg_K(y)dy,\label{poincaresx}
\ee
for any $0<r<\xb_1$ and
$f\in C^1(\overline{\mathcal{B}(x,r)\cap\xO})$.
\label{thm:3.9}
\end{theorem}
\begin{proof}

\minsk

\noindent
\textbf{Case 1. $d(x)\geq b r$} Since $d_K(x)\geq d(x),$ we can easily show that for any $y\in B(x,r)$
$\frac{b-1}{b}d(x)\leq d(y)\leq \frac{b+1}{b}d(x)$
and $\frac{ b-1}{ b}d(x)\leq d_K(y)\leq \frac{ b+1}{ b}d_K(x)$.
Thus the proof of \eqref{poincaresx} follows easily in this case.

\minsk

\noindent\textbf{Case 2. $d(x)\leq b r$ and $d_K(x)>b C_\xi r.$} By \eqref{conddk}, we again have that $\frac{ b-1}{b}d_K(x)\leq d_K(y)\leq \frac{ b+1}{ b}d_K(x)$. Using the last inequality and proceeding as the proof of \cite[Theorem 2.5]{FMT2}, we obtain the desired result.

\minsk

\noindent\textbf{Case 3.  $d(x)\leq b r$ and $d_K(x)\leq b C_\xi r.$}
By \eqref{ddK}, it is enough to prove the following inequality
\[
\inf_{\zeta\in\mathbb{R}}\int_{\mathcal{B}(x,r)\cap\xO}|f-\zeta|^2\xd^2  (\xd_{2,K}  +\xd )^\xg dy
\leq Cr^2\int_{\mathcal{B}(x,r)\cap\xO}|\nabla f|^2\xd^2(\xd_{2,K}+\xd)^\xg dy.
\]
This is a consequence of Lemma \ref{poincareflat}.
\end{proof}

By \eqref{eigenest} and the above theorem, we can easily prove the following result.

\begin{corollary}
Let $\mu\leq k^2/4$ and let $\xb_1$ be the constant in Lemma
\ref{doubling}. Then there exists a constant
$C=C(\xO,K,\xg,b)>0$ such that
for any $0<r<\xb_1$ any $f\in C^1(\overline{\mathcal{B}(x,r)\cap\xO})$
and all $x\in\Omega$ there holds
\[
\inf_{\xz\in\mathbb{R}}\int_{\mathcal{B}(x,r)\cap\xO}|f(y)-\xz|^2\ei^2(y)dy\leq Cr^2\int_{\mathcal{B}(x,r)\cap\xO}|\nabla f(y)|^2 \ei^2(y)dy \, .
\]
\end{corollary}
\begin{proof}
If $\dist(x,K)< \beta_0/16$ the result follows from Theorem
\ref{thm:3.9}. In case $\dist(x,K)> \beta_0/16$ the result is well known.
\end{proof}

In view of the proof of Lemma \ref{poincareflat}, Corollaries \ref{hardysobolev3} and \ref{hardysobolev4} and \eqref{eigenest},  we can prove the following Poincar\'e inequality in $\xO.$

\begin{theorem}
Let  $\mu\leq k^2/4$.  There exists a positive constant $C=C(\xO,K,\mu)$ such that
\be
\inf_{\xz\in\mathbb{R}}\int_{\xO}|f(y)-\xz|^2\ei^2(y)dy\leq C\int_{\xO}|\nabla f(y)|^2\ei^2(y)dy,
\label{poincaresxstoomega}
\ee
for any $f\in C^1(\overline{\xO}).$
\end{theorem}

\subsection{Moser inequality}

\begin{theorem}
Let $\xi\in K,$ $\xg\geq-k,$ $x\in V(\xi,\frac{\xb_0}{16})$ and let
$\xb_1$ be the constant in Lemma \ref{doubling}. Then for any $\xn\geq N+\max\{2,2+\xg\},$ there exists $C=C(\xO,K,\xn,\xb_1)$ such that
\begin{align}\nonumber
\int_{\mathcal{B}(x,r)\cap\xO}|f(y)|^{2(1+\frac{2}{\xn})}d^2(y)d^\xg_K(y)dy&\leq Cr^2\overline{\mathcal{M}}(x,r)^{-\frac{2}{\xn}}
\int_{\mathcal{B}(x,r)\cap\xO}|\nabla f(y)|^2d^2(y)d^\xg_K(y)dy\\
&\times\Big(\int_{\mathcal{B}(x,r)\cap\xO}|f(y)|^{2}d^2(y)d^\xg_K(y)dy\Big)^{\frac{2}{\xn}},\label{mosersx}
\end{align}
for any $0<r<\xb_1$ and all $f\in C^\infty_c(\mathcal{B}(x,r)\cap\xO)$.
\end{theorem}
\begin{proof}
The cases $[d(x)>br]$ and $[d(x)\leq br \mbox{ and } d_K(x)>bC_{\xi}r]$
are proved as in \cite[Theorem 3.5]{FMoT} and
\cite[Theorem 2.6]{FMT2} respectively, using also the inequalities
already obtained in the proof of Lemma \ref{doubling}.

So, let us assume that  $d(x)\leq b r$ and $d_K(x)\leq b C_\xi r$.
We consider first the case where $-k\leq\xg<0.$
By H\"older inequality, we have
\begin{align}\nonumber
&\bigg(\int_{\mathcal{B}(x,r)\cap\xO}|f(y)|^2d^2(y)d^\xg_K(y)dy\bigg)^{\frac{2(\xn-N-2)}{\xn(N+2)}}\\
&\leq \overline{\mathcal{M}}(x,r)^{\frac{4(\xn-N-2)}{\xn(N+2)(N+4)}}
\bigg(\int_{\mathcal{B}(x,r)\cap\xO}|f(y)|^{2(1+\frac{2}{N+2})}d^2(y)d^\xg_K(y)dy\bigg)^{\frac{2(\xn-N-2)}{\xn(N+4)}}. \label{mozer2}
\end{align}
Moreover
\begin{align}\nonumber
\int_{\mathcal{B}(x,r)\cap\xO}&|f(y)|^{2(1+\frac{2}{\xn})}d^2(y)d^\xg_K(y)dy\\ \nonumber
&\leq \overline{\mathcal{M}}(x,r)^{1-\frac{(\xn+2) (N+2)}{\xn(N+4)}}\bigg(\int_{\mathcal{B}(x,r)\cap\xO}|f(y)|^{2(1+\frac{2}{N+2})}d^2(y)d^\xg_K(y)dy\bigg)^{\frac{(\xn+2) (N+2)}{\xn(N+4)}}\\ \nonumber
&=\overline{\mathcal{M}}(x,r)^{1-\frac{(\xn+2) (N+2)}{\xn(N+4)}}\bigg(\int_{\mathcal{B}(x,r)\cap\xO}|f(y)|^{2(1+\frac{2}{N+2})}d^2(y)d^\xg_K(y)dy\bigg)^{1-\frac{2(\xn-N-2)}{\xn(N+4)}}\\  \nonumber
&\leq \overline{\mathcal{M}}(x,r)^{\frac{2}{N+2}-\frac{2}{\xn}}
\int_{\mathcal{B}(x,r)\cap\xO}|f(y)|^{2(1+\frac{2}{N+2})}d^2(y)d^\xg_K(y)dy\\
&\qquad \bigg(\int_{\mathcal{B}(x,r)\cap\xO}|f(y)|^2d^2(y)d^\xg_K(y)dy\bigg)^{-\frac{2(\xn-N-2)}{\xn(N+2)}}, \nonumber \\
&\leq\overline{\mathcal{M}}(x,r)^{\frac{2}{N+2}-\frac{2}{\xn}}
\Big(\int_{\mathcal{B}(x,r)\cap\xO}|f(y)|^{\frac{2(N+2)}{N}}d^2(y)d^\xg_K(y)dy\Big)^{\frac{N}{N+2}} \nonumber \\
&\qquad\Big(\int_{\mathcal{B}(x,r)\cap\xO}|f(y)|^2d^2(y)d^\xg_K(y)dy\Big)^{\frac{2}{\xn}}, \label{moser3}
\end{align}
where in the second to last inequality we have used \eqref{mozer2}.
By Corollary \ref{hardysobolev} and Proposition \ref{propdK}, we have
\begin{align}
\nonumber
&\Big(\int_{\mathcal{B}(x,r)\cap\xO}|f(y)|^{\frac{2(N+2)}{N}}d^2(y)d^\xg_K(y)dy\Big)^{\frac{N}{N+2}}\leq C\int_{\mathcal{B}(x,r)\cap\xO}|\nabla f(y)|^2d^2(y)d^{\frac{\gamma N}{N+2}}_K(y)dy\\
&\leq C r^{-\frac{2\xg}{N+2}}\int_{\mathcal{B}(x,r)\cap\xO}|\nabla f(y)|^2d^2(y)
d^{\xg}_K(y)dy \label{mozer4}
\end{align}
Now, by Lemma \ref{doubling}
\be
 \overline{\mathcal{M}}(x,r)\asymp C(\xO,K,\xg,N,C_\xi,\xb_0) r^{N+\xg+2}. \label{mozer5}
\ee
The desired result follows by \eqref{moser3}, \eqref{mozer4} and \eqref{mozer5}.

If $\xg>0$, the proof of \eqref{mosersx} is similar, the only difference is that we use
Corollary \ref{hardysobolev2} instead of Corollary \ref{hardysobolev}.
\end{proof}
By \eqref{eigenest} and the above theorem, we have
\begin{corollary}
Let $\mu\leq k^2/4$  and let
$\xb_1$ be the constant in Lemma \ref{doubling}.
Then for any $\xn\geq N+\max\{2,2+\xg\},$ there exists
$C=C(\xO,K,\xn,\xb_1)$ such that for any $x\in\Omega$, any
$r\in (0,\beta_1)$ and any $f\in H^1_0(\mathcal{B}(x,r)\cap\xO ;
\phi_{\mu}^2)$ there holds
\begin{align*}
\int_{\mathcal{B}(x,r)\cap\xO}|f|^{2(1+\frac{2}{\xn})}\ei^2 dy&\leq Cr^2\mathcal{M}(x,r)^{-\frac{2}{\xn}}
\Big( \int_{\mathcal{B}(x,r)\cap\xO}|\nabla f|^2\ei^2 dy\Big)\\
&\times\Big(\int_{\mathcal{B}(x,r)\cap\xO}f ^{2}\ei^2 dy\Big)^{\frac{2}{\xn}}.
\end{align*}
\end{corollary}
\subsection{Harnack inequality}

We consider the problem
\be
(\partial_t +L_{\ei})u:= u_t- \ei^{-2}
\text{div}(\ei^2\nabla u)=0,\quad \text{in}\;(0,T)\times\mathcal{B}(x,r)\cap\xO,\label{prombl*}
\ee
for any $T>0$ and $r<\frac{\xb_1}{4}$ where  $\xb_1$ is the constant in Lemma \ref{doubling}.  Similarly with Definition \ref{def1intro}
we have
\begin{definition}
Let $D\subset\xO$ be an open set.
A function $v\in C^1((0,T):H^1(D  ; \ei^2  ))$ is a weak subsolution of $v_t+\CL_{\mu}v=0$ in $(0,T)\times D$ if for any non-negative $\xF\in C^1_c((0,T):C_c^\infty(D))$ we have
\bal
\int_{0}^{T}\int_{D}(v_t\xF+\nabla v\cdot
\nabla\xF ) \ei^2 \, dy \, dt\leq 0.
\eal
\end{definition}
\noindent
We now set
$$
Q=(s-r^2,s)\times\mathcal{B}(x,r)\cap\xO
$$
$$
Q_\xd=(s-\xd r^2,s)\times\mathcal{B}(x,\xd r)\cap\xO.
$$
Now we are ready to apply the Moser iteration argument in order to prove the
Harnack inequality for nonnegative weak solutions. The proof is based on the ideas in the proof of Harnack inequality in noncompact smooth
manifold (see \cite[Chapter 5]{SC1}). Let us note here that
Theorem \ref{density} allows to us to consider test functions in
$C_c^\infty({\mathcal{B}(x,r)}))$ instead of
$C_c^\infty({\mathcal{B}(x,r)\cap\xO}))$.
Thus we are able to prove boundary Harnack inequalities.

Let us first state the $L^p$ mean value inequality for nonnegative subsolutions of the operator $\partial_t +L_{\ei}.$
\begin{theorem}
Let $\mu\leq k^2/4$,
$ \xn\geq N+\max\{2,2+2\xg_+\}$ and $p>0$. There exists a
constant $C(\xn,\xl,\xb_1,p,\xO,K)$ such that for any $x\in\xO$  and for any positive subsolution $v$ of (\ref{prombl*}) in $Q$ we have the estimate
$$\sup_{Q_\xd}|v|^p\leq\frac{C}{(\xd'-\xd)^{\xn+2}r^{2}\overline{\mathcal{M}}(x,r)}\int_{Q_{\xd'}}|v|^p \ei^2 \, dy \,  dt,$$
for each $0<\xd<\xd'\leq1.$\label{th1*}
\end{theorem}
The proof of the above theorem is similar to the proof of \cite[Theorem 5.2.9]{SC1} and we omit it (see also \cite[Theorem 2.12]{FMT2}).
Similarly one can establish the proof of the parabolic Harnack
inequality up to the boundary of Theorem \ref{Harnack}.

\medsk

Let  $k(t,x,y)$ be the heat kernel of the problem
\[
\BAL
\tarr{v_t=-L_{\mu} v,}{\mathrm{ in }\;\;(0,T]\times\xO,}
{v=0,}{\mathrm{on}\;\;(0,T]\times\partial\xO,}
{v(0,x)=v_0(x),}{\mathrm{in}\;\;\xO.}
\EAL
\]
By the parabolic Harnack inequality \eqref{harnackpar},
and following the methods of Grigoryan and Saloff-Coste
(see for example \cite[Theorem 2.7]{GS} and \cite[Theorem 5.4.12]{SC1})
we obtain the following
sharp two-sided heat kernel estimate for small time:
\begin{theorem}\label{weismalltime1}
Let $\xb_1$ be the constant of Lemma \ref{doubling}. Then there exist positive constants $A_1,\;A_2,$ $C_1$ and $C_2,$ such that for all $x,\;y\in \xO$ and all $0<t<\frac{\xb_1^2}{4}$ the heat kernel $k(t,x,y)$ satisfies
\[
\BAL
\frac{C_1}{\mathcal{M}^\frac{1}{2}(x,\sqrt{t})\mathcal{M}^\frac{1}{2}(y,\sqrt{t})}\exp\Big(-A_1\frac{|x-y|^2}{t}\Big)&\leq k(t,x,y)\\
&\hspace{-2cm}\leq\frac{C_2}{\mathcal{M}^\frac{1}{2}(x,\sqrt{t})\mathcal{M}^\frac{1}{2}(y,\sqrt{t})}\exp\Big(-A_2\frac{|x-y|^2}{t}\Big).
\EAL
\]
\end{theorem}

\begin{theorem}
\label{weimaintheorem}
Let $\xm\leq \frac{k^2}{4}$. There exists positive constants
$T=T(\xO,K,\xm)>0$ and $C=C(\xO,K,\xm)>1$ such that the heat kernel
$k(t,x,y)$ satisfies
\begin{eqnarray*}
&\hspace{-.4cm} \ia &
\hspace{-.3cm}C^{-1}\Big(\frac{1}{(d(x)+\sqrt{t})(d(y)+\sqrt{t})}\Big)
\Big(\frac{1}{(d_K(x)+\sqrt{t})(d_K(y)+\sqrt{t})}\Big)^{\xg_+}
t^{-\frac{N}{2}}\exp\Big(-C\frac{|x-y|^2}{t}\Big)\\[0.2cm]
&& \hspace{3cm} \leq k(t,x,y)\leq\\
&&  C\Big(\frac{1}{(d(x)+\sqrt{t})(d(y)+\sqrt{t})}\Big)
\Big(\frac{1}{(d_K(x)+\sqrt{t})(d_K(y)+\sqrt{t})}\Big)^{\xg_+}
t^{-\frac{N}{2}}\exp\Big(\!\! -C^{-1}\frac{|x-y|^2}{t}\Big) \\
&& \mbox{for any $0<t\leq T$ and $x,y\in\xO$.} \\[0.2cm]
&\hspace{-.4cm} \ib &  C^{-1}\leq k(t,x,y)\leq C \, \;\;
 \mbox{for any $ t>T$ and $x,y\in\xO$.}
\end{eqnarray*}
\end{theorem}
\noindent
\emph{Proof of Theorem \ref{weimaintheorem} (i).} This follows
easily from Theorem \ref{weismalltime1}
and Corollary \ref{doublingcoro}. \hfill$\Box$

\section{Heat kernel estimates for large time}

\subsection{Weighted logarithmic Sobolev inequality}

\begin{theorem}
\label{logth}
Let $\mu\leq k^2/4$. There exists a positive constant
$C=C(\xO,K,\mu)$ such that for any $\epsilon >0$ there holds
\be
\int_\xO u^{2}\ln\frac{|u|}{\norm{u}_{L^2(\xO ; \ei^2)}}
\ei^2dx\leq \xe\int_\xO  |\nabla u|^2\ei^2dx+ b(\xe)\int_\xO  u^2\ei^2dx,\label{log}
\ee
for all $u\in H^1(\xO ; \ei^2 )$;
here $b(\xe)=C-\frac{N+2+\max(\xg_+,0)}{4}\min(\ln\xe,0)$.
\end{theorem}
\begin{proof}
We may assume that $\norm{u}_{L^2(\xO ; \ei^2)}=1$.
Assume first that $-\frac{k}{2}\leq\xg_+<0$. Then
\begin{align*}
\int_\xO |u|^{2}\ln |u| \ei^2dx&=\frac{N}{4} \int_\xO |u|^{2}\ln|u|^{\frac{4}{N}}\ei^2dx\\
&\leq \frac{N}{4}\ln\bigg( \int_\xO |u|^{\frac{2(N+2)}{N}}\ei^2dx \bigg)\\
&=\frac{N+2}{4}\ln\bigg( \Big(\int_\xO |u|^{\frac{2(N+2)}{N}}\ei^2dx \Big)^{\frac{N}{N+2}}\bigg)\\
&\leq \frac{N+2}{4}\ln \left(C_0\left(\int_\xO  |\nabla u|^2\ei^2dx+\int_\xO  |u|^2\ei^2dx\right)\right),
\end{align*}
where in the last inequality, we used Corollary \ref{hardysobolev4} and \eqref{eigenest}. Using the fact that $\frac{N+2}{4}\log\theta= \frac{N+2}{4}\ln\frac{4\xe\theta}{C_0(N+2)} +\frac{N+2}{4}\ln\frac{C_0(N+2)}{4\xe}, \;\forall\xe,\theta>0,$ we obtain the desired result with $b(\xe)=1+\frac{N+2}{4}(\ln C_0 +\ln\frac{N+2}{4}-\ln\xe),$ if $0<\xe\leq 1$

Similarly, if $\xe\geq1$ and $-\frac{k}{2}\leq\xg_+<0,$ we obtain the desired result with $b(\xe)=1+\frac{N+2}{4}(\ln C_0 +\ln\frac{N+2}{4}).$

If $\xg_+>0$ we proceed as above and we use Corollary \ref{hardysobolev3} instead of Corollary \ref{hardysobolev4}, in order to obtain \eqref{log} with $b(\xe)=1+\frac{N+2+2\xg_+}{4}(\ln C_1 +\ln\frac{N+2+2\xg_+}{4}-\ln\xe),$ where $C_1$ is the constant in Corollary \ref{hardysobolev3}.
\end{proof}

\begin{theorem}
Let $\mu\leq k^2/4$ and
let  $u\in H^1(\xO ; \ei^2 )$ be such that $\int_\xO u \, \ei^2 dx=0.$ There exists a positive constant $C=C(\xO,K,\mu)$ such that for
any $\epsilon >0$ there holds
\[
\int_\xO u^{2}
\ln\frac{|u|}{\norm{u}_{L^2(\xO ; \ei^2)}}\ei^2 dx\leq \xe\int_\xO |\nabla u|^2 \ei^2
dx+ b(\xe)\int_\xO u^2\ei^2 dx,
\]
where $b(\xe)=C-\frac{N+2+\max(2\xg_+,0)}{4} \ln \xe $.
\end{theorem}
\begin{proof}
By \eqref{poincaresxstoomega} and in view of the proof of \eqref{log} we obtain the desired result.
\end{proof}

\noindent {\bf\em{Proof of Theorem \ref{weimaintheorem} (ii).}}
We normalize $\ei$ so that $\int _\xO\ei^2 dx=1$.
We define the bilinear form
$Q: H_0^1(\xO ; \ei^2)\times H_0^1(\xO ;\ei^2)\to \R$  by
$$Q(u,v)=\int_\xO\nabla u \cdot \nabla v \, \ei^2 dx.$$
We recall here that $H^1(\xO ; \ei^2)=H^1_0(\xO ; \ei^2)$ by \eqref{eigenest} and Theorem \ref{density}.

Let $\CL_{\mu}$ denote the self-adjoint operator on
$L^2(\xO ; \ei^2)$ associated to the form $Q$, so that formally we may write
\[
\CL_{\mu}u=- \ei^{-2} \, \text{div}\left(\ei^2\nabla  u\right).
\]
The operator $\CL_{\mu}$ generates a contraction semigroup
$T(t): L^2(\xO ; \ei^2)\to L^2(\xO ; \ei^2)$, $t\geq 0$, denoted also by  $e^{-\CL_{\mu}t}$. This semigroup is positivity preserving and by \cite[Lemma 1.3.4]{D2} we can easily show that satisfies the conditions of \cite[Theorems 1.3.2 and 1.3.3]{D2}. Thus, by \eqref{log}, we can apply \cite[Corollary 2.2.8]{D2} to deduce that
\be
\| e^{-\CL_{\mu} t} u\|_{L^\infty(\xO) }\leq C_t \|u\|_{L^2(\xO ; \ei^2)},
\quad \quad  t>0, \;\; u\in L^2(\Omega ; \phi_{\mu}^2),
\label{6}
\ee
where
$$
C_t=e^{\frac{1}{t}\int_0^t b(\xe)d\xe}.
$$
Hence, by \cite[Lemma 2.1.2]{D2}, $e^{-\CL_{\mu}t }$ is ultracontractive and has a kernel $k(t,x,y)$ such that
$$
0\leq k(t,x,y)\leq C_t.
$$
By the last inequality, the upper estimate in Theorem
\ref{weimaintheorem} (ii) follows easily.

For the lower estimate in \ref{weimaintheorem} (ii) we will give two proofs. One using the boundary Harnack inequality \eqref{harnackpar} and the other one proceeding as the proof of \cite[Theorem 6]{D1}.

\vspace{1em}

\noindent \textbf{First proof} (as in the proof of
\cite[Theorem 6]{D1}). First we note that since $H^1(\xO ;\ei^2)$ is
compactly embedded in  $L^2(\xO ;\ei^2)$, the operator $\CL_{\mu}$
has compact resolvent. In addition, we have that $\CL_{\mu}1=0$ and
hence, by \eqref{poincaresxstoomega},
$$
{\rm sp}(\CL_{\mu})\subset\{0\}\cup[\xl,\infty),
$$
for some $\xl>0.$ Thus, using the spectral theorem, we can easily show that for any $f\in L^2(\xO ;\ei^2)$ such that $\int_\xO f\ei^2dx=0$
we have
\be
\|e^{-\CL_{\mu}t}f\|_{L^2(\xO ; \ei^2)}\leq e^{-\xl t}
\norm{f}_{L^2(\xO ; \ei^2)},\qquad\forall t\geq 0.
\label{7}
\ee
Now, let  $f\in L^1(\xO ;\ei^2)$ and $\int_\xO f\ei^2dx=0.$ By \eqref{6} and \eqref{7}, we have
\bal
\|e^{-\CL_{\mu}t}f\|_{L^\infty(\xO)}= \|e^{-\CL_{\mu}\frac{t}{3}}
\big(e^{-\CL_{\mu}\frac{2t}{3}}f\big)\|_{L^\infty(\xO)}\leq C_{\frac{t}{3}}
\|e^{-\CL_{\mu}\frac{2t}{3}}f\|_{L^2(\xO ; \ei^2)}\leq e^{-\frac{\xl t}{3}}C_{\frac{t}{3}}
\|e^{-\CL_{\mu}\frac{t}{3}}f\|_{L^2(\xO ; \ei^2)}.
\eal
Taking adjoints we have
$$
\|e^{-\CL_{\mu}\frac{t}{3}}f\|_{L^2(\xO ; \ei^2)}
\leq C_{\frac{t}{3}}\norm{f}_{L^1(\xO ; \ei^2)},
$$
hence
\[
\|e^{-\CL_{\mu}t}f\|_{L^\infty(\xO)}
\leq e^{-\frac{\xl t}{3}}C_{\frac{t}{3}}^2\norm{f}_{L^1(\xO ; \ei^2)}.
\]
Let now $f\in L^1(\xO ;\ei^2)$. The function
$g:=f-\int_\xO f\ei^2 dx$ satisfies
$\int_\xO g\ei^2 dx=0$, thus
\[
e^{-\CL_{\mu}t}g=e^{-\CL_{\mu}t}f-\inprod{f}{1}_{L^2(\xO ; \ei^2)}.
\]
Hence the operator
\[
\tilde T(t)f=e^{-\CL_{\mu}t}f-\inprod{f}{1}_{L^2(\xO ;\ei^2)}
\]
satisfies
$$
\|\tilde T(t)f\|_{L^\infty(\xO)}=\|e^{-\CL_{\mu}t}g\|_{L^\infty(\xO)}\leq e^{-\frac{\xl t}{3}}C_{\frac{t}{3}}^2
\norm{g}_{L^1(\xO;\ei^2)}
\leq 2e^{-\frac{\xl t}{3}}C_{\frac{t}{3}}^2\norm{f}_{L^1(\xO;\ei^2)}.
$$
Therefore the integral kernel $\tilde k(t,x,y)$ of  $\tilde T(t)$
satisfies $\tilde k(t,x,y)=k(t,x,y)-1$ and
$$|\tilde k(t,x,y)|\leq 2e^{-\frac{\xl t}{3}}C_{\frac{t}{3}}^2.$$
The desired result follows if we choose $t$ large enough.

\vspace{1em}

\noindent\textbf{Second proof }(using the boundary Harnack inequality \eqref{harnackpar}).
Let $x_0 \in \xO.$ Then by \eqref{harnackpar} we can show that
$$
k(t-1,x,y)\leq C(\xO,K) k(t,x,x_0),
$$
for all $t\geq 2$ and $x,y\in \xO.$ Thus,
\begin{align*}
1=\int_\xO k(t-1,x,y)\ei^2(y)dy&\leq C(\xO,K)\int_\xO k(t,x,x_0)\ei^2(y)dy\\
&=C(\xO,K)k(t,x,x_0),\quad\forall t\geq2.
\end{align*}
The desired result follows.
\end{proof}

\subsection{Green function estimates}
\label{sub:greek}

In this subsection we prove the existence of the Green kernel of $L_\xm$
along with sharp two-sided estimates.

\begin{proposition}
\label{green}
Let $\mu\leq k^2/4$ and assume that $\lambda_{\mu}>0$.
For any $y \in \Omega $ there exists a minimal Green function  $G_{\mu}(\cdot,y)$ of the equation
\[
L_\mu u =\delta_y \quad \text{in } \Omega,
\]
where $\delta_y$ denotes the Dirac measure at $y$.
Furthermore, the following estimates hold
\ba \label{Greenest}
	G_{\xm}(x,y)&\asymp\left\{
	\BAL
	&|x-y|^{2-N}\min\Big\{1 ,\frac{d(x)d(y)}{|x-y|^2}\Big\} \left(\frac{d_K(x)d_K(y)}{\left(d_K(x)+|x-y|\right)\left(d_K(y)+|x-y|\right)}\right)^{\xg_+},\\
	& \hspace{8cm} \text{if } \xgp >-\frac{N}{2} , \\
	&|x-y|^{2-N} \min\Big\{1,\frac{d(x)d(y)}{|x-y|^2}
	\Big\}
	\Big(
	\frac{|x| \, |y|}{\left(|x|+|x-y|\right)\left(|y|+|x-y|\right)}
\Big)^{-\frac{N}{2}}\\
	&\quad \quad +\frac{d(x)d(y)}{(|x||y|)^{\frac{N}{2}}}\left|\ln\Big(\min\Big\{ \frac{1}{|x-y|^2}, \frac{1}{d(x)d(y)} \Big\}\Big)\right|,\quad\text{if }\; \xgp=-\frac{N}{2} .
	\EAL\right.
\ea
\end{proposition}
\begin{proof}
First, let $C_1>0$ and $T$ be as in Theorem \ref{maintheorem1}. We note that
\be\label{estside1}
\left(\Big(\frac{\sqrt{t}}{d(x)}+1\Big)\Big(\frac{\sqrt{t}}{d(y)}+1\Big)\right)^{-1}=\frac{d(x)d(y)}{(\sqrt{t}+d(x))(\sqrt{t}+d(y))}
\leq\min\{1,\frac{d(x)d(y)}{t}\}
\ee
and
\be\label{estside2} \BAL
\left(\Big(\frac{\sqrt{t}}{d(x)}+1\Big)\right.&\left.\Big(\frac{\sqrt{t}}{d(y)}+1\Big)\right)^{-1}e^{-\frac{C_1|x-y|^2}{t}}
=\frac{d(x)d(y)}{(\sqrt{t}+d(x))(\sqrt{t}+d(y))}e^{-\frac{C_1|x-y|^2}{t}} \\
&\geq C
\min\{1,\frac{d(x)d(y)}{t}\}
e^{-\frac{(1+C_1)|x-y|^2}{t}}
\EAL \ee
for all $x,y\in\xO$ and $0<t<T$, where $C=C(C_1,T)>0$.

By Theorem \ref{maintheorem1}, \eqref{eigenest} and estimates \eqref{estside1}--\eqref{estside2},  there exist $C_i=C_i(\xO,K,\mu)>0$, $i=1,2$ and $T=T(\xO,K,\mu)>0$ such that
for $t \in (0,T)$ and $x,y\in \xO$,
\begin{align}
 \BAL C_1\min\Big\{1,\frac{d(x)d(y)}{t}\Big\}
\Big(\frac{d_K(x)}{d_K(x)+\sqrt{t}}\Big)^{\xg_+}
\Big(\frac{d_K(y)}{d_K(y)+\sqrt{t}}\Big)^{\xg_+}t^{-\frac{N}{2}}&e^{-\frac{C_2|x-y|^2}{t}}\leq h(t,x,y)\\
\leq C_2\min\Big\{1,\frac{d(x)d(y)}{t}\Big\}
\Big(\frac{d_K(x)}{d_K(x)+\sqrt{t}}\Big)^{\xg_+}
\Big(\frac{d_K(y)}{d_K(y)+\sqrt{t}}\Big)^{\xg_+}&t^{-\frac{N}{2}}e^{-\frac{C_1|x-y|^2}{t}},
\EAL\label{heatest1}
\end{align}
while
\be
 C_1\leq
\frac{ h(t,x,y)}{d(x)d(y)d_K^{\xgp}(x)d_K^{\xgp}(y)e^{-\xl_\xm t}} \leq  C_2 ,\qquad\forall t\geq T, \; x,y\in \xO.
\label{heatest2}
\ee
By \eqref{heatest1} and \eqref{heatest2}, we deduce the existence of
the minimal Green kernel $G_{\xm}$ of $L_{\mu}$, given by
\be\label{Gh}
G_{\xm}(x,y)=\int_0^\infty h(t,x,y)dt= \int_0^Th(t,x,y)dt+\int_T^\infty h(t,x,y)dt.
\ee
Using \eqref{heatest2} we easily see that
the second integral in \eqref{Gh} satisfies the required upper estimate
in both cases considered (i.e. $\gamma_+ > -\frac{N}{2}$ or
$\gamma_+ = -\frac{N}{2}$).
We now concentrate on the first integral in \eqref{Gh}.

By the change of variable $s=\frac{|x-y|^2}{t}$, we obtain for $i=1,2$,
\begin{align*}
&\int_0^T\min\Big\{1,\frac{d(x)d(y)}{t}\Big\}
\Big(\frac{d_K(x)}{d_K(x)+\sqrt{t}}\Big)^{\xg_+}
\Big(\frac{d_K(y)}{d_K(y)+\sqrt{t}}\Big)^{\xg_+}
t^{-\frac{N}{2}}e^{-\frac{C_i|x-y|^2}{t}}dt=|x-y|^{2-N} \\
&\int_{\frac{|x-y|^2}{T}}^\infty\min
\Big\{1,s\frac{d(x)d(y)}{|x-y|^2}\Big\}
\left(\Big(\frac{|x-y|}{\sqrt{s}d_K(x)}+1\Big)\Big(\frac{|x-y|}{\sqrt{s}d_K(y)}+1\Big)\right)^{-\xgp}
s^{\frac{N}{2}-2}e^{-C_i s}ds \\
& =: |x-y|^{2-N}S_i(x,y) \, .
\end{align*}
By \eqref{heatest1} we therefore have for some $c_1,c_2>0$ that
\begin{equation}
c_1 |x-y|^{2-N} S_2(x,y)  \leq  \int_0^T h(t,x,y)dt    \leq c_2
|x-y|^{2-N} S_1(x,y) \; , \qquad x,y\in\Omega.
\label{1234}
\end{equation}
In the sequel, we assume that $\frac{|x-y|^2}{T}<\frac{1}{2}.$  The proof in the case
$\frac{|x-y|^2}{T} > \frac12$ is similar, indeed simpler.
We write
\ba \label{3green}
\BAL
&S_1=\int_{\frac{|x-y|^2}{T}}^{1}
\min\Big\{1,s\frac{d(x)d(y)}{|x-y|^2}\Big\}
\left(\Big(\frac{|x-y|}{\sqrt{s}d_K(x)}+1\Big)\Big(\frac{|x-y|}{\sqrt{s}d_K(y)}+1\Big)\right)^{-\xgp}
s^{\frac{N}{2}-2}e^{-C_1 s}ds\\
&\quad\quad +\int_{1}^\infty\min
\Big\{1,s\frac{d(x)d(y)}{|x-y|^2}\Big\}
\left(\Big(\frac{|x-y|}{\sqrt{s}d_K(x)}+1\Big)\Big(\frac{|x-y|}{\sqrt{s}d_K(y)}+1\Big)\right)^{-\xgp}
s^{\frac{N}{2}-2}e^{-C_1 s}ds.
\EAL
\ea
Concerning the second term in the RHS of \eqref{3green}
we have
\begin{eqnarray*}
&& \int_{1}^\infty \min\Big\{1,s\frac{d(x)d(y)}{|x-y|^2}\Big\}
\bigg(\Big(\frac{|x-y|}{\sqrt{s}d_K(x)}+1\Big)\Big(\frac{|x-y|}{\sqrt{s}d_K(y)}+1\Big)\bigg)^{-\xgp}
s^{\frac{N}{2}-2}e^{-C_1 s}ds\\
&& \qquad  \leq C \min\Big\{1,\frac{d(x)d(y)}{|x-y|^2}\Big\}
\left(\Big(\frac{|x-y|}{d_K(x)}+1\Big)\Big(\frac{|x-y|}{d_K(y)}+1\Big)\right)^{-\xgp},
\end{eqnarray*}
and therefore the required estimate is satisfied.

Let $\xg_+\leq0$. For the first term in the RHS of \eqref{3green} we have
\begin{align}
&\int_{\frac{|x-y|^2}{T}}^{1}\min\Big\{1,s\frac{d(x)d(y)}{|x-y|^2}
\Big\}
\left(\Big(\frac{|x-y|}{\sqrt{s}d_K(x)}+1\Big)\Big(\frac{|x-y|}{\sqrt{s}d_K(y)}+1\Big)\right)^{-\xgp}
s^{\frac{N}{2}-2}e^{-C_1 s}ds\nonumber \\
&=\int_{\frac{|x-y|^2}{T}}^{1}\min\Big\{1,s\frac{d(x)d(y)}{|x-y|^2}
\Big\}
\left(\Big(\frac{|x-y|}{d_K(x)}+ \sqrt{s}\Big)\Big(\frac{|x-y|}{d_K(y)}
+\sqrt{s} \Big)\right)^{-\xgp}
s^{\frac{N}{2}+\xgp-2}e^{-C_1 s}ds\nonumber \\
&\leq C\left(|x-y|^{-2\xgp}\big(d_K(x)d_K(y)\big)^{\xgp}
\int_{\frac{|x-y|^2}{T}}^1\min\Big\{1,s\frac{d(x)d(y)}{|x-y|^2}\Big\}
s^{\frac{N}{2}+\xgp-2}e^{-C_1 s}ds  \right.\nonumber \\
&\;\; \left.  +|x-y|^{-\xgp} \big(d_K(x)d_K(y)\big)^{\!\xgp}
\!\!\!\!\!\int_{\frac{|x-y|^2}{T}}^{1}\!\!\!
\min\Big\{ \! 1,s\frac{d(x)d(y)}{|x-y|^2}\Big\}(d_K(x)+d_K(y))^{-\xgp}
s^{\frac{N}{2}+\frac{\xgp}{2}-2}e^{-C_1 s}ds\right.\nonumber \\
&\quad \left.+\int_{\frac{|x-y|^2}{T}}^{1}
\min\Big\{1,s\frac{d(x)d(y)}{|x-y|^2}\Big\}
s^{\frac{N}{2}-2}e^{-C_1 s}ds\right)\nonumber \\
&=: C(J_1+J_2+J_3)  \label{5green}
\end{align}
It is easily seen that
\[
J_3 \leq C \min\Big\{1,\frac{d(x)d(y)}{|x-y|^2}\Big\} \,  .
\]
Concerning $J_1$ and $J_2$ we consider two cases.

\minsk
\noindent
{\bf Case I.} $-\frac{N}{2}<\gamma_+ \leq 0$.
In view of \eqref{1234} and \eqref{5green}, it is enough to establish
that for $i=1,2$ we have
\be
J_i \leq
\min\Big\{1 ,\frac{d(x)d(y)}{|x-y|^2}\Big\} \left(\frac{d_K(x)d_K(y)}{\left(d_K(x)+|x-y|\right)\left(d_K(y)+|x-y|\right)}\right)^{\xg_+}
\!\! , \quad i=1,2.
\label{1234a}
\ee
In order to prove \eqref{1234a} we shall need to consider
additional cases.

\minsk
\noindent{\bf\em Case Ia.} $\frac{d(x)d(y)}{|x-y|^2} \leq 1$. In this case
it is immediate that
\[
J_1 =C |x-y|^{-2-2\xgp}\big(d_K(x)d_K(y)\big)^{\xgp} d(x)d(y).
\]
and
\[
J_2 = C|x-y|^{-2-\xgp} (d_K(x)d_K(y))^{\xgp} \big(d_K(x)+d_K(y)
\big)^{-\xgp}
d(x)d(y).
\]
Hence inequality \eqref{1234a} is satisfied.

\minsk

\noindent
{\bf\em Case Ib.} $\frac{d(x)d(y)}{|x-y|^2} > 1$.
In this case we have
$\frac{1}{4}d_K(y)\leq d_K(x)\leq 4d_K(y)$. Indeed, suppose that
$d_K(x) > 4d_K(y)$. Then, since $d_K(x)\leq |x-y|+d_K(y)$, we easily obtain that $d_K(y)\leq \frac13  |x-y|$ and
$d_K(x)\leq \frac43 |x-y|$; hence $d(x)d(y)\leq \frac49 |x-y|^2$, a contradiction.

To proceed we first note that
\begin{eqnarray}
J_1 &\leq & |x-y|^{-2\xgp}\big(d_K(x)d_K(y)\big)^{\xgp}\left(
\frac{d(x)d(y)}{|x-y|^2}\int_0^{\frac{|x-y|^2}{d(x)d(y)}}
s^{\frac{N}{2}+\xgp-1}e^{-C_1 s}ds\right.  \nonumber \\
&&\hspace{3.5cm} \left.+\int_{\frac{|x-y|^2}{d(x)d(y)}}^1
s^{\frac{N}{2}+\xgp-2}e^{-C_1 s}ds\right)
\label{j1}
\end{eqnarray}
and similarly
\begin{eqnarray}
J_2 &\leq  & |x-y|^{-\xgp} \Big(\frac{  d_K(x)d_K(y)}{ d_K(x)+d_K(y)}
\Big)^{\xgp}\left(
\frac{d(x)d(y)}{|x-y|^2}\int_0^{\frac{|x-y|^2}{d(x)d(y)}}
s^{\frac{N}{2}+ \frac{\xgp}{2}-1}e^{-C_1 s}ds\right.  \nonumber \\
&&\hspace{4cm} \left.+\int_{\frac{|x-y|^2}{d(x)d(y)}}^1
s^{\frac{N}{2}+\frac{\xgp}{2}-2}e^{-C_1 s}ds\right)
\label{j2}
\end{eqnarray}
We  now consider different subcases.

\minsk
\noindent {\em Case 1.} $-\frac{N}{2}<\xgp< -N+2$.
From \eqref{j1} and \eqref{j2} we obtain
\[
J_1 \leq c  , \qquad J_2  \leq c.
\]
It  follows that \eqref{1234a} is satisfied.

\minsk
\noindent {\em Case 2.} $\xgp = -N+2>-\frac{N}{2}$.
In this case \eqref{j1} and \eqref{j2} give
\[
J_1 \leq c
\]
and
\[
J_2 \leq c |x-y|^{-\xgp} \Big(\frac{  d_K(x)d_K(y)}{ d_K(x)+d_K(y)}
\Big)^{\xgp}\bigg(1+ \ln\Big( \frac{d(x)d(y)}{ |x-y|^2} \Big)\bigg)\leq c
\]
Again it is easily seen that \eqref{1234a} is satisfied.

\minsk
\noindent {\em Case 3.} $ \max\{-\frac{N}{2},-N+2 \}< \xgp < -\frac{N-2}{2}$. In this case we obtain
\[
J_1 \leq c ,
\qquad
J_2 \leq c  |x-y|^{-\xgp} \Big(\frac{  d_K(x)d_K(y)}{ d_K(x)+d_K(y)}
\Big)^{\xgp}\leq c
\]
and  \eqref{1234a}  once again follows.

\minsk
\noindent {\em Case 4.} $\xgp = -\frac{N-2}{2}  <0$. In this case we
obtain
\bal
J_1 &\leq c|x-y|^{-2\xgp}\big(d_K(x)d_K(y)\big)^{\xgp} \left(1+ \ln\Big( \frac{d(x)d(y)}{ |x-y|^2} \Big)\right)\leq c,
\\
J_2 &\leq c  |x-y|^{-\xgp} \Big(\frac{  d_K(x)d_K(y)}{ d_K(x)+d_K(y)}
\Big)^{\xgp}\leq c
\eal
and  \eqref{1234a}  once again follows.
\minsk

\noindent {\em Case 5.} $-\frac{N-2}{2}<\xgp\leq 0$. In this case we
obtain
\[
J_1 \leq c|x-y|^{-2\xgp}\big(d_K(x)d_K(y)\big)^{\xgp} \leq c,
\qquad J_2 \leq c  |x-y|^{-\xgp}
\Big(\frac{  d_K(x)d_K(y)}{ d_K(x)+d_K(y)}\Big)^{\xgp}\leq c
\]
and  \eqref{1234a}  once again follows.

\minsk
\noindent{\bf Case II.} $\gamma_+ = -\frac{N}{2}$.
The proof is very similar to the previous case and for the sake
of brevity we shall only
consider $J_1$, where the main difference appears. We note that
in this case we have $d_K(x)=|x|$.

We  assume that
$\frac{|x-y|^2}{T} \leq \frac12$. The proof in the case
$\frac{|x-y|^2}{T} > \frac12$ is similar, indeed simpler.

\minsk
\noindent
{\bf\em Case IIa.} $\frac{d(x)d(y)}{|x-y|^2} \leq 1$.
In this case we easily obtain
\[
J_1 \leq c \, |x-y|^{N-2} d(x)d(y)
( |x| \, |y| )^{-\frac{N}{2}}\log\Big(\frac{T}{|x-y|^2}\Big),
\]
and this is estimated using the second term in the RHS of \eqref{Greenest}.

\minsk
\noindent{\bf\em Case IIb.} $\frac{d(x)d(y)}{|x-y|^2} \geq 1$.
We may assume that $\frac{|x-y|^2}{d(x)d(y)} > \frac{|x-y|^2}{T} $,
otherwise we need only consider the second of the two integrals
below.

We have
\begin{eqnarray*}
J_1 &=& |x-y|^{N} ( |x| \, | y|)^{-\frac{N}{2}}\bigg(
\frac{d(x)d(y)}{|x-y|^2}\int_{\frac{|x-y|^2}{T}}^{\frac{|x-y|^2}{d(x)d(y)}}
s^{-1}e^{-C_1 s}ds +\int_{\frac{|x-y|^2}{d(x)d(y)}}^1
s^{-2}e^{-C_1 s}ds\bigg)  \\
&\leq &  c |x-y|^{N-2}d(x)d(y)
\big(|x| \, | y|\big)^{-\frac{N}{2}}
\log\Big(\frac{T}{d(x)d(y)}\Big) ,
\end{eqnarray*}
which satisfies the upper bound in \eqref{Greenest}.
Hence the upper bound has been established in all cases.

This concludes the proof of the upper estimate when $\xgp\leq 0$.
If $\xgp>0$ then the proof is essentially similar, indeed simpler, and
is omitted.

The proof of the lower bound is much simpler. For example,
in case $\gamma_+ \leq 0$ we have from \eqref{1234}
\[
G_{\mu}(x,y) \geq c_1 |x-y|^{2-N} S_2(x,y)
\geq c  |x-y|^{2-N} J_1(x,y) ,
\]
where $J_1$ is as above, the only difference being that the exponential
factor in the integrand is  $e^{-C_2s}$ instead of $e^{-C_1s}$. The result
then follows easily.
\end{proof}

\section{The linear elliptic problem}
\label{section5}

\subsection{Subsolutions and Supersolutions}

We recall the definition of the function $\tilde{d}_K$ from (\ref{dtilda}).
Given parameters $\epsilon>0$ and $M\in\R$ we define the functions
\[
\begin{array}{ll}
\eta_{\xg_+,\varepsilon}= e^{-M d}
(d+\tilde d^2_K)\tilde d_{K}^{\xg_+}-d\tilde d_{K}^{\xg_++\xe}
& \!\!	  \zeta_{\xgp,\varepsilon}=e^{M d} (d+\tilde d^2_K)
\tilde d_{K}^{\xg_+}+d\tilde d_{K}^{\xg_++\xe} \\
\eta_{\xg_-,\varepsilon}=e^{-M d}(d+\tilde d^2_K)
\tilde d_{K}^{\xg_-}+d\tilde d_{K}^{\xg_-+\xe}
& \!\!\zeta_{\xg_-,\varepsilon}=e^{M d}(d+\tilde d^2_K)
\tilde d_{K}^{\xg_-}-d\tilde d_{K}^{\xg_-+\xe} \\
\xz_{\scaleto{+}{3pt},\varepsilon}=\! e^{-M d}
(-\ln \tilde d_{K})(d \! +\tilde d^2_K) \tilde d_{K}^{-\frac{k}{2}}
 \! - \! d\tilde d_{K}^{-\frac{k}{2}+\xe}
&\!\! {\xz_{\scaleto{-}{3pt},\varepsilon}}=\! e^{Md}
(-\ln \tilde d_{K})
(d+\tilde d^2_K)\tilde d_{K}^{-\frac{k}{2}}\! + \!
d\tilde d_{K}^{-\frac{k}{2}+\xe}
\end{array}
\]

\begin{lemma}
Let $\mu\leq k^2/4$ and $0<\varepsilon<1$.
There exist positive constants $\xb_0=\xb_0(\xO,K,\xm,\xe)$
and $ M=M(\xO,K,\mu,\xe)$ such that the following hold in $K_{\xb_0}\cap \xO$: \newline
$\ia$ The functions $\eta_{\xg_+,\varepsilon}$ and
$ \zeta_{\xgp,\varepsilon}$ are non-negative
in $K_{\xb_0}\cap \xO$ and satisfy
\[
	L_\xm  \eta_{\xg_+,\varepsilon}\geq 0 ,
\qquad\quad  L_\xm  \zeta_{\xgp,\varepsilon}\leq 0 , \quad
\mbox{ in }K_{\xb_0}\cap \xO.
\]
$\ib$ If  $\mu <k^2/4$ and $\xe <\min\{1,\sqrt{k^2-4\mu}\})$ then
$	\eta_{\xg_-,\varepsilon}$ and $\zeta_{\xg_-,\varepsilon}$
are non-negative in $K_{\xb_0}\cap \xO$ and satisfy
\be
L_\xm  \eta_{\xg_-,\varepsilon}\geq 0 , \qquad\quad
 L_\xm  \zeta_{\xg_-,\varepsilon}\leq 0 , \quad \mbox{ in }
 K_{\xb_0}\cap \xO .
 \label{supsol2}
\ee	
$\ic$ The functions ${\xz_{\scaleto{+}{3pt},\varepsilon}}$
and	${\xz_{\scaleto{-}{3pt},\varepsilon}}$ are non-negative
in $K_{\xb_0}\cap \xO$ and satisfy
\[
L_{ \frac{k^2}{4}} {\xz_{\scaleto{+}{3pt},\varepsilon}} \geq 0 ,
\qquad\quad
L_{\frac{k^2}{4}}{\xz_{\scaleto{-}{3pt},\varepsilon}}\leq 0 ,
\quad \mbox{ in }K_{\xb_0}\cap \xO.
\]
\label{subsup}
\end{lemma}
\begin{proof} Let $M\in\mathbb{R}$.
By Proposition \ref{propdK} we have in $\xO\cap K_{\xb_0} $,
\begin{align*}\nonumber
\xD(d^a \tilde{d}_{K}^{b})
&= d^{a-2} \tilde d_K^{b}\Big(  a(a-1) +ad \xD   d \Big) \\
& \quad +
d^a \,\tilde d_K^{b-2}
\Big(2ab+b(k-1+f)+ b(b-1)(1+h)\Big) \\
\nabla e^{Md}\cdot \nabla(d^a \tilde{d}_{K}^{b})
&= M e^{Md}(ad^{a-1} \tilde d_K^b+bd^{a+1}\tilde d_K^{b-2})\\
\xD e^{Md}&=e^{Md}(M^2+M\xD d)
\end{align*}
Thus
\begin{align*}
&L_\xm( e^{Md}d^a \tilde{d}_{K}^{b})=-e^{Md}d^{a-1}\tilde d_K^b
\left(M^2 d+Md\xD d+2aM+a\xD d +a(a-1)d^{-1} \right)\\
&- e^{Md}d^a \tilde d_K^{b-1}\Big(\frac{2Mbd+b f+b(b-1)h+\xm g}{\tilde d_K}\Big) - \big(b(k-1)+b(b-1)+2ab+\xm\big)
e^{Md}d^a \tilde d_K^{b-2}
\end{align*}
Now let $M\in \mathbb{R}$ and $0<\xe<1$. Using the above formulas
we find
\bal
\hspace{-.8cm}L_\xm&(e^{Md} (d+\tilde d^2_K)\tilde{d}_{K}^{\xgp})
-L_\xm( d\tilde{d}_{K}^{\xgp+\xe})\\
&=-e^{Md}\tilde d_K^{\xgp}\Big((M^2 d+Md\xD d+2M+\xD d)
 +(M^2+M\xD d)\tilde d_K^2  \Big) \\
&\quad - e^{Md} d \tilde d_K^{\xgp -2}\Big(
 2M\xgp d +\xgp f + \xgp(\xgp-1)h +\mu g  \Big) \\
&\quad -e^{Md}  \tilde d_K^{\xgp} \Big( 2 (\xgp+k) +( \xgp +2)
\big( (\xgp+1)h  +f +2M  d\big) \Big)\\
&\quad + \epsilon(2\xgp +k+\epsilon)d \tilde{d}_K^{\xgp +\epsilon-2} \\
&\quad + d \tilde{d}_K^{\xgp +\epsilon-2}
\Big(  (\xgp+\epsilon)(\xgp+\epsilon -1)h +(\xgp+\epsilon)f
+\mu g  \Big) \\
&\quad +  (\Delta d) \tilde{d}_K^{\xgp +\epsilon}.
\eal
The RHS in the last equality consists of six terms.
We now choose $\xb_0$ small enough and $M<0$ so that
the sum of the first, third and sixth terms is non-negative in
$K_{\xb_0}\cap \xO$.
The fourth  term is clearly positive, and by taking $\beta_0$ smaller if
necessary it may also control the second and the fifth terms.
Hence $L_\xm  \eta_{\xg_+,\varepsilon}\geq 0$ in $K_{\xb_0}\cap \xO$.

The proofs of the other cases of the lemma are similar and are omitted. For (iii) we also use the relations
\begin{align*}\nonumber
\xD\ln \tilde d_K&=\frac{\xD \tilde d_K}{\tilde d_K}-\frac{|\nabla\tilde d_K|^2}{\tilde d_K^2}\\
\nabla \ln \tilde d_K \cdot  \nabla(e^{Md}d\tilde{d}_{K}^{b})&=\tilde d_K^{b-2}e^{Md}\Big(Md^2+d+b\, d|\nabla \tilde d_K|^2\Big)
\end{align*}
and
\begin{align*}
&-L_\xm \big( (-\ln\tilde d_K)e^{Md}d\tilde{d}_{K}^{b}\big)
=(-\ln\tilde d_K)e^{Md}\tilde d_K^b\left(M^2 d+Md\xD d+2M+\xD d\right)\\
&\qquad +(-\ln\tilde d_K) e^{Md}d\tilde d_K^{b-1}\Big(\frac{2Mbd+b f+b(b-1)h+\xm g}{\tilde d_K}\Big) \\
&\qquad +(-\ln\tilde d_K)\big( \, b(k+1)+b(b-1)+\xm \, \big)
e^{Md}d\tilde d_K^{b-2}\\
&\qquad +e^{Md}d\tilde{d}_{K}^{b-2}
\Big( -2Md -f+(1-2b)h -2b-k \Big).
\end{align*}
\end{proof}

\begin{lemma}
Let $\xb_0>0$ be the constant in Lemma \ref{subsup}, $\xi\in\partial\xO$ and $0<r<\frac{\xb_0}{16}.$ We assume that $u\in H^1_{loc}(B_{r}(\xi)\cap\xO)\cap C(B_r(\xi)\cap\xO)$ is $L_\xm$-harmonic in $B_{r}(\xi)\cap\xO$ and
\bel{bdrcond} \lim_{\dist(x,F)\to 0}
\frac{u(x)}{\tilde W(x)}=0,\quad\forall\; \text{compact} \; F\subset B_r(\xi)\cap\partial\xO.
\ee
Then there exists  $C=C(u,\xO,K,r)>0$ such that
\bel{regu2}
|u| \leq C\xf_\xm \, , \qquad x\in B_{\frac{r}{4}}(\xi) \cap \Omega \, .
\ee
Moreover, if
$0\leq\eta_r\leq1$ is a smooth function with compact support in
$B_{\frac{r}{2}}(\xi)$ with $\eta_r=1$ on $B_{\frac{r}{4}}(\xi)$, then
\bel{regu1}
\frac{\eta_ru}{\xf_\xm}\in H^1_0(\xO ;\xf_\xm^2).
\ee
Furthermore, if $u$ is nonnegative there exists $\conl{charnack}=\conr{charnack}(\xO,K)>0$ such that
\be
\frac{u(x)}{\xf_\xm(x)}\leq \conr{charnack}\frac{u(y)}{\xf_\xm(y)},\quad\forall x,y\in B_{\frac{r}{16}}(\xi)\cap\xO.
\label{harnack}
\ee
\label{lemharnack}
\end{lemma}
\begin{proof}
We will only consider the case  $\xm<k^2/4$ and $\xi\in K_\frac{\xb}{16}\cap\partial\xO$; the proof of the other cases is very similar and we omit it.

Since $u$ is $L_\xm$- harmonic in $B_r(\xi)\cap\xO$, by standard elliptic
estimates we have that $u\in C^2(B_r(\xi)\cap\xO)$.
Set $w_l= \max\{u-l\eta_{\xgn,\varepsilon},0\}$ where $l>0$ and
$\eta_{\xgn,\varepsilon}$ is the supersolution in \eqref{supsol2}.
Then by Kato's formula we have
$$
L_\mu w_l \leq 0  \, , \qquad  \text{in } B_r(\xi)\cap\xO.
$$
Setting $v_l=\frac{w_l}{\phi_\mu}$, by straightforward calculations we have
\bel{subsolution}
-\div(\phi_\mu^2\nabla v_l)+\lambda_\mu \phi_\mu^2v_l\leq0
\, ,  \qquad\text{in } B_r(\xi)\cap\xO.
\ee
We note here that $v_l=0$ if $u\leq l\eta_{\ap,\varepsilon}$, thus by the assumptions we can easily obtain that $v_l\in H^1(B_\frac{r}{2}(\xi) ; \phi_\mu^2 )$.

By Theorem \ref{th1*}, we can prove the existence of a constant $r_{\xb_0}$ and $C=C(K)>0$ such that for any $r'\leq\min\{\frac{r}{2},r_{\xb_0}\}$ and $p\geq1$ the following inequality holds
\bel{moser}
\sup_{x\in B_{\frac{r'}{2}}(\xi)\cap \xO} v_l\leq C
\bigg(\Big(\int_{B_{r'}(\xi) \cap \xO}\phi_\mu^2dx\Big)^{-1}
\int_{B_{r'}(\xi) \cap \xO} |v_l|^p\phi_\mu^2dx\bigg)^\frac{1}{p}.
\ee
From \eqref{bdrcond} and the definition of $w_l$, we have
\[
w_l \leq u_+ \leq C\tilde W =C (d+\tilde d^2_K)\tilde d_{K}^\xgn , \quad \text{in } B_{\frac{r}{2}}(\xi) \cap\xO.
\]
This and \eqref{eigenest} imply that
\bal
\int_{B_{r'}(\xi) \cap \xO} |v_l|\phi_\mu^2dx&= \int_{B_{\frac{r}{2}}(\xi) \cap\xO} |w_l|\phi_\mu dx\\
&\leq C \int_{B_{\frac{r}{2}}(\xi) \cap \xO} (d+\tilde d^2_K)d\tilde d_{K}^{-k} dx \leq C  \int_{B_{\frac{r}{2}}(\xi) \cap \xO} d_K^{2-k}dx<\infty.
\eal
Thus by \eqref{moser} and the last inequality we deduce that
$$
\sup_{B_{\frac{r'}{2}}(\xi) \cap K}v_l< C_1
$$
for some constant $C_1>0$ which does not depend on $l$. Thus
$$
w_l \leq C_1\phi_\mu \; , \qquad \text{in } B_{\frac{r'}{2}}(\xi) \cap\xO.
$$
By letting $l \to 0$, we derive
$$
u_+ \leq C_1\phi_\mu \; , \qquad \text{in } B_{\frac{r'}{2}}(\xi) \cap\xO.
$$
Thus by a covering argument we can find a constant $C_2>0$ such that
\bel{fragma}
u_+\leq C_2\phi_\mu \; , \qquad \text{in } B_{\frac{r}{2}}(\xi) \cap\xO.
\ee
This implies $v_0:=\frac{u_+}{\phi_\mu} <C_2$ in $B_{\frac{r}{2}}(\xi) \cap\xO$.

Using $\eta^2_r v_l$ as a test function in \eqref{subsolution} we can easily obtain
 $$\int_{B_\frac{r}{2}(\xi) \cap\xO}|\nabla (\eta_r v_l)|^2\xf_\xm^2dx+\xl_\xm \int_{B_\frac{r}{2}(\xi) \cap\xO}|\eta_r v_l|^2\xf_\xm^2dx\leq
 \frac{C}{r^2}\int_{B_\frac{r}{2}(\xi)\cap\xO}|v_l|^2\xf_\xm^2dx.$$
By \eqref{fragma} and by letting  $l \to0$ we obtain that $\eta_r v_0 \in H^1(\Omega ;\phi_\mu^2)$,  which in turn implies that $ \frac{\eta_r u_+}{\phi_\mu}\in H^1(\xO ;\xf_\xm^2)$.  Applying the same argument
to $-u$ we obtain
\[
u_-\leq C_2\phi_\mu \quad \text{in } B_{\frac{r}{2}}(\xi) \cap\xO,
\]
and
$ \frac{\eta_r u_-}{\phi_\mu}\in H^1(\Omega ; \phi_\mu^2)$. By using the fact that $u=u_+-u_-$, we obtain \eqref{regu1} and \eqref{regu2}.

We next prove the boundary Harnack inequality  \eqref{harnack}. Let $u$
be a nonnegative $L_\mu$-harmonic function and put
$v=\frac{u}{\phi_\mu}$. Then
$v \in H^1(B_{\frac{r}{4}}(\xi); \phi_\mu^2)$ and
$v$ satisfies
$$ -\phi_\mu^{-2}\div(\phi_\mu^2 \nabla v) + \lambda_\mu v = 0 ,
\qquad \mbox{ in }
 B_\frac{r}{4}(\xi) \cap \Omega.
$$
The function
$\hat v(x,t):=e^{\lambda_\mu t}  v(x)$ then satisfies
$$ \partial_t \hat v - \phi_\mu^{-2}\div(\phi_\mu^2 \nabla \hat v) = 0, \quad \text{in } B_{\frac{r}{4}}(\xi) \cap \Omega \times (0,\frac{r^2}{16}).
$$
By the  Harnack inequality \eqref{harnackpar},
\begin{align*}  \textrm{ess sup} &\Big\{ \hat v(t,x): (t,x) \in
(\frac{r^2}{64},\frac{r^2}{32}) \times
 \CB(\xi, \frac{r}{8}) \cap \Omega   \Big\} \\
&\leq C\,  \textrm{ess inf} \Big\{ \hat v(t,x): (t,x) \in
(\frac{3r^2}{64},\frac{r^2}{16}) \times \CB(\xi, \frac{r}{8} )
\cap \Omega    \Big\}.
\end{align*}
This implies \eqref{harnack}.
\end{proof}

\begin{lemma} \label{comparison}
Let $\mu\leq k^2/4$ and assume that $\lambda_{\mu}>0$.
Let $u\in H^1_{loc}(\xO)\cap C(\xO)$ be $L_\xm$-subharmonic in $\xO$. Assume that
\be
 \limsup_{\dist(x,F)\to 0}\frac{u(x)}{\tilde W(x)}\leq0,\quad\forall\; \text{compact} \; F\subset \partial\xO.\label{boundary0}
\ee
Then $u\leq0$ in $\xO.$
\end{lemma}
\begin{proof}
First we note that $u_+=\max(u(x),0)$ is a nonnegative $L_\mu$-subharmonic function in $\Omega$. Let $v=\frac{u_+}{\xf_\xm}$. In view of the proof of \eqref{regu1}, $v \in H^1_0(\xO; \xf_\xm^2 )$; moreover
by a straightforward calculation we have
\bel{su0}
-\div(\xf_\xm^2\nabla v)+\lambda_\mu \phi_\mu^2 v \leq 0 \quad \text{in } \Omega.
\ee
Since $v \in H^1_0(\Omega ; \phi_\mu^2 )$, we can use it as a test function for \eqref{su0} and obtain
$$\int_{\Omega}|\nabla v|^2\phi_\mu^2dx+\lambda_\mu \int_{\Omega}|v|^2\phi_\mu^2dx\leq0.$$
Hence $v=0$ and the result follows.
\end{proof}

\subsection{Existence and uniqueness}
The aim of this subsection is to prove existence and uniqueness of the
solution of $L_\xm u=f,$ with smooth boundary data. We also prove the
boundary Harnack inequalities and maximum principle for the operator
$L_\xm.$ Let us first define the notion of a weak solution.

\begin{definition}
Let $f\in L^2(\xO)$. We say that  $u$ is a weak solution of
\bel{LE1} L_\mu u=f\, , \quad \text{ in }  \Omega
\ee
if  $\frac{u}{\phi_\mu}\in H^1_0(\Omega ; \phi_\mu^2)$ and
\[
\int_\xO\nabla u \cdot \nabla\psi \, dx
-\xm \int_\xO\frac{u\psi}{d^2_K} dx=\int_\xO f\psi \, dx, \qquad\forall
 \psi\in C_c^\infty(\xO).
\]
\end{definition}
In the next lemma we give the first existence and uniqueness result.

\begin{lemma} \label{existence1}
Let $\mu\leq k^2/4$ and assume that $\lambda_{\mu}>0$.
For any $f\in L^2(\xO)$ there exists a unique weak solution $u$ of \eqref{LE1}. Furthermore there holds
\be
\int_\xO u^2dx\leq C\int_\xO f^2dx ,
\label{weaksolutionest}
\ee
where $C=C(\xl_\xm)>0.$
\end{lemma}
\begin{proof} We first observe that $u$ is a weak solution of \eqref{LE1} if and only if $v=\frac{u}{\phi_\mu}$ satisfies
\bel{lem}
\int_\Omega \phi_\mu^2\nabla v \cdot
\nabla\zeta dx+\lambda_\mu \int_\Omega \phi_\mu^2 v\zeta dx=\int_\Omega \phi_\mu  f\zeta dx \; , \qquad \forall \zeta \in  H^1_0(\xO ; \phi_\mu^2).
\ee
We define on $H_0^1(\Omega ; \phi_\mu^2 )$ the inner product
\[
\langle \psi, \zeta \rangle_{\phi_\mu^2} = \int_{\Omega} \phi_\mu^2(\nabla \psi \cdot \nabla \zeta + \lambda_\mu  \psi \, \zeta)dx
\]
and consider the bounded linear functional
$T_f$ on $H_0^1(\Omega ; \phi_\mu^2)$ given by
\[
T_f(\zeta)=\int_{\Omega} \phi_\mu f \zeta dx.
\]
Then \eqref{lem} becomes
\bel{lem1ex}
\langle v,\zeta \rangle_{\phi_\mu^2} =
T_f(\zeta) \quad \forall \zeta \in H_0^1(\Omega ;\phi_\mu^2).
\ee
By Riesz representation  theorem there exists a unique function
$v \in H_0^1(\Omega ; \phi_\mu^2)$ satisfying \eqref{lem1ex}.
 Furthermore, by choosing $\zeta=v$ in \eqref{lem} and then using
 Young's inequality, we obtain
\bel{lem2ex} \int_\xO\xf_\xm^2|\nabla v|^2 dx+\frac{\xl_\xm}{2}\int_\xO \phi_\mu^2 v^2 dx\leq C(\xl_\xm)\int_\xO f^2 dx.
\ee
By putting $u=\phi_\mu v$, we deduce that $u$ is the unique weak solution of \eqref{LE1}. Moreover, \eqref{weaksolutionest} follows from \eqref{lem2ex}.
\end{proof}

The next lemma will be useful in order to prove existence and uniqueness of solution for the equation $L_\xm u=f$ with zero boundary data.
\begin{lemma}\cite[Lemma 5.3]{gktai}\label{anisotita}
Let $\gamma<N$ and $\xa \in (0, \min\{ k,\gamma\})$.
There exists a positive constant $C=C(\xa,\xg,\xO, K)$ such that
\[
\sup_{x\in\xO}\int_\xO|x-y|^{-N+\xg}d_K^{-\xa}(y)dy< C.
\]
\end{lemma}

In the following lemma we prove the existence of solution for the equation  $L_\xm u=f$ with zero boundary data, as well as
pointwise estimates.
\begin{lemma} \label{existence2}
Let $\mu\leq k^2/4$ and assume that $\lambda_{\mu}>0$,
$\gamma_{-}-1 <b<0$ and $f\in L^\infty(\xO)$. Then there exists a unique  $u\in H^1_{loc}(\xO)\cap C(\xO)$ which satisfies
$L_\xm u=fd_K^{b}$ in the sense of distributions as well as
\eqref{boundary0}. Moreover, for any $\xg\in(-\infty, \xgp]\cap(-\infty,b+1)\cap(-\infty,0]$ there exists a positive constant $C=C(\xO,K,b, \xm,\xg)$ such that
\be
|u(x)|\leq C\|f\|_{L^\infty(\xO)}d(x)d_K^{\xg}(x),\qquad x\in \xO.\label{veryweakest}
\ee
\end{lemma}
\begin{proof}
We assume first that $f\geq0.$ Set $f_n=\min\{fd_K^{b},n\}$.
By Lemma \ref{existence1}, there exists a unique solution $u_n$ of $L_\xm v=f_n$ in $\xO$.
Furthermore, a standard argument yields the representation formula
$$
u_n(x)=\int_{\xO} G_{\xm}(x,y)f_n(y)dy.
$$
We assume first that $0<\xm<\frac{k^2}{4}.$ By \eqref{Greenest} we have
\bal
0&\leq\int_{\xO} G_{\xm}(x,y)f_n(y)dy\\
&\leq C_1\int_{\xO} \min\left\{\frac{1}{|x-y|^{N-2}},\frac{d(x)d(y)}{|x-y|^N}\right\} \left(\frac{d_K(x)d_K(y)}{\left(d_K(x)+|x-y|\right)\left(d_K(y)+|x-y|\right)}\right)^{\xg_+} f_n(y)dy\\
&\leq C d_K^\xgp(x)\int_{\xO}|x-y|^{-N+2-2\xgp}
\min\Big\{1,\frac{d(x)d(y)}{|x-y|^2}\Big\}d_K^\xgp(y) f_n(y)dy\\
&\quad + C \int_{\xO}|x-y|^{-N+2-\xgp}
\min\Big\{1,\frac{d(x)d(y)}{|x-y|^2}\Big\}d_K^\xgp(y) f_n(y)dy\\
&\quad + C  d_K^\xgp(x)\int_{\xO}|x-y|^{-N+2-\xgp}
\min\Big\{1,\frac{d(x)d(y)}{|x-y|^2}\Big\} f_n(y)dy\\
&\quad + C \int_{\xO}|x-y|^{-N+2}
\min\Big\{1,\frac{d(x)d(y)}{|x-y|^2}\Big\} f_n(y)dy\\
&=C (I_1+I_2+I_3+I_4).
 \eal
First we note that if $d_K(y)\leq \frac{1}{4}d_K(x)$ then $|x-y|\geq\frac{3}{4}d_K(x).$ Thus for $\gamma \leq \xgp$, we have
\bal
I_1&=d_K^\xgp(x)\int_{\xO\cap\{d_K(y)\leq \frac{1}{4}d_K(x)\}}|x-y|^{-N+2-2\xgp}\min\Big\{1,\frac{d(x)d(y)}{|x-y|^2}\Big\}d_K^\xgp(y) f_n(y)dy\\
&\quad +d_K^\xgp(x) \int_{\xO\cap\{d_K(y)> \frac{1}{4}d_K(x)\}} |x-y|^{-N+2-2\xgp}\min\Big\{1,\frac{d(x)d(y)}{|x-y|^2}\Big\}d_K^\xgp(y) f_n(y)dy\\
&\leq C \|f\|_{L^\infty(\xO)}d_K^{\xg}(x)\int_{\xO\cap\{d_K(y)\leq \frac{1}{4}d_K(x)\}}|x-y|^{-N+2-\xg-\xgp}\min\Big\{1,\frac{d(x)d(y)}{|x-y|^2}\Big\}d_K^{b+\xgp}(y)dy\\
&\quad +C \|f\|_{L^\infty(\xO)}d_K^{\xg}(x)\int_{\xO\cap\{d_K(y)> \frac{1}{4}d_K(x)\}} |x-y|^{-N+2-2\xgp}\min\Big\{1,\frac{d(x)d(y)}{|x-y|^2}\Big\} d_K^{b-\xg+2\xgp}(y)dy\\
&\leq C \|f\|_{L^\infty(\xO)}d_K^{\xg}(x)d(x)\int_{\xO\cap\{d_K(y)\leq \frac{1}{4}d_K(x)\}}|x-y|^{-N-\xg-\xgp}d_K^{b+\xgp+1}(y)dy\\
&\quad +C \|f\|_{L^\infty(\xO)}d_K^{\xg}(x)d(x)\int_{\xO\cap\{d_K(y)> \frac{1}{4}d_K(x)\}} |x-y|^{-N-2\xgp} d_K^{b-\xg+2\xgp+1}(y)dy\\
&\leq C \|f\|_{L^\infty(\xO)}d_K^{\xg}(x)d(x)
\eal
where in the last inequalities we have used Lemma \ref{anisotita}.

Similarly we can prove that
 \begin{align*}
 I_1+I_2+I_3+I_4\leq C\|f\|_{L^\infty(\xO)}d_K^{\xg}(x)d(x)
 \end{align*}

Combining the above estimates, we deduce that for any $\xg\in(-\infty, \xgp]\cap(-\infty,b+1),$ there exists a positive constant
$C=C(\xO,K,\xm ,b,\xg)$ such that
\be
|u_n(x)|\leq C\|f\|_{L^\infty(\xO)}d(x)d_K^{\xg}(x),\qquad  x\in \xO.\label{veryweakest2a}
\ee
If we choose $\xg\in (\xgn,\xgp]\cap(\xgn,b+1),$ then we can show that
\be
\lim_{\dist(x,F)\to 0}\frac{d(x)d_K^{\xg}(x)}{\tilde W(x)}=0,\quad\forall\; \text{compact} \; F\subset \partial\xO.
\label{convestim}
\ee

Thus by the above inequality, \eqref{veryweakest2a} and applying
Lemma \ref{comparison}, we can easily show that $u_n\nearrow u$
locally uniformly in $\xO$ and in $H^1_{loc}(\xO).$ Furthermore, by
standard elliptic theory $u\in C^1(\xO)$ and, by
\eqref{veryweakest2a},
\be
|u(x)|\leq C\|f\|_{L^\infty(\xO)}d(x)d_K^{\xg}(x),\qquad  x\in \xO.\label{veryweakest3}
\ee
The uniqueness follows by \eqref{veryweakest3}, \eqref{convestim} and Lemma \ref{comparison}.

For the general case, we set $u=u_+-u_-$ where $u_{\pm}$ are the unique solutions of $L_\xm v=f_{\pm}d_K^{-b}$ in $\Omega \setminus K$ respectively, which satisfy \eqref{veryweakest}. Thus $u$ satisfies \eqref{veryweakest} and the result follows in the case $0<\xm<\frac{k^2}{4}$.

The proof in the cases $\mu=\frac{k^2}{4}$ and $\xm\leq0$ is similar and is omitted.
\end{proof}

The following lemma is the main result of this subsection.
\begin{lemma} \label{mainlemma1}
Let $\mu\leq k^2/4$ and assume that $\lambda_{\mu}>0$.
For any $h \in C(\partial \Omega)$ there exists a unique $L_{\xm}$-harmonic
function $u \in H^1_{loc}(\xO)\cap C(\xO)$ satisfying
\[
\lim_{x\in\Omega,\;x\rightarrow y\in\partial\xO}\frac{u(x)}{\tilde W(x)}=h(y)\qquad\text{uniformly in } y\in\partial\xO.
\]
Furthermore  there exists a constant $c=c(\xO,K)>0$
\[
\left|\!\left|\frac{u}{\tilde W}\right|\!\right|_{L^\infty(\Omega  )}\leq c\|h\|_{C( \partial\Omega)}.
\]
\end{lemma}
\begin{proof}
\noindent Uniqueness is a consequence of Lemma \ref{comparison}.

\emph{Existence.} We will only consider  the case $0<\xm<\frac{k^2}{4},$ the proof in the other cases is very similar. First we assume that $h\in C^2(\overline{\xO})$. Then a function $u\in  C^2(\xO)$ is
$L_\mu$-harmonic if and only if $v:=\tilde W h  -u$ is a solution of
\be
\label{10}
L_\xm v=L_\xm (\tilde Wh)=h(L_\xm\tilde W) -2\nabla \tilde W\cdot \nabla h-\tilde W\xD h , \qquad \text{in } \Omega \, ;
\ee
Arguing as in the proof of Lemma \ref{subsup} we see that
there exists $C=C(\Omega, K, \mu,\beta_0)$ such that
\[
|L_\xm \tilde W|\leq Cd_K^{\xgn} , \qquad \text{in } \Omega.
\]
Hence  \eqref{10} can be written as
\[
L_\mu v = fd_K^{\xgn} , \qquad \text{in } \Omega ,
\]
with $\norm{f}_{L^\infty(\xO)}\leq C(\xgn,\xO,K)\norm{h}_{C^2(\overline{\xO})}.$

By Lemma \ref{existence2} there exists a unique solution $v$ of \eqref{10} that satisfies
\[
|v(x)|\leq C\|h\|_{C^2(\overline{\xO})}d(x)d_K^{\xg}(x),\qquad x\in \xO,
\]
for any $\gamma \in (\xgn,\xgp]\cap(\xgn,\xgn+1)$. Thus
\be
\label{tofragma}
\left|\frac{u(x)}{\tilde W(x)}-h(x)\right|\leq C\|h\|_{C^2(\overline{\xO})}\frac{d(x)d_K^{\xg}(x)}{\tilde W(x)} , \qquad  x\in \xO,
\ee
and the desired result follows in this case, since
\[
 \lim_{\dist(x,F)\to 0}\frac{d(x)d_K^{\xg}(x)}{\tilde W(x)}=0,\qquad\forall\; \text{compact} \; F\subset \partial\xO.
\]
for any $\xg\in (\xgn,\xgp]\cap(\xgn,\xgn+1).$

Suppose now that
$h\in C(\partial\xO)$. We can then find a sequence
$\{h_n\}_{n=1}^\infty$ of smooth functions in $\partial\xO$ such that $h_n\rightarrow h$ in $L^\infty(\partial\xO).$ Then there exist $H_n\in C^2(\overline{\xO})$
with value $H_n|_{\partial\xO}=h_n$
and $\|H_n\|_{L^\infty(\overline{\xO})}\leq C\|h_n\|_{L^\infty(\partial\xO)}$ where $C$ does not depend on $n$ or $h_n$.
By the previous case there exists a unique weak solution $u_n$ of $L_{\xm}u=0$ satisfying
\begin{align*}
\left|\frac{u_n(x)}{\tilde W(x)}-H_n(x)\right|\leq C\|H_n\|_{C^2(\overline{\xO})}\frac{d(x)d_K^{\xg}(x)}{\tilde W(x)},
\qquad\forall x\in \xO ,
\end{align*}
for some $C$ which does not depend on $n$ and $h_n$.

By \eqref{tofragma} and Lemma \ref{comparison}, we can easily show that
$$
\left| \frac{u_n(x)-u_m(x)}{\tilde W(x)}\right|\leq C\|h_n-h_m\|_{L^\infty(\partial\xO)}, \qquad  x \in \Omega ;
$$
thus $u_n\rightarrow u$ locally uniformly in $\xO$.

Now, let $y\in \partial\xO.$ Then
\[
\left|\frac{u(x)}{\tilde W(x)}-h(y)\right|
\leq \left|\frac{u(x)-u_n(x)}{\tilde W(x)}\right|+\left|\frac{u_n(x)}{\tilde W(x)}-h_n(y)\right|+\left|h_n(y)-h(y)\right|
\]
and the result follows by letting successively $x\to y$ and $n\to\infty$.
\end{proof}

\section{Martin kernel}
\subsection{$L_\mu$-harmonic measure}

Let $x_0\in\xO,$ $h\in C(\partial\xO)$ and denote $L_{\xm ,x_0}(h):=v_h(x_0)$ where $v_h$ is the solution of the Dirichlet problem (see Lemma \ref{mainlemma1})
\[
\left\{ \BAL
L_{\xm }v&=0,\qquad\mathrm{in}\;\;\xO, \\
\trti(v)&=h, \qquad\mathrm{in}\;\;\partial\xO,
\EAL \right.
\]
where $\trti(v)=h$ is understood in the sense of
Lemma \ref{mainlemma1} (cf. also \eqref{bdrcond3}).
By Lemma \ref{comparison}, the mapping
$h\mapsto L_{\xm ,x_0}(h)$ is a
positive linear functional on $C(\partial\xO).$ Thus there exists a unique
Borel measure on $\partial\Omega$, called {\it $L_{\xm }$-harmonic measure} in $\xO,$ denoted by $\xo^{x_0}$, such that
$$v_{h}(x_0)=\int_{\partial\xO}h(y) d\xo^{x_0}(y).$$
Thanks to the Harnack inequality the measures $\xo^x$ and $\xo^{x_0}$,
$x_0,\,x\in \xO$, are mutually absolutely continuous. For every fixed $x$ we denote the Radon-Nikodyn derivative by
\bel{Kmu}
K_{\xm}(x,y):=\frac{dw^x}{dw^{x_0}}(y),\qquad\mathrm{for}\;\xo^{x_0}\text{- almost all }y\in\partial\xO.
\ee

Let $\xi\in\partial\xO$. We set $\xD_r(\xi)=\partial\xO\cap B_r(\xi)$
and denote by $x_r=x_r(\xi)$ the point in $\Omega$ determined by
$d(x_r)=|x_r-\xi|=r$.
We recall here that $\xb_0=\xb_0(\xO,K,\xm)>0$ is small enough and has been defined in Lemma \ref{subsup}.

\begin{lemma}\label{Lemm10}
Let $\mu\leq k^2/4$ and assume that $\lambda_{\mu}>0$.
Let $0<r\leq \xb_0$. We assume that $u$ is a positive $L_{\mu }$-harmonic function in $\Omega$ such that\smallskip
\begin{eqnarray*}
& \ia &  \frac{u}{\tilde W}\in C(\overline{\xO\setminus B_r(\xi)}),\\
& \ib & \lim_{x \in \Omega , x \to x_0}\frac{u(x)}{\tilde W(x)}=0,\quad\forall x_0\in\partial\xO\setminus \overline{B_{r}(\xi)},
\text{uniformly with respect to $x_0$.}
\end{eqnarray*}
\noindent Then
\ba\BAL
c^{-1}\frac{u(x_r(\xi))}{G_\xm(x_r(\xi),x_{\frac{r}{16}}(\xi))}&G_\xm(x,x_{\frac{r}{16}}(\xi))\leq u(x)\\
&\leq c\frac{u(x_r(\xi))}{G_\xm(x_r(\xi),x_{\frac{r}{16}}(\xi))}G_\xm(x,x_{\frac{r}{16}}(\xi)), \qquad
\forall x\in\Omega\setminus\overline{B_{2r}(\xi)},
\EAL
\label{ekt2}
\ea
with $c>1$ depending only $\Omega$, $K$ and $\mu$.
\end{lemma}  %%%%%%%%%%%%%%%%%%%%%%%%%%%%%%%%%%%%%%%%%%%%%%%%%%%%%%%%%%%%%%PROOF%%%%%%%%%%%%%%%%%%%%%%%%%%%%%%%%%%%%%%%%%%%%%%%%%%%%%%%%%%%%%%%%%%%%%%%%%%%%%%%%%%%%%%%%%%%%%%%%%
\begin{proof} It follows from Lemma \ref{lemharnack} that there exists
$c>1$ such that
\bal\BAL
c^{-1}\frac{u(x_{2r}(\xi))}{G_\xm(x_{2r}(\xi),x_{\frac{r}{16}}(\xi))}&G_\xm(x,x_{\frac{r}{16}}(\xi))\leq u(x)\\
&\leq c\frac{u(x_{2r}(\xi))}{G_\xm(x_{2r}(\xi),x_{\frac{r}{16}}(\xi))}G_\xm(x,x_{\frac{r}{16}}(\xi)),
\qquad \forall x\in\Omega\cap\partial B_{2r}(\xi),
\EAL
\eal
Applying Harnack inequality between $x_{2r}(\xi)$ and $x_{r}(\xi)$ we obtain
\bal\BAL
c^{-1}\frac{u(x_r(\xi))}{G_\xm(x_r(\xi),x_{\frac{r}{16}}(\xi))}&G_\xm(x,x_{\frac{r}{16}}(\xi))\leq u(x)\\
&\leq c\frac{u(x_r(\xi))}{G_\xm(x_r(\xi),x_{\frac{r}{16}}(\xi))}G_\xm(x,x_{\frac{r}{16}}(\xi)),
\qquad \forall x\in\Omega\cap\partial B_{2r}(\xi).
\EAL
\eal
For $\xe>0$ let
\[
u_\xe(x)=u(x)-c\frac{u(x_{r}(\xi))}{G_\xm(x_{r}(\xi),x_{\frac{r}{16}}(\xi))}G_\xm(x,x_{\frac{r}{16}}(\xi))-\xe v_1(x),
\]
where $c$ is as above.
Then $u_\xe$ is $L_{\mu }$-harmonic and the function $u_\xe^+=\max(u_\xe,0)$ has compact support in $\Omega\setminus\overline{B_{2r}(\xi)}$.
Set $v_\xe=\frac{u_\xe}{\ei}$ and $v_\xe^+=\frac{u_\xe^+}{\ei}$. Using
$u_\xe^+$ as a test function we obtain
\[
\int_{\Omega\setminus\overline{B_{2r}(\xi)}} \nabla v_\xe
\cdot \nabla v_\xe^+\ei^2dx+\xl_\xm
\int_{\Omega\setminus\overline{B_{2r}(\xi)}} v_\xe v_\xe^+\ei^2dx=0.
\]
Letting $\xe\to0$ in the above equation we get
$$
\xl_\xm \int_\xO |v^+|^2\ei^2dx\leq 0,
$$
hence $u(x)-c\frac{u(x_{r}(\xi))}{G_\xm(x_{r}(\xi),x_{\frac{r}{16}}(\xi))}G_\xm(x,x_{\frac{r}{16}}(\xi))\leq 0$ for all
$x\in\Omega\setminus\overline{B_{2r}(\xi)}$.
The proof of the lower estimate in \eqref{ekt2} is similar and we omit it.
\end{proof}

\subsection{The Poisson kernel of $L_{\xm }$}
In this section we establish some properties of the Poisson kernel associated to $L_{\xm }$.

\begin{definition}
A function $\mathcal{K}$ defined in $\xO$ is called a kernel function for $L_\mu$ with pole at $\xi\in\partial\xO$ and basis at $x_0\in\xO$ if
\begin{eqnarray*}
& {\rm (i)} & \mbox{$\mathcal{K}(\cdot,\xi)$ is $L_{\xm }$-harmonic in $\Omega$,} \\
& {\rm (ii)} & \mbox{
 $ \frac{\mathcal{K}(\cdot,\xi)}{\tilde W(\cdot)}\in C(\overline{\xO}\setminus\{\xi\})$ and for any $\eta \in\prt\Omega\setminus\{\xi\}$
 we have }
\lim_{x \in \Omega,\; x \to \eta}
\frac{\mathcal{K}(x,\xi)}{\tilde W(x)}=0,\\
& {\rm (iii)} & \mbox{$\mathcal{K}(x,\xi)>0$ for each $x\in\xO$ and $\mathcal{K}(x_0,\xi)=1.$}
\end{eqnarray*}

\end{definition}

\begin{proposition}
Assume that $\lambda_{\mu}>0$.
There exists a unique kernel function for $L_{\xm }$
with pole at $\xi$ and basis at $x_0.$
\label{uniq}
\end{proposition}
\begin{proof}
The proof is similar to that of \cite[Theorem 3.1]{caffa} and we
include it for the sake of completeness.

\minsk
\noindent
\textit{Existence.} We shall prove that the function
$K_{\mu}(x,\xi)$ defined by \eqref{Kmu} has the required properties.

Fix $\xi\in\partial\xO$.  Set
\[
u_n(x)=\frac{\omega^x(\xD_{2^{-n}}(\xi))}{\omega^{x_0}(\xD_{2^{-n}}(\xi))},\qquad\forall n\geq n_0.
\]
Clearly $ u_n(x) \to K_{\mu}(x,\xi)$, $x\in\Omega$.
Since $u_n\geq0$, $L_{\xm }u_n=0$ in $\xO$ and $u_n(x_0)=1$ the
 sequence $\{u_n\}$ is locally bounded in $\xO$ by Harnack inequality.
 Hence we can find a subsequence, again denoted by $\{u_n\},$ which
 converges to $K_{\mu}(\cdot,\xi)$ locally uniformly in $\Omega$.

Let $\eta\in \partial\xO\setminus\{\xi\}$ and let $n_1\in\BBN$ be such
 that $\eta\in\partial\xO\setminus
 \overline{B_{2^{-n+1}}(\xi)},\;\forall  n\geq n_1$.
By Lemma \ref{Lemm10} we have
\[
u_n(x)\leq c\frac{u_n(x_{2^{-n_1}}(\xi))}{G_\xm(x_{2^{-n_1}},x_{2^{-n_1-4}}(\xi))}G_\xm(x,x_{2^{-n_1-4}}(\xi)),\qquad\forall x\in \Omega \setminus\overline{B_{2^{-n_1+1}}(\xi)},
\]
which implies
\[
K_{\mu}(x,\xi)
\leq c\frac{u_n(x_{2^{-n_1}}(\xi))}{G_\xm(x_{2^{-n_1}},x_{2^{-n_1-4}}(\xi))}G_\xm(x,x_{2^{-n_1-4}}(\xi)),\qquad\forall x\in \Omega\setminus\overline{B_{2^{-n_1+1}}(\xi)}.
\]
It follows that
$$
\lim_{x \in \Omega ,\; x \to \eta}\frac{K_{\mu}(x,\xi)}{\tilde W(x)} = 0,
$$
hence $K_{\mu}(x,\xi)$ is a kernel function for $L_\mu$ with pole at $\xi$ and basis at $x_0$.

\minsk
\noindent
\textit{Uniqueness.} Assume $f$ and $g$ are two kernel functions
for $L_{\xm}$ in $\xO$ with pole at $\xi$ and basis at $x_0$.
Let $0<r<\xb_0$. By Lemma \ref{Lemm10} and
the properties of $f$ and $g$ there holds
\[
\frac{1}{c'}\frac{f(x_r(\xi))}{g(x_r(\xi))}\leq
\frac{f(x)}{g(x)}\leq c'\frac{f(x_r(\xi))}{g(x_r(\xi))} \; ,
\qquad\forall x\in\xO\setminus\overline{B_{2r}(\xi)}.
\]
In particular we can obtain if we take $x=x_0$
$$\frac{f(x_r(\xi))}{g(x_r(\xi))}\leq  c',$$
and hence
$$
\frac{f(x)}{g(x)}\leq  c'^2=:c \, , \qquad\forall x\in  \xO.
$$
We derive that for  any two kernel functions $f$ and $g$ for $L_{\xm}$ with pole at $\xi$ and basis at $x_0$ there holds
\[
f(x)\leq cg(x)\leq c^2f(x) \, , \qquad\quad  x\in  \xO.
\]
Obviously $c\geq1.$ If $c=1$ the result is proved. If $c>1$ then
we set $A=\frac{1}{c-1} $ and $f+A(f-g)$ is also a kernel function for $L_{\xm}$ with pole at $\xi$ and basis at $x_0$.
Repeating the argument for the functions $f+A(f-g)$ and $g$
 we obtain that
\[
f+A(f-g)+A\big(f-g+A(f-g)\big),
\]
is also a kernel function with pole at $\xi$ and basis at $x_0$.
Proceeding in this manner we conclude that for each positive integer
$k$ there exist nonnegative numbers
$a_{1k},...,a_{kk}$ such that
\[
f+\Big(kA+\sum_{i=1}^ka_{ik}\Big)(f-g)
\]
is a kernel function with pole at $\xi$ and basis at $x_0$.
Hence
\[
f+\Big(kA+\sum_{i=1}^ka_{ik}\Big)(f-g)\leq c f.
\]
This last inequality can hold  for all $k$ only if $f\equiv g$.
\end{proof}
\begin{proposition}
Assume that $\lambda_{\mu}>0$.
For any $x\in\Omega$, the function $\xi\mapsto K_{{\xm }}(x,\xi)$ is continuous on $\prt\xO$.
\end{proposition}
\begin{proof}
The proof is an adaptation of that of \cite[Corollary 3.2]{caffa}.
Suppose that $\{\xi_n\}$ is a sequence converging to $\xi$.
Then the sequence $\{K_{\xm}(\cdot,\xi_n)\}$ of positive solutions of $L_{\mu}u=0$ in $\Omega$ has a subsequence which converges locally uniformly in $\xO$ to a  positive $L_{\mu}$-harmonic function.
Moreover, for any $r>0,$ $\frac{K_{\xm}(x,\xi_n)}{\tilde W(x)}$ converges to zero uniformly in $n$ as $x\rightarrow \eta\in\partial\xO\setminus B_r(\xi)$.
Hence the limit function of the subsequence is the kernel function
$K_{{\xm }}(x,\xi)$.
By the uniqueness of the kernel function we conclude that the convergence
$$
K_{{\xm }}(x,\xi_n)\rightarrow K_{{\xm }}(x,\xi)
$$
holds for the entire sequence $\{\xi_n\}$.
\end{proof}

We can now identify the Martin boundary and topology with their classical
analogues. We begin by recalling the definitions of the Martin boundary and
related concepts.

Let $x_0\in\Omega$ be fixed. For $x,\;y\in \Omega$ we set
$$\mathcal{K}_\xm(x,y):=\frac{G_{L_{\xm }}(x,y)}{G_{L_{\xm }}(x_0,y)}.$$

Consider the family
of sequences $\{y_k\}_{k\geq1}$ of points of $\xO$ without cluster points in
$\xO$ for which $\mathcal{K}_\xm(x,y_k)$ converges in $\xO$ to a $L_{\mu}$-harmonic
function, denoted by $\mathcal{K}_\xm(x,\{y_k\})$. Two
such sequences ${y_k}$ and ${y_k'}$ are called equivalent if $\mathcal{K}_\xm(x,\{y_k\})=\mathcal{K}_\xm(x,\{y_k'\})$
and each equivalence class is called an element of the Martin boundary
$\xG$. If $Y$ is
such an equivalence class (i.e., $Y\in \xG$) then $\mathcal{K}_\xm(x,Y)$ will denote the corresponding
harmonic limit function. Thus each $Y\in\xO\cup\xG$ is associated with a unique
function $\mathcal{K}_\xm(x,Y).$ The Martin topology on $\xO\cup\xG$ is given by the metric
$$
\xr(Y,Y')=\int_{A}\frac{|\mathcal{K}_\xm(x,Y)-\mathcal{K}_\xm(x,Y')|}{1+|\mathcal{K}_\xm(x,Y)-\mathcal{K}_\xm(x,Y')|}dx ,
\qquad Y,Y'\in \xO\cup \xG,
$$
where $A$ is a small enough neighbourhood of $x_0$.
The function $\mathcal{K}_\xm(x,Y)$ is a $\xr$-continuous
function of $Y\in\xO\cup\xG$ for any fixed $x \in\xO$.
Moreover $\xO\cup\xG$ is compact and complete with respect
to $\xr$, $\xO\cup\xG$ is the $\xr$-closure of $\xO$ and the $\xr$-topology
is equivalent to the Euclidean topology in $\xO$.

\begin{proposition} \label{martinest}
Assume that $\lambda_{\mu}>0$.
There is a one-to-one correspondence between the Martin boundary of $\xO$ and the Euclidean boundary $\partial\xO .$ If $Y \in\xG$ corresponds
to $\xi\in \partial\xO$ then $\mathcal{K}_\xm(x,Y)=K_{\xm}(x,\xi).$ The Martin topology on $\Omega  \cup \xG$ is equivalent
to the Euclidean topology on $\Omega \cup \partial\xO.$
\end{proposition}
\begin{proof}
The proof is similar as the one of Theorem 4.2 in \cite{hunt} and we
include it for the sake of completeness.
By uniqueness of the kernel function we have that
$$\mathcal{K}_\xm(x,\{y_k\})=K_{\xm}(x,\xi),$$
where $\{y_k\}$ is a sequence in $\xO$ such that $y_k\rightarrow \xi\in\partial\xO.$ It follows that each
point of $\xG$ may be associated with a point of $\partial\xO.$ Lemma \ref{Lemm10} clearly shows that $K_{\xm}(\cdot,\xi)\neq K_{\xm}(\cdot,\xi')$ if $\xi\neq\xi'$. Hence, the functions $\mathcal{K}_\xm(x,y_k)$ cannot converge if the
sequence $\{y_k\}$ has more than one cluster point on $\partial\xO$
and different points
of $\partial\xO$ must be associated with different points of $\xG.$ This gives a one-to-one correspondence
between $\partial\xO$ and $\xG$ with $\mathcal{K}_\xm(x,Y)=K_{\xm}(x,\xi)$ when $Y \in\xG$ corresponds
to $\xi \in \partial\xO.$
If  $\xi_k\rightarrow \xi$ in the Euclidean topology then
$\mathcal{K}_\xm(x,Y_k)\rightarrow\mathcal{K}_\xm(x,Y)$ and,
therefore, $Y_k\rightarrow Y$ in the $\xr$-topology by
Lebesgue's dominated convergence theorem.
On the other hand suppose that $Y_k\rightarrow Y$
in the $\xr$-topology. If $\xi_k$ does not converge to $\xi$ in the
Euclidean topology there is
a subsequence $\xi_{k_j}$ such that $\xi_{k_j}\rightarrow\xi'\neq\xi$ in
the Euclidean topology. Then
$Y_{k_j}\rightarrow Y'$ and $Y_{k_j}\rightarrow Y$ in the $\xr$-
topology with $Y\neq Y',$ which is impossible. Therefore, the Martin
$\xr$-topology on $\xO\cup \xG$ is equivalent to the Euclidean
topology on $\xO\cup\partial \xO.$
 \end{proof}

\noindent {\bf\em{Proof of Theorem \ref{poisson}.}}
The result follows immediately by Proposition \ref{martinest} and Proposition \ref{green}. $\hfill\Box$

\medsk

The next lemma will be used to prove the representation formula
of Theorem \ref{Lemm11intro}.
\begin{lemma}\label{tyx}
Assume that $\lambda_{\mu}>0$.
Let $F \subset\partial\xO$ and $D$ be an open smooth neighbourhood of $F.$ We assume $\xO\cap D\subset \xO_\xb$ for some $\xb>0.$ Let $u$ be a positive $L_\xm$-harmonic function in $\xO.$ There exists a $L_\xm$-superharmonic function $V$ such that
$$
V(x)=\left\{\BA {lll}v(x),\qquad&\text{in }\;\xO\setminus D,\\[1mm]
u(x),\qquad&\text{in }\;\xO\cap \overline{D},
\EA\right.
$$
where $v$ satisfies
$$
\left\{
\begin{array}{ll}
L_\mu v=0, & \text{in }\, \xO\setminus\overline{D},\\[1mm]
\lim_{x \in \xO\setminus\overline{D} ,\; x \to y}v(x)=u(y), & \forall y\in \partial D\cap\xO , \\[1mm]
\lim_{x \in \xO\setminus\overline{D},\; x\to y}\frac{v(x)}{\tilde W(x)}=0,
 & \forall y\in\partial\Omega  \setminus \overline{D}.
\end{array}
\right.
$$
\end{lemma}
\begin{proof}
The function $u$ is  $C^2$ in $\xO$ since it is $L_\xm$-harmonic. We assume that $\{r_n\}_{n=0}^\infty$ is a decreasing sequence $r_n\searrow0$ and $r_1<\frac{\xb_0}{16}$. We set
$D_{r_n}=\{\xi \in  \partial D\cap\xO : d(\xi)>2r_n\}$.

Let $0\leq\eta_n\leq1$ be a smooth function such that $\eta_n=1$ in  $\overline{D}_{r_n}$ with compact support in $D_{\frac{r_n}{2}}$. In view of the proof of Lemmas \ref{mainlemma1} and \ref{existence1}, for $m>n$, we can find a  unique solution $v_{n,m}$ of
\[
\left\{
\begin{array}{ll}
L_\mu v =0, & \text{in }\, (\Omega \setminus \overline{\xO}_{\frac{r_m}{2}})\setminus \overline{D},\\[1mm]
\lim_{x \to y}v(x)=\eta_n(y)u(y), & \forall y\in \partial D\cap(\Omega \setminus \overline{\xO}_{\frac{r_m}{2}}) , \\[1mm]
\lim_{x\to y}v(x)=0, &\forall y\in(\partial \xO_{\frac{r_m}{2}})\setminus \overline{D}.
\end{array}
\right.
\]
By comparison principle we have $0\leq v_{n,m} \leq u$
and $v_{n,m} \leq v_{n,m+1}$.
In addition, there exists a constant $c_n=c_n(\|u\|_{L^\infty(D_\frac{r_n}{2})}, \inf_{x\in D_\frac{r_n}{2}}\xf_\xm)$ such that
\[
0\leq v_{n,m}(x)\leq\min\{u(x),c_n\xf_\xm(x)\},
\qquad x\in (\Omega \setminus \overline{\xO}_{\frac{r_m}{2}})\setminus \overline{D}.
\]
Thus $v_{n,m}$ converges to some function $v_n$ as $m\to \infty$
locally uniformly in $\xO\setminus \overline{D}$ and
\bel{fragma12}
0\leq v_{n}(x)\leq\min\{u(x),c_n\xf_\xm(x)\},
\qquad  x\in \xO\setminus \overline{D} \, ,
\quad n\in \mathbb{N}.
\ee
Let $\xi\in \partial\xO\setminus \overline{D}$.
By \eqref{fragma12} and \eqref{harnack} there exists $r_0<\frac{\dist(\xi,\partial D)}{4}$ such that
$$
\frac{v_n(x)}{\phi_\mu(x)}\leq c\frac{v_n(y)}{\phi_\mu(y)}
\leq c\frac{u(y)}{\phi_\mu(y)}, \quad
\forall x,y\in B_{\frac{r_0}{4}}(\xi) \cap \Omega.
$$
Thus $v_n$ converges to some function $v$ locally uniformly
in $ \Omega$.
The desired result now follows easily.
\end{proof} \medskip

We consider a {\it smooth exhaustion} of $\Omega$, that is
an increasing sequence of bounded open smooth domains $\{\xO_n\}$ such that $\overline{\xO_n}\subset \xO_{n+1}$, $\cup_n\xO_n=\xO$ and $\mathcal{H}^{N-1}(\partial \Omega_n)\to \mathcal{H}^{N-1}(\partial \Omega)$.
The operator $L_{\xm }^{\xO_n}$ defined by
\be\label{redu1}
L_{\xm }^{\xO_n}u=-\xD u-\frac{\xm }{d^2_K}u
\ee
is uniformly elliptic and coercive in $H^1_0(\xO_n)$ and its first eigenvalue $\lambda_{\xm }^{\xO_n}$ is larger than $\lambda_{\xm }$.
For $h\in C(\prt \xO_n)$ the problem
\[
\left\{
\begin{array}{ll}
L_{\xm }^{\xO_n}v=0, &  \text{in } \, \xO_n , \\[1mm]
v=h, &  \text{on } \prt \xO_n,
\end{array}
\right.
\]
admits a unique solution which allows to define the $L_{\xm }^{\xO_n}$-harmonic measure on $\prt \xO_n$
by
\[
v(x_0)=\myint{\prt \xO_n}{}h(y)d\omega^{x_0}_{\xO_n}(y).
\]
Thus the Poisson kernel of $L_{\xm }^{\xO_n}$ is
\be\label{redu2'}
K_{L_{\xm }^{\xO_n}}(x,y)=
\myfrac{d\omega^{x}_{\xO_n}}{d\omega^{x_0}_{\xO_n}}(y),\qquad
x\in\xO_n , \;\;  y\in\prt \xO_n.
\ee
%%%%%%%%%%PROOF%%%%%%%%%%%%%%%%%%%%%%%%%%%%%%%%%%%%%%%%%%%%%%%%%%%%%%%%%%%%%%%%%%%%%%%%%%%%%%%%%%%%%%%%%%%%%%%%%
%%%%%%%%%%PROOF%%PROPOSITION%%%%%%%%%%%%%%%%%%%%%%%%%%%%%%%%%%%%%%%%%%%%%%%%%%%%%%%%%%%%%%%%%%%%%%%%%%%%%%%%%%%%
\begin{proposition}\label{22222}
Assume that $\lambda_{\mu}>0$ and $x_0\in \xO_1$. Then for every $Z\in C(\overline{\xO}),$
\be\label{2.27}
\lim_{n\rightarrow\infty}\int_{\partial \xO_n}
Z(x)\tilde{W}(x)d\omega^{x_0}_{\Omega_n}(x)=\int_{\partial \xO}Z(x)d\omega^{x_0}(x).
\ee
\end{proposition}
\begin{proof}
Let $n_0\in \mathbb{N}$ be such that
$$
\hbox{dist}(\partial \xO_n,\partial\xO)<\frac{\xb_0}{16},\qquad\forall n\geq n_0.
$$
For $n\geq n_0$ let $w_n$ be the solution of
\[
\left\{
\begin{array}{ll}
L_{\mu }^{\xO_n}w_n=0,  & \text{in }\,  \xO_n , \\[1mm]
w_n=\tilde W, & \text{on } \prt \xO_n.
\end{array}
\right.
\]
In view of the proof of Lemma \ref{mainlemma1}, there exists a positive constant $c=c(\xO,K,\xm )$ such that
$$
\norm{\frac{w_n}{\tilde W}}_{L^\infty(\xO_n)}\leq c,
\qquad \forall n\geq n_0.
$$
Furthermore
\be
w_n(x_0)=\int_{\partial \xO_n}\tilde W(x)d\omega^{x_0}_{\xO_n}(x)<c.\label{wn}
\ee
We extend $\omega^{x_0}_{\xO_n}$ to a Borel measure on $\overline{\xO}$ by
setting $\omega^{x_0}_{\xO_n}(\overline{\xO}\setminus \xO_n) = 0,$ and keep
the notation $\omega^{x_0}_{\xO_n}$ for the extension. Because of (\ref{wn})
the sequence $\{\tilde W\omega^{x_0}_{\xO_n}\}$ is bounded in the space
$\mathfrak M_b(\overline\Omega)$ of bounded Borel measures in $\overline\Omega$.
Thus there exists a subsequence, still denoted by
$\{\tilde W\omega^{x_0}_{\Omega_n}\}$, which converges narrowly to some
positive measure, say $\widetilde{\omega}$, which is clearly supported on $\partial\xO$ and satisfies
$\norm{\widetilde{\omega}}_{\mathfrak M_b}\leq c$ by (\ref{wn}).
Thus for every
$Z \in C(\overline{\xO})$ there holds
$$
\lim_{n\rightarrow\infty}\int_{\partial \xO_n}Z \, \tilde{W} d \omega^{x_0}_{\Omega_n}=\int_{\partial\xO}Zd \widetilde{\omega}.
$$
Setting $\xz=Z\lfloor_{\prt\Omega}$ and
$$
z(x):=\int_{\partial\xO}K_{\xm}(x,y)\xz(y)d \omega^{x_0}(y)
$$
we then have
$$
\lim_{d(x)\to 0}\myfrac{z(x)}{\tilde W(x)}=\zeta\qquad \text{ and }
\qquad z(x_0)=\int_{\partial\xO}\zeta d \omega^{x_0}.
$$
By Lemma \ref{mainlemma1}, $\frac{z}{\tilde W}\in C(\overline{\xO})$.
Since $\frac{z}{\tilde W}\lfloor_{\prt \xO_n}$ converges uniformly to
$\zeta$ as $n\to\infty$, there holds
$$
z(x_0)=\int_{\partial \xO_n}z\lfloor_{\prt \xO_n}d \omega^{x_0}_{\xO_n}=
\int_{\partial \xO_n}\tilde W\frac{z\lfloor_{\prt \xO_n}}{\tilde W}
d \omega^{x_0}_{\xO_n}\to \int_{\partial\xO}\zeta d \tilde\omega ,
\qquad \text{ as }\;n\to\infty.
$$
It follows that
$$
\int_{\partial\xO}\xz d \widetilde{\omega}=\int_{\partial\xO}\xz d \omega^{x_0},
\qquad\forall\xz\in C(\partial\xO).
$$
Consequently $d\widetilde{\omega}=d \omega^{x_0}.$ Because the limit does
not depend on the subsequence it follows that the whole sequence
${\tilde W(x)d\omega^{x_0}_{\xO_n}}$ converges weakly to $\xo^{x_0}.$
This implies \eqref{2.27}.
\end{proof}

%%%%%%%%%%%%%%%%%%%%%%%%%%%%%%%%%%%%%%%%%%%%%%%%%%%%%%%%%%%%%%%%%%%%%%%%%%%%%%%%%%%%%%%%%%%%%%%%%%%%%%%%%%%%%%%%%%%%%%%%%%%%%%%%%%%%%%%%%%%%%%%%%%%%%%%%%%%%%%%%%%%%%%%%%%%%%%

\noindent {\bf\em{Proof of Theorem \ref{Lemm11intro}.}}
The proof which is presented below follows the ideas of the one of \cite[Th. 4.3]{hunt}.
Let $B$ be a relatively closed subset of $\Omega$. We define
\[
R^B_u(x):=\inf\big\{\psi(x):\;\psi\;\text{is a
nonnegative supersolution in}\; \xO\;
\text{with}\;\psi\geq u\;\text{on}\;B\big\}.
\]
For a closed subset $F$ of $\partial\xO,$ we define
$$
\xn^x(F):=\inf\{R^{\xO\cap\overline{G}}_u(x):\;F\subset G,\;G\;\text{open in}\;\mathbb{R}^N \}.
$$
The set function $\xn^x$ defines a regular Borel
measure on $\partial\xO$ for each fixed $x\in\xO$.
Since $\xn^x(F)$ is a positive $L_\xm$-harmonic function in $\xO$
the measures $\xn^x$, $x\in\Omega$, are mutually
absolutely continuous by Harnack inequality. Hence,
$$
\xn^x(F)=\int_Fd\xn^x(y)=\int_F\frac{d\xn^x}{d\xn^{x_0}}d\xn^{x_0}(y).
$$
We assert that $\frac{d\xn^x}{d\xn^{x_0}}$ = $K_{\xm}(x, y)$ for
$\xn^{x_0}$-a.e.  $y\in \partial\xO$. By Besicovitch's
theorem,
$$
\frac{d\xn^x}{d\xn^{x_0}}(y)=\lim_{r\to0}\frac{\xn^x(\xD_{r}(y))}{\xn^{x_0}(\xD_r(y))},
$$
for $\xn^{x_0}$-a.e.  $y\in \partial\xO$.
In view of the proof of Proposition \ref{uniq}, we can prove that the
function $\xn^x(\xD_{r}(y))$ is $L_\xm$-harmonic and
\[
\lim_{x \in \xO,\; x\to \xi}\frac{  \xn^x(\xD_{r}(y))  }{\tilde W(x)}=0,
\qquad\forall \xi\in\partial\Omega  \setminus \overline{\xD}_{r}(y).
\]
Proceeding as in the proof of Proposition \ref{uniq}, we may prove that $\frac{d\xn^x}{d\xn^{x_0}}$ is a kernel function, and by the
uniqueness of kernel functions the assertion follows. Hence
$$\xn^x(A)=\int_AK_{\xm}(x,y)d\xn^{x_0}(y),$$
for all Borel $A\subset\partial\xO$ and in particular
$$u(x)=\xn^x(\partial\xO)=\int_{\partial\xO}K_{\xm}(x,y)d\xn^{x_0}(y).$$
Suppose now that
$$
u(x)=\int_{\partial\xO}K_{\mu}(x,y)d\xn(y),
$$
for some nonnegative Borel measure $\xn$ on $\partial\xO$.
We will show that $\xn(F)=\xn^{x_0}(F)$ for any closed set $F\subset  \partial\xO$.

Choose a sequence of open sets $\{G_\ell \}$ in $\mathbb{R}^N$ such that $\cap_{\ell=1}^\infty G_\ell=F$ and
$$
\xn^{x}(F)=\lim_{l\rightarrow\infty}R^{\Omega  \cap \overline{G}_\ell}_u(x).
$$
Since
$$
R^B_u(x)\leq R^A_u(x), \qquad\text{if }\; B\subset A,
$$
we can choose $\{G_\ell\}$ so that
$\overline{G}_{\ell+1}\subset G_\ell,\;\forall \ell \geq1$ and
$ G_\ell$ to be a $C^2$ domain for all $\ell \geq1.$ In view of the proof
of Lemma \ref{tyx}, we may assume that
$R^{\Omega  \cap \overline{G}_\ell}_u(x)=V_\ell$ where $V_\ell$ is the
$L_\mu$-superharmonic in Lemma \ref{tyx} for $D=G_\ell$.
Furthermore we have that
$R^{\Omega \cap \overline{G}_\ell}_u(x)=u(x)$
in $\Omega \cap \overline{G}_\ell $ and
$R^{\Omega \cap \overline{G}_\ell}_u(x) \leq u(x)$
for all $x\in \Omega$.

We consider an increasing sequence of smooth domains
$\{\Omega_\ell \}$ such that
$\overline{\Omega_\ell}\subset \Omega_{\ell+1}$,
$\cup_{\ell=1}^\infty\Omega_\ell=\Omega$,
$ G_\ell\cap\xO\subset\overline{\xO}\setminus \xO_\ell $,
$\mathcal{H}^{N-1}(\partial \Omega_\ell) \to
\mathcal{H}^{N-1}(\partial \Omega)$.
Let $w^{x_0}_{\xO_n}$ be  the $L_\xm$-harmonic measure in
$\partial \xO_n$ (see \eqref{redu1}-\eqref{redu2'}).
%%%
Let $n>\ell$ and let $v_n$ be the unique solution of
\[
\left\{
\begin{array}{ll}
L_\xm v=0, & \mbox{ in }\xO_n , \\
v=R^{\Omega\cap\overline{G}_\ell}_u , &\mbox{ on }\partial \xO_n.
\end{array}
\right.
\]
Since $R^{\Omega \cap \overline{G}_\ell}_u(x)$
is a supersolution in $\xO$ we have
$R^{\Omega \cap \overline{G}_\ell}_u(x)\geq v_n(x)$, $x\in \xO_n$.
Hence
\[
R^{\Omega \cap \overline{G}_\ell}_u(x_0) \geq v_n(x_0) =\int_{\partial{\xO_n}}R^{\Omega \cap \overline{G}_\ell}_u(y)dw^{x_0}_{\xO_n}(y)
\geq\int_{\partial  \xO_n\cap G_\ell}R^{\Omega \cap \overline{G}_\ell}_u(y)dw^{x_0}_{\xO_n}(y).
\]
Now, by Lemma \ref{tyx},
\begin{align*}
\int_{\partial \xO_n\cap G_\ell}R^{\Omega \cap \overline{G}_\ell}_u(y)dw^{x_0}_{\xO_n}(y)
&=\int_{\partial \xO_n\cap G_\ell}u(y)dw^{x_0}_{\xO_n}(y)\\
&=\int_{\partial \xO_n\cap G_\ell}\int_{\partial\Omega}K_{\xm}(y,\xi)d\xn(\xi)dw^{x_0}_{\xO_n}(y)\\
&=\int_{\partial\Omega}\int_{\partial \xO_n\cap G_\ell}K_{\xm}(y,\xi)dw^{x_0}_{\xO_n}(y)d\xn(\xi)\\
&\geq\int_{F}\int_{\partial \xO_n\cap G_\ell}K_{\xm}(y,\xi)dw^{x_0}_{\xO_n}(y)d\xn(\xi).
\end{align*}
%%%
Let $\xi\in F$. We have
\begin{align*}
1=K_{\mu}(x_0,\xi)&=\int_{\partial \xO_n\cap G_\ell}K_{\xm}(y,\xi)dw^{x_0}_{\xO_n}(y)+\int_{\partial{\xO_n}\setminus G_\ell}K_{\xm}(y,\xi)dw^{x_0}_{\xO_n}(y)\\
\end{align*}
But
\[
K_{\mu}(y,\xi)\leq c d(y)d_K^{\xgp}(y) ,
\qquad\forall y\in\partial{\xO_n}\setminus G_\ell,
\]
thus by Proposition \ref{22222} we have that
$$\lim_{n\rightarrow\infty}\int_{\partial{\xO_n}\setminus G_\ell}K_{\xm}(y,\xi)dw^{x_0}_{\xO_n}(y)=0.$$
Combining all the above inequality and using Lebesgue's dominated convergence theorem we obtain
$$R^{\Omega \cap \overline{G}_\ell}_u(x_0)\geq\lim_{n\rightarrow\infty}\int_{F}\int_{\partial \xO_n \cap G_\ell}K_{\xm}(y,\xi)dw^{x_0}_{\xO_n}(y)d\xn(\xi)=\xn(F),$$
which implies
$$\xn^{x_0}(F)\geq\xn(F).$$

 For the opposite inequality, let $m<\ell$. Then
 \begin{align*}
R^{\Omega \cap \overline{G}_\ell}_u(x_0)&=\int_{\partial \xO_\ell}R^{\Omega \cap \overline{G}_\ell}_u(y)dw^{x_0}_{\xO_\ell}(y)\\
&=\int_{\partial \xO_\ell \cap G_{m}}R^{\Omega \cap \overline{G}_\ell}_u(y)dw^{x_0}_{\xO_\ell}(y)+\int_{\partial \xO_\ell\setminus  G_{m}}R^{\Omega \cap \overline{G}_\ell}_u(y)dw^{x_0}_{\xO_\ell}(y).
\end{align*}
In view of the proof of Lemma \ref{tyx}, we have that
$$
R^{\Omega \cap \overline{G}_\ell}_u(x)\leq C d(x)d_K^{\xgp}(x),\qquad\forall x\in \xO\setminus G_{m}.
$$
 Thus by Proposition \ref{22222} we have
$$
\lim_{l \rightarrow\infty}\int_{\partial \xO_\ell \setminus  G_{m}}R^{\Omega \cap \overline{G}_\ell}_u(y)dw^{x_0}_{\xO_\ell}(y)=0,
$$
and
\begin{align*}
\int_{\partial \xO_\ell \cap G_{m}}R^{\Omega \cap \overline{G}_\ell}_u(y)dw^{x_0}_{\xO_\ell}(y)&\leq\int_{\partial \xO_\ell\cap G_{m}}u(y)dw^{x_0}_{\xO_\ell}(y)\\
&=\int_{\partial \xO_\ell \cap G_{m}}\int_{\partial\Omega}K_{\xm}(y,\xi)d\xn(\xi)dw^{x_0}_{\xO_\ell}(y)\\
&=\int_{\partial\Omega}\int_{\partial \xO_\ell \cap G_{m}}K_{\xm}(y,\xi)dw^{x_0}_{\xO_\ell}(y)d\xn(\xi).
\end{align*}
If $\xi\in\partial\Omega\setminus G_{m-1}$ we have again by Proposition \ref{22222} that
$$\lim_{\ell \rightarrow\infty}\int_{\partial \xO_\ell \cap G_{m}}K_{\xm}(y,\xi)dw^{x_0}_{\xO_\ell}(y)=0.$$
%%%
If $\xi\in\partial\xO\cap G_{m}$, then
$$
\int_{\partial \xO_\ell \cap G_{m}}K_{\xm}(y,\xi)dw^{x_0}_{\xO_\ell}(y)\leq K_{\xm}(x_0,\xi)=1.
$$
Combining all the above inequalities, we obtain
$$\xn^{x_0}(F)=\lim_{\ell\rightarrow\infty}R^{\Omega \cap \overline{G}_\ell}_u(x_0)\leq \int_{\partial\Omega\cap\overline{G}_{m-1}}K_{\xm}(x_0,\xi)d\xn(\xi)=\xn(\partial\Omega\cap \overline{G}_{m-1}),$$
which implies
$$\xn^{x_0}(F)\leq \xn(F). $$
Thus we get the desired result.
$\hfill\Box$

\section{Boundary value problem for linear equations}

\subsection{Boundary trace}

We first examine the boundary trace of $\BBK_\mu[\nu]$.
\begin{lemma}
\label{tracemartin}
Let $\mu\leq k^2/4$ and assume that $\lambda_{\mu}>0$.
Then
for any $\nu \in \GTM(\partial \Omega)$ we have $\tr(\BBK_\mu[\nu])=\nu$.	
\end{lemma}
\begin{proof}
The proof is the similar to the proof of Lemma 2.2 in \cite{MV} and
we omit it.
\end{proof}
\begin{lemma} \label{tracegreen}
Let $\mu\leq k^2/4$ and assume that $\lambda_{\mu}>0$.
For $\tau\in\mathfrak{M}(\xO;\ei)$  we set $u=\mathbb{G}_{\mu}[\tau].$ Then $u\in W^{1,p}_{loc}(\xO)$ for every $1<p<\frac{N}{N-1}$
and $\tr(u)=0$.
\end{lemma}
\begin{proof}
By \cite[Theorem 1.2.2]{MVbook}, $u\in W^{1,p}_{loc}(\xO)$ for every $1<p<\frac{N}{N-1}$. Let $\{\xO_n\}$ be a smooth exhaustion of $\xO$
(cf. \eqref{redu1}) and $v_n$ be the unique solution of
$$
 \left\{ \BAL
L_{\xm }^{\xO_n}v&=0,\qquad&&\text{in } \xO_n ,\\
v&=u, \qquad&&\text{on } \prt \xO_n.
\EAL \right.
$$
We note here that $v_n(x_0)=\int_{\prt \xO_n} u(y) d\omega^{x_0}_{\xO_n}(y)$.
We first assume that $\tau\geq0$.  Let $G^{\xO_n}_{\mu}$ be the Green kernel of $L_\xm$ in $\xO_n$.
Then $G^{\xO_n}_{\xm}(x,y)\nearrow G_{\mu}(x,y)$ for any $x,y\in \xO$,
$x\neq y$. Putting $\tau_n=\tau|_{\xO_n}$ and $u_n=\BBG_\mu^{\xO_n}[\tau_n]$ we then have
$u_n\nearrow u$ a.e. in $\xO$. By uniqueness we have that $u=u_n+ v_n$ a.e.
in $\xO_n$. In particular, $u(x_0)=u_n(x_0)+v_n(x_0)$ and therefore
$\lim_{n \to \infty}v_n(x_0)=0$. Consequently, $\tr(u)=0$.
	
In the general case, the result follows by linearity.
\end{proof}

\begin{theorem}
Let $\mu\leq k^2/4$ and assume that $\lambda_{\mu}>0$.
We then have \newline
$\ia$ Let $u$ be a positive $L_\mu$-superharmonic function in the sense of distributions in $\Omega$. Then $u \in L^1(\Omega;\ei)$ and there exist $\tau \in \GTM^+(\Omega;\ei)$ and $\nu \in \GTM^+(\partial\Omega)$ such that
\bel{reprweaksol} u=\BBG_{\mu}[\tau]+\BBK_{\mu}[\nu].
\ee
In particular, $u \geq \BBK_\mu[\nu]$ in $\Omega$ and $\tr(u)=\nu$.

\noindent
$\ib$ Let $u$   be a positive $L_\mu$-subharmonic function  in the sense of
distributions in $\Omega$. Assume that there exists a positive
$L_\mu$-superharmonic function $w$ such that $u \leq w$ in $\Omega$.
Then $u \in L^1(\Omega;\ei)$ and there exist $\tau \in \GTM^+(\Omega;\ei)$
and $\nu \in \GTM^+(\partial\Omega)$ such that
\bel{vGK2} u+\BBG_{\mu}[\tau]=\BBK_{\mu}[\nu].
\ee
In particular, $u \leq \BBK_\mu[\nu]$ in $\Omega$ and $\tr(u)=\nu$.
\end{theorem}
\begin{proof}
(i) Since $L_\mu u \geq 0$ in the sense of distributions in $\Omega$, there
exists a nonnegative Radon measure $\tau$ in $\Omega $ such that
$L_\mu u=\tau$ in the sense of distributions. By \cite[Lemma 1.5.3]{MVbook},
$u\in W^{1,p}_{loc}(\xO).$
	
Let $\{\xO_n\}$ be a smooth exhaustion of $\xO$ (cf. \eqref{redu1}).
Denote by $G_\mu^{\xO_n}$ and $P_\mu^{\xO_n}$ the Green kernel and the
Poisson kernel of $L_\mu$ in $\xO_n$ respectively
(recalling that $P_\mu^{\xO_n} = -\partial_{\bf n}G_\mu^{\xO_n}$).
Then $u=\BBG_{\mu}^{\xO_n}[\tau]+ \BBP_\mu^{\xO_n}[u]$, where
$\BBG_\mu^{\xO_n}$ and $\BBP_\mu^{\xO_n}$ are the Green operator
and the Poisson operator for $\xO_n$ respectively.
	
Since $\tau$ and $\BBP_\mu^{\xO_n}[u]$ are nonnegative and
$G_{\mu}^{\xO_n}(x,y)\nearrow G_{\mu}(x,y)$ for any $x,y\in \xO$, $x\neq y$,
we obtain $0 \leq \mathbb{G}_{\mu}[\tau]\leq u$  a.e. in  $\xO$. In particular, $0 \leq \mathbb{G}_{\mu}[\tau](x_0) \leq u(x_0)$ where $x_0 \in \Omega$ is a fixed reference point. This, together with the estimate $G_\mu(x_0,\cdot) \gtrsim \ei$ a.e. in $\xO$, implies $\tau \in \GTM(\Omega;\ei)$.
	
Moreover, we see that $u-\mathbb{G}_{\mu}[\tau]$ is a nonnegative $L_\mu$-
harmonic function in $\xO$. Thus by Theorem
\ref{Lemm11intro}
there exists a unique $\nu \in \GTM^+(\partial\xO)$ such that \eqref{reprweaksol} holds. \smallskip
	
(ii) Since $L_\mu u \leq 0$ in the sense of distributions in $\xO$, there
exists a nonnegative Radon measure $\tau$ in $\xO$ such that
$L_\mu u =-\tau$ in the sense of distributions.
By \cite[Lemma 1.5.3]{MVbook}, $u \in W^{1,p}_{loc}(\xO)$. Let $\xO_n$ and
$\BBP_\mu^{\xO_n}$ be as in (i).
Then  $u+\BBG_{\mu}^{\xO_n}[\tau]=\BBP_\mu^{\xO_n}[u]$. This, together with
the fact that $u \geq 0$ and  $\BBP_\mu[u] \leq w$, implies
$\BBG_\mu^{\xO_n}[\tau] \leq w$. By using a similar argument as in (i), we
deduce that $\tau \in \GTM(\xO;\ei)$ and there exists
$\nu \in \GTM^+(\partial \Omega )$ such that \eqref{vGK2} holds.
\end{proof}

\subsection{Boundary value problem for linear equations}

We recall (cf. \eqref{xkom}) that for $\mu\leq k^2/4$ we have
defined
\[
 {\bf X}_\mu(\Omega ,K):=\{ \zeta \in H_{loc}^1(\Omega): \phi_\mu^{-1} \zeta \in H^1(\Omega;\phi_\mu^{2}), \, \phi_\mu^{-1}L_\mu \zeta \in L^\infty(\Omega)  \}.
\]
\begin{lemma} \label{testfunc-prop}
Let $\mu\leq k^2/4$ and assume that $\lambda_{\mu}>0$.
Then any $\zeta \in {\bf X}_\mu(\Omega ,K)$ satisfies
$|\zeta| \leq c \ei$ in $\Omega$.
\end{lemma}
\begin{proof}
Let $\xz\in{\bf X}_\mu(\Omega ,K)$ and $g=L_{\xm }\zeta.$
Then there exist $C=C(\norm{g\ei^{-1}}_{L^\infty(\xO)},\xl_\xm)$ such that
$|g|\leq C \xl_\xm\ei$ in $\xO$.
Set  $\tilde\xz=C^{-1}\ei^{-1}\zeta$.
Then,
\bal
\int_\Omega \phi_\mu^2\nabla\tilde\xz \cdot
\nabla\psi \,  dx+
\lambda_\mu \int_\Omega \phi_\mu^2 \tilde\xz \psi \, dx=
\frac{1}{C}\int_\Omega \phi_\mu  g\psi \, dx\leq \xl_\xm\int_\xO\ei^2\psi \, dx
\; , \qquad \forall 0\leq\psi \in  H^1_0(\xO ; \phi_\mu^2).
\eal
By taking $\psi=(\tilde\xz -1)_+$ as test function in the above inequality,
we obtain that $\tilde\xz \leq 1,$ which implies $\xz\leq C\ei$ in $\xO$.
Applying the same argument to $-\xz$ completes the proof.
\end{proof}

\begin{lemma}\label{existence3}
Let $\mu\leq k^2/4$ and assume that $\lambda_{\mu}>0$.
Given $\tau\in\GTM(\xO;\ei)$ there exists a unique weak
solution $u$ of \eqref{NHLPintro} with $\xn=0$. Furthermore $u=\mathbb{G}_{\mu}[\tau]$ and there holds
	\begin{align}  \label{esti1}
	\| u \|_{L^1(\Omega;\ei)}  \leq \frac{1}{\lambda_\mu} \| \tau \|_{\GTM(\xO;\ei)}.
	\end{align}
\end{lemma}
\begin{proof}
{\emph{A priori estimate.}} Assume $u \in L^1(\Omega;\ei)$ is a weak
solution of \eqref{NHLPintro} with $\nu=0$.
Let $\zeta \in {\bf X}_\mu(\Omega,K)$ be such that $L_\mu\zeta=\sign( u)\ei$. By Kato's inequality,
$$
L_\mu |\zeta|\leq \sign(\zeta)L_\mu \zeta \leq \ei=L_\mu\Big(\frac{1}{\lambda_\mu}\phi_\mu \Big).
$$
Hence by Lemmas \ref{comparison} and \ref{testfunc-prop}
we deduce that $|\zeta|\leq \frac{1}{\xl_\xm} \ei$ in $\xO$. This, combined with \eqref{NHLPintro}
(for $\nu=0$) implies \eqref{esti1}.
\medskip
	
	\noindent {\emph{Uniqueness.}} The uniqueness follows directly from \eqref{esti1}. \medskip
	
	\noindent {\emph{Existence.}} Assume $\tau=fdx$ with $f \in L^\infty(\Omega)$ with compact support in $\xO$. The existence
of a solution $u$ follows by Lemma \ref{existence1}.
	
Since $f\in L^\infty(\Omega)$ has compact support in $\xO$, there exists a positive constant $c=c({\text{supp}(f), \|f\|_{\infty},
\xO,K,\xm})$ such that
$|f|\leq c\ei$.
It follows that $u\in  {\bf X}_\mu(\Omega)$ and therefore
$|u(x)|\leq C\ei(x)$, $x\in\xO$, by Lemma \ref{testfunc-prop}.

Next we will show that $u=\BBG_\mu[f]$. Set $w=\BBG_\mu[f]$.
We can easily show that $w$ satisfies $L_\xm w=f$ in the sense of
distributions in $\xO$ and by \eqref{Greenest} there exists a positive
constant $C$ such that $|w(x)| \leq C\ei(x)$  for all $x\in\xO$. Therefore,
$$
\lim_{\dist(x,F) \to 0}\frac{|u(x)-w(x)|}{\tilde W(x)} \leq  C\lim_{\dist(x,F) \to 0}\frac{\ei(x)}{\tilde W(x)} =0
$$
for any compact set $F \subset \partial \Omega$.
Furthermore, we note that $|u-w|$ is $L_\mu$-subharmonic in $\xO$. Hence from Lemma \ref{comparison}, we deduce that $|u-w|=0$, i.e. $u=w$ in $\xO$.
	
	Now assume that $\tau =fdx$ with $f\in L^1(\Omega;\ei)$. Let $\{\xO_n\}$ be a smooth exhaustion of $\xO$ (see \eqref{redu1}). Set $f_n=\chi_{\xO_n}g_n(f)\in L^\infty(\Omega),$ where

\bal
g(t)=\left\{
\begin{array}{ll}
 n, &\text{if}\; t\geq n, \\
t,  &\text{if}\; -n<t<n,\\
-n,  &\text{if}\; t\leq- n .
\end{array}
\right.
\eal
Then $f_n\rightarrow f$ in $ L^1(\Omega;\ei)$. Put $u_n:=\BBG_\mu[f_n]$. Then
\[
 \int_{\Omega}u_n L_{\xm }\zeta \, dx=\int_{\Omega} f_n \zeta  \, dx ,
\qquad\forall \xi \in\mathbf{X}_\xm(\xO,K).
\]
	By \eqref{esti1} we can easily prove that $u_n=\BBG_\mu[f_n] \to \BBG_\mu[f]:=u$ in $L^1(\Omega;\ei)$. Then by letting $n \to \infty$ and using Lemma \ref{testfunc-prop}, we deduce the desired result when $f\in L^1(\Omega;\ei)$.
	
Assume finally that $\tau \in \GTM(\xO;\ei)$. Let $\{f_n\}$ be a sequence in  $L^1(\Omega;\ei)$ such that
$f_n\rightharpoonup \tau $ in $C_{\ei}(\xO)$,
where
$C_{\ei}(\xO)=\{ \zeta \in C(\xO) : \ei \zeta  \in L^\infty(\Omega)\}$.
Then proceeding as above we can prove that
$u_n=\BBG_\mu[f_n] \to \BBG_\mu[\tau]:=u$ in $L^1(\Omega;\ei)$
and $u$ satisfies \eqref{NHLPintro} with $\xn=0.$
\end{proof}

\noindent
{\bf \em Proof of Theorem \ref{thm:val}.}
First we note that by Theorem \ref{poisson}, we can easily show that
\ba\label{210}
\| \BBK_\mu[|\nu|] \|_{L^1(\Omega;\ei)}
\leq c \| \nu \|_{\GTM(\partial\Omega )}.
\ea

\noindent{\textit{Existence.}}
The existence and \eqref{reprweaksol1} follow from Lemma \ref{existence3} and \eqref{210}. \medskip
	
\noindent \textit{A priori estimate \eqref{esti2}.}
This follows from \eqref{210}, \eqref{esti1} and \eqref{reprweaksol1}. \medskip

\noindent {\textit{Uniqueness.}}  Uniqueness follows from \eqref{esti2}.  \medskip
	
\noindent {\emph{Proof of estimates \eqref{poi4}--\eqref{poi5}.}}
Assume $d\tau=fdx+d\rho$ and let $\{\xO_n\}$ be a smooth exhaustion of $\xO$. Let $v_\tau^n$ be the solution of
$$ \left\{
\begin{array}{ll}
	L_{\xm }^{\xO_n}v=0, & \text{ in } \xO_n\\
	v=\mathbb{G}_{\mu}[\tau], & \text{ on } \prt \xO_n,
\end{array}
\right.
$$
and $w_\xn=\mathbb{K}_{\xm}[\nu]$. Then, by uniqueness,
$u=\mathbb{G}^{\xO_n}_{\mu}[\tau|_{\xO_n}]+v_\tau+w_\xn$ and
$|u|\leq \mathbb{G}_{\mu}[|\tau|]+w_{|\xn|}$ $\mathcal{H}^{N-1}$-a.e. on $\partial \xO_n$ (here on $\mathcal{H}^{N-1}$ denotes the Hausdorff measure on $\partial \xO_n$).
	
Let $\eta \in C_c^2(\overline{\xO_n})$ be non-negative and such that
$\eta=0$ on $\partial\Omega_n$.
By \cite[Proposition 1.5.9]{MVbook},
\begin{align*}
	\int_{\xO_n}|u|L_\xm \eta \, dx
	 \leq \int_{\xO_n}\sign(u)f\eta \, dx+ \int_{\xO_n}\eta d|\rho|-\int_{\partial \xO_n}|u|\frac{\partial\eta}{\partial {\bf n}^n}dS
	\end{align*}
where ${\bf n}^n$ is the unit outer normal vector on $\partial \xO_n.$
Since $|u| \leq \BBG_\mu[|\tau|] + w_{|\nu|}$ a.e. on
$\partial \xO_n$ and $\frac{\partial \eta}{\partial {\bf n}^n} \leq 0$
on $\partial \xO_n$, using integration by parts we obtain
\begin{align*}
	-\int_{\partial \xO_n}|u|\frac{\partial\eta}{\partial {\bf n}^n}dS \leq -
\int_{\partial \xO_n}(\mathbb{G}_{\mu}[|\tau|]+w_{|\xn|})
\frac{\partial\eta}{\partial {\bf n}^n}dS= \int_{\xO_n}(v_{|\tau|}^n+w_{|\xn|})L_\xm\eta \, dx.
	\end{align*}
	Hence
	\begin{align}
\int_{\xO_n}|u|L_\xm \eta dx \leq \int_{\xO_n}\sign(u)f\eta \, dx+ \int_{\xO_n}\eta d|\rho| +\int_{\xO_n}(v_{|\tau|}^n+w_{|\xn|})L_\xm\eta \, dx.
\label{ee1}
\end{align}
Let $\xz\in \mathbf{X}_\xm(\xO,K)$, $\zeta > 0$ in $\xO$. Let $z_n$ and $\zeta_n$ be respectively solutions of
$$
\left\{ \BAL
L_{\xm }z_n&=L_{\xm }\xz, \quad&&\text{in } \xO_n , \\
z_n&=0 , \quad&&\text{on } \prt \xO_n,
\EAL \right. \quad\quad
 \left\{ \BAL
L_{\xm }\zeta_n&=\sign(z_n) L_{\xm }\xz , \quad&&\text{in } \xO_n , \\
\zeta_n&=0 , \quad&&\text{on } \prt \xO_n.
\EAL \right.
$$
By Kato's inequality, $L_\mu |z_n|\leq \sign(z_n) L_\mu z_n$ in the sense of distributions in $\xO_n$. Hence by a comparison argument, we have that $|z_n|\leq \xz_n$ in $\xO_n$. Furthermore it can be checked  that $z_n \to \zeta$ and $\xz_n \to \xz$ in $L^1(\Omega;\ei)$ and locally uniformly in $\xO$.
	
Now note that \eqref{ee1} is valid for any nonnegative solution $\eta\in C_c^2(\xO_n)$. Thus we can use $\zeta_n$ as a test function in \eqref{ee1} to obtain
\begin{equation}
\label{ee2}
\begin{aligned}
\int_{\xO_n}|u|\sign(z_n) L_\xm \xz dx &\leq \int_{\xO_n}\sign(u)f\xz_n dx+ \int_{\xO_n}\xz_n d|\rho| \\
&\qquad +\int_{\xO_n}(v_{|\tau|}^n+w_{|\xn|})\sign(z_n)L_\xm\xz dx.
\end{aligned}
\end{equation}
Also, since $\mathbb{G}_{\mu}[|\tau|]=\mathbb{G}_{\mu}^{\xO_n}[|\tau||_{\xO_n}]+v_{|\tau|}^n$ a.e. in $\xO_n$, we deduce that $v_{|\tau|}^n\rightarrow 0$ in $L^1(\Omega;\ei)$ as $n \to \infty$. Thus sending $n \to \infty$ in \eqref{ee2} we obtain \eqref{poi4} since $\zeta>0$ in $\xO$. Estimate \eqref{poi5} follows by adding \eqref{poi4} and \eqref{NHLPintro}. Thus the proof is complete when $\xz$ is positive.
	
If $\xz$ is nonnegative we set $\xz_\xe=\xz+\xe\ei.$ Then estimates \eqref{poi4} and \eqref{poi5} are valid for $\xz_\xe$ for any $\xe>0.$  The desired result follows by letting $\varepsilon \to 0$.
\end{proof}

\section{The Nonlinear Problem}
\subsection{Weak $L^p$  estimates}

We denote by
$L^p_w(\Omega;\tau)$, $1 \leq p < \infty$, $\tau \in \GTM^+(\Omega)$, the
weak $L^p$ space (or Marcinkiewicz space) defined as follows:
a measurable function $f$ in $\Omega$
belongs to this space if there exists a constant $c$ such that
\[
\lambda_f(a;\tau):=\tau(\{x \in \Omega: |f(x)|>a\}) \leq ca^{-p},
\forevery a>0.
\]
The function $\lambda_f$ is called the distribution function of $f$ (relative to
$\tau$). For $p \geq 1$, denote
$$ L^p_w(\Omega;\tau)=\{ f \text{ Borel measurable}:
\sup_{a>0}a^p\lambda_f(a;\tau)<\infty\}, $$
\bel{semi}
\norm{f}^*_{L^p_w(\Omega;\tau)}=(\sup_{a>0}a^p\lambda_f(a;\tau))^{\frac{1}{p}}. \ee
This is not a norm, but for $p>1$, it is
equivalent to the norm
\[
\norm{f}_{L^p_w(\Omega;\tau)}=\sup\left\{
\frac{\int_{\omega}|f|d\tau}{\tau(\omega)^{1/p'}}:\omega \sbs \Omega, \omega \text{
measurable},\, 0<\tau(\omega)<\infty \right\}.
\]
More precisely,
\bel{equinorm} \norm{f}^*_{L^p_w(\Omega;\tau)} \leq \norm{f}_{L^p_w(\Omega;\tau)}
\leq \myfrac{p}{p-1}\norm{f}^*_{L^p_w(\Omega;\tau)}. \ee
When $d\tau=\ei dx$, for simplicity, we use the notation $L_w^p(\Omega;\ei)$.  Notice that,
$$
L_w^p(\Omega;\ei) \sbs L^{r}(\Omega;\ei), \quad  \forevery r \in [1,p).
$$
From \eqref{semi} and \eqref{equinorm} follows that for any
$u \in L_w^p(\Omega;\ei)$ there holds
\ba\label{300}
\int_{\{|u| \geq s\} }\ei dx \leq s^{-p}\norm{u}^p_{L_w^p(\Omega;\ei)}.
\ea

Let us recall \cite[Lemma 2.4]{BVi} which will be useful in the sequel.
\begin{proposition}
\label{bvivier}
Let $\omega$ be a nonnegative bounded Radon measure on $\partial \Omega$ and $\eta\in C(\xO)$ be a positive weight function. Let $\CH$ be a continuous nonnegative function
on $\xO\times \partial\Omega$. For $\xl > 0$ let
$$
A_\xl(y)=\{x\in\xO :\;\; \CH(x,y)>\xl\} \; , \quad  \quad
m_{\xl}(y)=\int_{A_\xl(y)}\eta(x)dx.
$$
Let $y\in\partial\Omega$ and
suppose that there exist $C>0$ and $\tau>1$ such that
$m_{\xl}(y)\leq C\xl^{-\tau}$ for every $\lambda>0$. Then the function
$$
\BBH[\omega](x):=\int_{D}\CH(x,y)d\omega(y)
$$
belongs to $L^\tau_w(\Omega;\eta )$ and
$$
\|\BBH[\omega]\|_{L^\tau_w(\Omega;\eta)}\leq
(1+\frac{C\tau}{\tau-1})\omega(\partial\Omega).
$$
\end{proposition}

\begin{theorem}
\label{lpweakmartin1}
Let $\mu\leq k^2/4$. We set $p=\min\left(\frac{N+1}{N-1},\frac{N+\xgp+1}{N+\xgp-1}\right)$
and in addition assume that $\lambda_{\mu}>0$.
Then there exists a positive constant $C=C(\Omega,K,\mu)$ such that
\[
\norm{\mathbb{K}_{\mu}[\nu]}_{L_w^{p}(\Omega;\ei)} \leq
C\norm{\nu}_{\mathfrak{M}(\partial\xO)}
\]
for any measure $\nu\in \mathfrak{M}(\partial\xO)$.
\end{theorem}
\begin{proof} Without loss of generality, we may assume that $\nu$ is nonnegative. Let $\lambda>0$ and $y\in \partial\xO$.
Since $K$ is compact, there exist $\xb_1=\xb_1(K,\xb_0)$ and points $\xi^1,\ldots ,\xi^{\ell}\in K$
such that $K_{\xb_1}
\subset \cup_{j=1}^\ell V(\xi^j,\frac{\xb_0}{4})=:B$.

We first assume that  $0<\xm<\frac{N^2}{4}$  and $y\in V(\xi^i,\frac{\xb_0}{4})$ for some $i=1,...,l$. Since $\xgp<0$,
we have from \eqref{Martinest1} that
\[
K_\xm(x,y)\asymp \frac{d(x)d^{-\xgp}(x)}{|x-y|^N}+\frac{d(x)d_K^\xgp(x)}{|x-y|^{N+2\xgp}},\qquad \mbox{ in }\xO\times\partial\xO.
\]
Set
\[
F_1(x,y)=\frac{d(x)d^{-\xgp}(x)}{|x-y|^N} ,\quad\quad F_2(x,y)=\frac{d(x)d_K^\xgp(x)}{|x-y|^{N+2\xgp}}
\]
and
\begin{align*}
A_\xl(y):=\big\{x\in \xO:\;F_2(x,y)>\xl \big\}, \quad\quad
m_{\xl}(y):=\int_{A_\xl(y)}d(x)d_K^{\xgp}(x) dx
\end{align*}
Note that
\[
|x-y|<\xl^{-\frac{1}{N+\xg_+-1}},\quad\forall x\in A_\xl(y)\cap\{d_K(x)> |x-y|\},
\]
thus
\[
\int_{A_\xl(y)\cap\{d_K(x)> |x-y|\}}d(x)d_K^{\xgp}(x) dx\leq \int_{\{|x-y|<\xl^{-\frac{1}{N+\xg_+ -1}}\}}|x-y|^{1+\xgp}dx=C
\xl^{-\frac{N+\xgp+1}{N+\xg_+ -1}}.
\]
Now,
\ba\label{w1}
d_K(x)\leq \xl^{-\frac{1}{N+\xgp-1}}\quad\text{and}\quad|x-y|<\xl^{-\frac{1}{N+2\xgp}}d_K^{\frac{\xgp+1}{N+2\xgp}}(x),\quad\forall x\in A_\xl(y)\cap\{d_K(x)\leq |x-y|\}.
\ea
Let $\xi\in K$ and $\xl>\big(\frac{4}{\xb_0}\big)^{N+\xgp+1}.$
By \eqref{propdist} we have
 \bal
 \int_{A_\xl(y)\cap\{d_K(x)\leq |x-y|\}\cap V(\xi^i,\frac{\xb_0}{4})}d(x)d_K^{\xgp}(x) dx\leq C_1\int_{A_\xl(y)\cap\{d_K(x)\leq |x-y|\}\cap V(\xi^i,\frac{\xb_0}{4})} (\xd_K^{\xi})^{1+\xgp}(x)dx .
\eal
Set $z_i= x_i-\Gamma_{i,K}^{\xi^i}(x'')$ for $i=1,...,k,$ $z_i=x_i-y_i$ for $i=k+1,...,N$ and $z=(z',z'')$ where
$z'=(z_1,...,z_k)$ and $z''=(z_{k+1},...,z_N).$
By \eqref{w1} and \eqref{propdist} we have
\[
|z'|\leq C_1\xl^{-\frac{1}{N+\xgp-1}}\qquad\text{ and }\qquad|z''|\leq
\xl^{-\frac{1}{N+2\xgp}}|z'|^{\frac{\xgp+1}{N+2\xgp}},
\]
thus
\bal
&\hspace{-1cm}\int_{A_\xl(y)\cap\{d_K(x) \leq |x-y|\}\cap V_K(\xi^i,\frac{\xb_0}{4})}
(\xd_K^{\xi})^{1+\xgp}(x)dx\\
 &\leq \int_{\{|z'|\leq C_1\xl^{-\frac{1}{N+\xgp-1}}\}}\int_{\{|z''|\leq
\xl^{-\frac{1}{N+2\xgp}}|z'|^{\frac{\xgp+1}{N+2\xgp}}}|z'|^{1+\xgp}dz''dz'\\
 &=C \xl^{-\frac{N+\xgp+1}{N+\xgp-1}}.
\eal

Hence,
\bal
& \hspace{-1cm}
\int_{A_\xl(y)\cap\{d_K(x)\leq |x-y|\}}d(x)d_K^{\xgp}(x) dx \\
& =
 \int_{A_\xl(y)\cap\{d_K(x)\leq |x-y|\}\cap B}d(x)d_K^{\xgp}(x) dx
 +\int_{A_\xl(y)\cap\{d_K(x) \leq |x-y|\}\setminus B}d
 (x)d_K^{\xgp}(x) dx\\
 &\leq C \xl^{-\frac{N+\xgp+1}{N+\xgp-1}}+\int_{A_\xl(y)\cap
 \{d_K(x) \leq |x-y|\}\setminus B}d(x)d_K^{\xgp}(x) dx.
\eal
But $d_K(x)>\xb_1$ for any $x\in \xO\setminus K_{\xb_1}$, hence the set
$A_\xl(y)\cap\{d_K(x)\leq |x-y|\}\setminus B$ is empty if $\lambda$ is large enough.

Recalling  also \eqref{eigenest}, we have thus proved that for large enough $\lambda>0$ we have
\[
\int_{A_\xl (y)} \phi_{\mu} \,  dx \leq
c\lambda^{-\frac{N+\xgp+1}{N+\xg_+ -1}}.
\]
Set $\mathbb{F}_i[\xn](x)=\int_{\partial\xO}F_i(x,y)d\xn(y)$,
$i=1,2$. We apply Proposition \ref{bvivier} with $\mathcal{H}(x,y)=F_2(x,y)$, $\eta=\ei$ and  $\omega=\xn$. We obtain
\be\label{F2weak estimate}
\norm{\mathbb{F}_2[\nu]}_{L_w^{\frac{N+\xgp+1}{N+\xgp-1}}(\Omega;\ei)} \leq
C\norm{\nu}_{\mathfrak{M}(\partial\xO)}.
	\ee

Similarly we can prove,

\be\label{F1weak estimate}
	\norm{\mathbb{F}_1[\nu]}_{L_w^{\frac{N+1}{N-1}}(\Omega;\ei)} \leq
	C\norm{\nu}_{\mathfrak{M}(\partial\xO)}.
	\ee
The desired result follows by \eqref{F2weak estimate}, \eqref{F1weak estimate}.

The proof of the other cases is similar and is omitted.
\end{proof}

\begin{theorem}\label{weakestk=N}
Let $K=\{0\} \subset\partial\Omega$,
 $\xm= \frac{N^2}{4}$ and assume that $\lambda_{\mu}>0$.
For $\lambda>e$ set
$$
A_\xl =\{x\in\xO:\;\; K_\xm(x,0)>\xl\}\quad.
$$
Then there exists a positive constant $C=C(\xO)$ such that
$$
m_{\xl}:=\int_{A_\xl}\ei dx\leq C \left(\xl^{-1}  |
\ln\xl|\right)^{\frac{N+2}{N-2}}.
$$
\end{theorem}
\begin{proof}
By \eqref{eigenest} we have $\xf_\xm(x) \asymp
 d(x)|x|^{-\frac{N}{2}}$
in $\Omega$ while by \eqref{Martinest2} we have
\[
K_{\mu}(x,0) \asymp \frac{d(x)}{|x|^{\frac{N}{2}}}
\big|\ln |x| \, \big| \, ,\quad\;  \text{in}\;\; \Omega.
\]
Now we note that if $r\in(0,\frac{1}{2})$ and $\xl>e,$ then
\begin{align}
-r^{-1}\ln r>\xl\Rightarrow r\leq \xl^{-1}\ln\xl.
\label{34}
\end{align}
Indeed, if $-r^{-1}\ln r>\xl$ and $r> \xl^{-1}\ln\xl,$ then
$-(\xl^{-1}\ln\xl )^{-1}\ln r>\xl$
and hence $r<\xl^{-1}$,
which is clearly a contradiction. Thus, by \eqref{34} and the fact that $d(x)\leq |x|$ in $\xO$, we can easily prove that
\begin{align}\label{35}
\{x\in\xO: \; K_\xm(x,0)>\xl\}\subset
\{x\in\xO: \; |x|\leq C\left(\xl^{-1}|\ln\xl|\right)^{\frac{2}{N-2}}\},
\end{align}
for some positive constant $C=C(\xO).$
From \eqref{35} follows that for $\lambda>e$ there holds
$$
\int_{A_\xl}\ei dx\leq C_1 \int_{ \{|x|\leq C\left(\xl^{-1}|\ln\xl|\right)^{\frac{2}{N-2}}\}}|x|^{-\frac{N-2}{2}} dx
\leq C_2 \left(\xl^{-1}|\ln\xl|\right)^{\frac{N+2}{N-2}},
$$
as required.
\end{proof}

\begin{theorem}
Let $\mu\leq k^2/4$ and assume that $\lambda_{\mu}>0$.
We set $p_{\partial\xO}=\frac{N+1}{N-1}$ and $p_K=\frac{N+\xgp+1}{N+\xgp-1}$. We then have

\noindent
$\ia$ Let $\xn\in \mathfrak{M}(\partial\xO)$ with compact support $F$, where $F\subset\partial\xO\setminus K.$ Then there exists a positive constant
$C=C(\Omega,K,\mu,\text{\rm dist}(F,K))$ such that
\[
\norm{\mathbb{K}_{\mu}[\nu]}_{L_w^{p_{\partial\xO}}(\Omega;\ei)} \leq
C\norm{\nu}_{\mathfrak{M}(\partial\xO)}.
\]
$\ib$ Assume in addition that  $\xm< \frac{N^2}{4}$.
There exists a positive constant $C=C(\Omega,K,\mu)$ such that
for any $\xn\in \mathfrak{M}(\partial\xO)$ with compact support in $K$
there holds
\[
\norm{\mathbb{K}_{\mu}[\nu]}_{L_w^{p_K}(\Omega;\ei)} \leq
C\norm{\nu}_{\mathfrak{M}(\partial\xO)}.
\]
$\ic$ Let $\xm= \frac{N^2}{4}$. For any $0<\xg<2$ there exists a
positive constant $C=C(\Omega,\mu,\xg)$ such that
for any $\xn\in \mathfrak{M}(\partial\xO)$ which is
concentrated at $0\in\partial\xO$ there holds
\[
\norm{\mathbb{K}_{\mu}[\nu]}_{L_w^{\frac{N+2}{N-\xg}}(\Omega;\ei)} \leq
C\norm{\nu}_{\mathfrak{M}(\partial\xO)}.
\]
\end{theorem}
\begin{proof}
The proof is similar to the proof of Theorem \ref{lpweakmartin1}, and we omit it. We note that for (iii) we use Theorem \ref{weakestk=N}.
\end{proof}

\subsection{Existence and Uniqueness}

Let $g: \mathbb{R}\rightarrow\mathbb{R}$ be a nondecreasing continuous function such that $g(0)=0.$

\begin{definition} \label{def:weak-sol}
 Let $\xn \in \mathfrak{M}(\partial\xO)$. A function $u$ is a weak solution of
\[
\left\{ \BAL
 L_\mu u+g(u)&=0\qquad \text{in }\;\Omega,\\
\tr(u)&=\nu,
\EAL \right.
\]
if $u\in L^1(\Omega;\ei)$, $g(u) \in L^1(\Omega;\ei)$ and
\[
\int_{\xO}u L_{\xm }\zeta \, dx+ \int_{\xO}g(u)\zeta \, dx =
\int_{\Omega} \mathbb{K}_{\xm}[\xn]L_{\xm }\zeta \, dx \, ,
\qquad\forall \zeta \in\mathbf{X}_\xm(\xO ,K).
\]
\end{definition}

\begin{lemma}[ {\cite[Lemma 5.1]{GkT1}} ]
\label{subcrcon} Assume $g$ satisfies
\[
\int_1^\infty  s^{-p-1}(\ln s)^{q} \big(g(s)-g(-s)\big)ds<\infty
\]
for some $p>1$ and $q \geq 0$. Let $v$ be a measurable function defined in
$\Omega$. For $s>0$ set
$$
E_s(v)=\{x\in \xO  : | v(x)|>s\} \quad \text{and} \quad
e(s)=\int_{E_s(v)} \ei dx.
$$
Assume that there exists a positive constant $C$ such that
\[
e(s) \leq Cs^{-p}(1+ \ln s)^q, \quad \forall s>1.
\]
Then $g(|v|), g(-|v|), g(v) \in L^1(\Omega;\ei)$.
\end{lemma}
\begin{theorem}\label{nonlinearexistence}
Let $\mu\leq k^2/4$, $\lambda_{\mu}>0$,
$\nu \in \GTM(\partial \Omega)$ and
$g: \mathbb{R}\rightarrow\mathbb{R}$ be a nondecreasing continuous function such that $g(0)=0$.  Assume that $g(\pm\BBK_\xm[\xn_\pm])\in L^1(\xO;\ei).$ Then  there exists a unique weak solution $u$ of \eqref{NLinintro}. Furthermore, there holds
\bal
u+\mathbb{G}_{\mu}[g(u)]=\mathbb{K}_{\xm}[\xn],\quad a.e.
\;\text{in}\;\xO .
\eal
\end{theorem}
\begin{proof}
First we assume that $g$ is bounded and set $M=\sup_{t\in\mathbb{R}}|g(t)|.$
Let $D\subset\subset \xO$ be an open smooth domain.
Then $v=\mathbb{K}_{\xm}[\xn]\in C^2(\overline{D}),$ since it is
$L_\xm$-harmonic in $\xO.$

We note that by \eqref{eigenest}, there exists $C=C(D,K,\xO,\xm)>1$ such that

\bal
C^{-1}\leq \ei(x)\leq C\quad
\text{and}\quad|\nabla\ei(x)|<C,\quad\forall x\in \overline{D}.
\eal
Let $\psi\in C^\infty_c(D).$ Writing $\psi=\ei\tilde\psi,$ we have
\small
\begin{eqnarray*}
\int_D|\nabla\psi|^2dx&=&\int_D|\nabla(\ei\tilde\psi)|^2dx
\leq
2\bigg(\int_D|\nabla \tilde\psi|^2\ei^2dx
+\int_D|\tilde\psi|^2|\nabla\ei|^2dx \bigg)\\
&\leq &C\bigg(\int_D|\nabla
\tilde\psi|^2\ei^2dx+\xl\int_D|\tilde\psi|^2\ei^2dx\bigg)=
C\bigg(\int_{D}|\nabla\psi|^2dx-\xm\int_D\frac{|\psi|^2}{d^2_K}dx\bigg).
\end{eqnarray*}
\normalsize
Thus for any  $w\in H_0^1(D)$ the following inequality holds
\ba
\int_{D}|\nabla w|^2dx-\xm\int_D\frac{|w|^2}{d^2_K}dx\leq \int_{D}|\nabla w|^2dx\leq C\left(\int_{D}|\nabla w|^2dx-\xm\int_D\frac{|w|^2}{d^2_K}dx\right).\label{coer}
\ea
Set $\widetilde{g}(t,x)=g(t+v(x))-g(v(x))$.
Let  $\CJ_\xm$ be the functional defined  by the expression
\[
\BAL
\CJ_\xm[w]=\frac{1}{2}\int_{D}\left(|\nabla w|^2-\frac{\mu}{d^2}w^2+2J[w]\right)dx+\int_{D}g(v)w \, dx,
\EAL
\]
where $J[w](x)=\int_{0}^{w(x)}\widetilde{g}(t,x)dt$ with domain
$$D(\CJ_\xm)=\{w\in H_0^1(D):J(w)\in L^1(D)\}.
$$
By \eqref{coer}, $\CJ_\mu$ is a convex lower semicontinuous and coercive functional. Thus there exists a minimizer $w\in H_0^1(\xO)\cap C^2(D)\cap C(\overline{D})$ such that
\[
\BAL
\left\{
\begin{array}{ll}
L_{\mu}w+g(w+v)=0, & {\text{in } D,}\\
w=0, & \text{on}\;\prt D.
\end{array}
\right.
\EAL
\]
The function $u_D=w+v$ then satisfies
\begin{equation}\label{M8}\BAL
\left\{
\begin{array}{ll}
L_{\mu}u_D+g(u_D)=0, & {\text{in } D,}\\
u_D=v, & \text{on}\;\prt D,
\end{array}
\right.
\EAL
\end{equation}
and $u_D+\BBG^{D}(g(u_D))=v$ in $D$.

Let $\{\xO_n\}$ be a smooth exhaustion of $\xO$ (see \eqref{redu1}) and $u_n$ be the solution of \eqref{M8}, with $D=\xO_n.$ Then,
$$
u_n+\BBG^{\xO_n}(g(u_n))=v,
$$
which implies
\be
|u_n|\leq \BBG^{\xO}(M)+|v|\leq C dd_K^{\min(\xg_+,0)}+|v|,
\qquad \text{ in }
\xO_n.
\label{exi1}
\ee
Thus there exists a subsequence, say $\{u_n\}$, such that $u_n\to u$ locally uniformly in $\xO.$
Also, by \eqref{exi1}, we have that $u_n\to u$ in $L^1(\Omega ;\ei)$ and by the dominated convergence theorem
$$
0=u_n+\BBG^{\xO_n}(g(u_n))-v\to u+\BBG^{\xO}(g(u))-v.
$$
Thus $u+\BBG(g(u))=v,$ i.e. $u$ is a weak solution of \eqref{def:weak-sol}.

Suppose now that $g$ is unbounded. We set
\bal
g_n(t)=\left\{
\begin{array}{ll}
n, &\text{ if }\; t\geq n, \\
g(t), & \text{ if}\; -n<t< n , \\
-n, &\text{ if}\; t\leq -n.
\end{array}
\right.
\eal
We denote by $\tilde u_n$ the weak solution of \eqref{def:weak-sol} with $g_n$ in place of $g$.
Applying \eqref{poi5} with $\xz=\ei$ we obtain
$$
-\int_{\Omega}(\tilde u_n-\BBK[\xn_+])_+\ei \, dx+ \int_{\Omega} \sign_+(\tilde u_n-\BBK[\xn_+])g_n(\tilde u_n)\ei\, dx \leq 0,
$$
which implies $\tilde u_n\leq \BBK[\xn_+]$.
Similarly, by \eqref{poi5} we have that
\be
-\BBK[\xn_-]\leq \tilde u_n\leq \BBK[\xn_+] .
\label{exi2}
\ee
Thus
\be
|g_n(\tilde u_n)|=g_n((\tilde u_n)_+)-g_n(-(\tilde u_n)_-)
\leq g_n(\BBK[\xn_+])-g_n(-\BBK[\xn_-])\leq g(\BBK[\xn_+])-g(-\BBK[\xn_-]).
\label{exi3}
\ee
From \cite[Lemma 1.5.3]{MVbook} and the above inequality
follows that
$\{\tilde u_n\}$ is uniformly bounded in $W^{1,p}(D)$ for any open smooth $D\subset\subset \xO$ and $1<p<\frac{N+1}{N-1}.$ Thus there exists a subsequence such that $\tilde u_n\to \tilde u$ a.e. in $\xO$.

Since $g(\pm\BBK[\xn_\pm])\in L^1(\Omega ;\ei)$, by \eqref{exi2} and \eqref{exi3}, we can easily prove that $\tilde u_n\to\tilde u $ and $g_n(\tilde u_n)\to \tilde g(\tilde u)$ in $L^1(\Omega ;\ei).$ As a consequence $\tilde u$ is a weak solution of \eqref{def:weak-sol}. Uniqueness follows by \eqref{poi4}.
\end{proof}
\medsk
\begin{proof}[\textbf{Proof of Theorem \ref{exist-subGKintro}}]
a) By Theorem \ref{lpweakmartin1} and \eqref{300}, we may apply Lemma \ref{subcrcon} with $p=\min\left(\frac{N+1}{N-1},\frac{N+\xgp+1}{N+\xgp-1}\right)$ and $q=0$ to conclude that $g(\pm\BBK_\xm[\xn_\pm])\in L^1(\xO;\ei).$ This and Theorem \ref{nonlinearexistence} imply the desired result.

Parts (b) and (c) are proved in a similarly.
\end{proof}

\noindent
{\bf\em Proof of Theorem \ref{exist-subGK4}.} By Theorem \ref{weakestk=N} and Lemma \ref{subcrcon} with $p=q=\frac{N+2}{N-2},$ we have that $g(\BBK_\xm[\xa\xd_0])\in L^1(\xO;\ei),$ which, together with Theorem \ref{nonlinearexistence}, implies the desired result.

\appendix\section{Pointwise estimates on eigenfunctions}
\setcounter{equation}{0}
In this appendix, we prove sharp two-sided pointwise estimates for eigenfunctions of \eqref{eigenvalue}. Let $\xb>0$ be small enough and $\xG=\partial\xO$ or $K.$ Let $\eta_{\xb,\xG}\in C^\infty_c(\xG_\xb)$ be such that $0\leq\eta_{\xb, \Gamma} \leq 1$ in $\BBR^N$ and $\eta=1$ in $\overline{\xG}_{\frac{\xb}{2}}.$ We set
\bal
\xz_\xb=(1-\eta_{4\xb,\partial\xO})+\eta_{4\xb,\partial\xO}
d\big( (1-\eta_{\xb,K})+\eta_{\xb,K} \tilde d_K^{\xg_+} \big)
\quad\text{in}\;\xO.
\eal
We consider the minimizing problem
\[
\xl_\xm:=\inf_{u\in C_c^\infty(\xO)\setminus \{0\}}\frac{\int_\xO|\nabla u|^2dx-\xm\int_\xO \frac{u^2}{d_K^2}dx}{\int_\xO u^2 dx}>-\infty.
\]
Setting $u=\xz_\xb v$ we obtain that
\ba\label{eigenvalueap2}
\xl_\xm =\inf_{v\in C_c^\infty(\xO)\setminus \{0\}}\frac{\int_\xO\xz^2_\xb|\nabla v|^2dx-\int_\xO v^2(\xz_\xb\xD\xz_\xb+\xm\frac{\xz^2_\xb}{d_K^2})dx}{\int_\xO \xz_\xb^2 u^2 dx}.
\ea
By \cite[Lemma 3.1]{FF} there exists $\xb_0$ and a positive constant $C=C(\xO,K,\xb_0)$ such that
\ba\label{logfall}
\int_{K_{\xb_0}\cap\xO}|\nabla u|^2 dx-\frac{k^2}{4}\int_{K_{\xb_0}\cap\xO}\frac{u^2}{d_K^2}dx\geq C \int_{K_{\xb_0}\cap\xO}\frac{|u|^2}{d_K^2|\ln d_K|^2}dx,\quad\forall u\in C_c^\infty(K_{\xb_0}\cap\xO).
\ea

In view of the proof of Lemma \ref{subsup}, for $\xe>0$ there exist positive constants $M=(\xO, K,\xe)$  and $\xb_1=\xb_1(\xO,K,\xe)$ such that the function
\bal
\tilde \xf := e^{M d}d\tilde d_{K}^{\xg_+}+d\tilde d_{K}^{\xg_++\xe}
\, \asymp \, d\tilde d_{K}^{\xg_+}
\eal
satisfies $L_\xm\tilde\xf\leq0$ in $K_{\xb_1}\cap\xO $.

Now let $u\in C_c^\infty(K_{\xb_1}\cap\xO)$. Setting $u=\tilde\xf v,$ by \eqref{logfall} we have
\ba
\label{logfall2}
\int_{K_{\xb_1}\cap\xO}d^2\tilde d_{K}^{2\xg_+}|\nabla v|^2 dx\geq C \int_{K_{\xb_1}\cap\xO}\frac{d^2 v^2}{\tilde d_K^{2-2\xg_+}|\ln \tilde d_K|^2}dx,\quad\forall v\in C_c^\infty(K_{\xb_1}\cap\xO)
\ea
Now, by \cite[Theorem 3.2]{FMT}, there exists $\xb_2=\xb_2(\xO)>0$ such that
\[
\int_{\xO_{\xb_2}}|\nabla u|^2 dx\geq
\frac{1}{4}\int_{\xO_{\xb_2}}\frac{u^2}{d^2}dx \, ,
\qquad\forall u \in C_c^\infty(\xO_{\xb_2}).
\]
Setting $u=dv,$  we have that there exists a positive constant $\xb_3=\xb_3(\xO)<\xb_2$ such that
\ba\label{classicalhardy2}
\int_{\xO_{\xb_3}}d^2|\nabla v|^2 dx\geq\frac{1}{8}
\int_{\xO_{\xb_3}}v^2dx \, , \qquad\forall v \in C_c^\infty(\xO_{\xb_3}).
\ea
We denote by  $H^1_0(\xO;d^2\tilde d_K^{2\xg_+})$ the closure of
$C_c^\infty(\xO)$ in the norm
\bal
\norm{u}^2_{H^1(\xO; d^2\tilde d_K^{2\xg_+})}=\int_\xO u^2d^2\tilde d_K^{2\xg_+} dx+\int_\xO|\nabla u|^2d^2\tilde d_K^{2\xg_+} dx
\eal
\begin{proposition}
Let $\xm\leq\frac{k^2}{4}$ and $\xb\leq\frac{1}{16}\min(\xb_3,\xb_1).$ Then there exists a minimizer $v_\xm\in H^1_0(\xO;d^2\tilde d_K^{2\xg_+})$ of \eqref{eigenvalueap2}.
\end{proposition}
\begin{proof}
Let $\{w_k\}_k\subset C_c^\infty(\xO)$ be a minimizing sequence of \eqref{eigenvalueap2} normalized by
\bal
\int_\xO\zeta_\xb^2w_k^2 dx=1 , \quad\quad  k\in\BBN.
\eal
First we note that $\zeta_\xb^2 \asymp d^2\tilde d_K^{2\xg_+}$
in $\xO$ and
\ba
\label{a7}
\Big|\xz_\xb\xD\xz_\xb+\xm\frac{\xz^2_\xb}{d_K^2}\Big|
\leq Cd\tilde d_K^{2\xg_+}\, , \qquad\text{in}\; K_{\frac{\xb}{2}},
\ea
where $C$ depends only on $\xO,K$ and $\xb_0$.
For any $v\in  C_c^\infty(K_{\xb_5}\cap\xO)$ we have
\bal
\int_{K_{\xb_5}\cap\xO} d \tilde d_K^{2\xg_+-\frac{1}{2}}
v^2dx=\frac{1}{2}\int_{K_{\xb_5}\cap\xO}
d_K^{2\xg_+-\frac{1}{2}}(\nabla d^2\cdot \nabla d) v^2dx \, ,
\eal
so by integration by parts, H\"{o}lder inequality, Proposition \ref{propdK} (b) and \eqref{logfall2}, we find that for any $\xe>0$ there exits $\xb_5=\xb_5(\xO,K,\xe)$ such that
\ba
\label{200}
\int_{K_{\xb_5}\cap\xO} d \tilde d_K^{2\xg_+-\frac{1}{2}}  v^2dx
\leq \xe\int_{K_{\xb_5}\cap\xO}|\nabla v|^2d^2\tilde d_K^{2\xg_+} dx,
\ea
Now, there holds
\[
\Big|\xz_\xb\xD\xz_\xb+\xm\frac{\xz^2_\xb}{d_K^2} \Big| \leq Cd \, ,
\quad\quad\text{in}\; \xO\setminus K_{\frac{\xb}{2}},
\]
where $C$ depends only on $\xO,K$ and $\xb_0$.

Let $r>0.$ By \eqref{classicalhardy2} and proceeding as in the proof of \eqref{200}, we have that for any $\xe>0$ there exists $\xb_6=\xb_6(\xO,K,\xe,r)$ such that
\[
\int_{\xO_{\xb_6}\setminus K_{r}} d|v|^2dx\leq  \xe\int_{\xO_{\xb_6}\setminus K_{r}}|\nabla v|^2d^2\tilde d_K^{2\xg_+} dx \, ,
\qquad\forall v\in C_c^\infty(\xO_{\xb_6}\setminus K_{r}).
\]
Combining all above, we may deduce that for any $\xe>0$ there exists $M(\xe,\xb)$ such that

\[
\Big|\int_\xO w^2_k(\xz_\xb\xD\xz_\xb+\xm\frac{\xz^2_\xb}{d_K^2})dx
\Big| \leq \xe \int_\xO\xz^2_\xb|\nabla w_k|^2dx+ M.
\]
Hence, the sequence $\{w_k\}$ is uniformly bounded in $H^1_0(\xO;d^2\tilde d_K^{2\xg_+}).$ Thus there exists $v_\xm\in H^1_0(\xO;d^2\tilde d_K^{2\xg_+})$ and a subsequence $w_k,$ denoted by the same index $k,$ such that $w_k\rightharpoonup v_\xm$ in $H^1_0(\xO;d^2\tilde d_K^{2\xg_+})$;
it follows that $w_k\to v_\xm$ in $L^2_{loc}(\xO)$ and a.e. in $\xO.$

By compactness we have that $w_k\to v_\xm$ in $L^2(\xO;\zeta_\xb^2)$.
Moreover, from \eqref{200} and \eqref{classicalhardy2} we have
\bal
\int_\xO w^2_k(\xz_\xb\xD\xz_\xb+\xm\frac{\xz^2_\xb}{d_K^2})dx\to \int_\xO v^2_\xm(\xz_\xb\xD\xz_\xb+\xm\frac{\xz^2_\xb}{d_K^2})dx.
\eal
The desired result now follows by the lower semicontinuity of the gradient term.
\end{proof}
\begin{proposition}
Let $\xm\leq\frac{k^2}{4}$  and $\xb\leq\frac{1}{16}\min(\xb_3,\xb_1)$.
The function $\xf_\xm=v_\xm\xz_\xb$ satisfies
\bal
L_\xm\ei=\xl_\xm\ei , \qquad \mbox{ in } \xO.
\eal
and has the asymptotics
\bal
\ei\asymp d\tilde d_K^{\xg_+} , \qquad\text{in}\;\;\xO.
\eal
\end{proposition}
\begin{proof}
First we note that $\xz_\xb\asymp d\tilde d_K^{\xg_+}$.
Furthermore $(1-\eta_{\xb,K})\xf_\xm\in H_0^1(\xO)$ for small
$\xb>0$. Hence by standard elliptic theory, we have that for any $r>0$ there exists $C=C(r,\xO,K,\xm)$ such that
\[
\ei\asymp C d\quad \text{in}\; \xO\setminus K_r,
\]
which implies
\[
v_\xm\asymp C\quad\text{in}\;\;\xO\setminus K_r.
\]
We will show that $v_\xm\geq c$ in $\xO$.
Let $\xL>-\xl_\xm$.
For any $\xe\in(0,1)$, there exists $\xb_0<\frac{\xb}{4}$ such that
the function
\bal
\tilde \xf= e^{M d}d\tilde d_{K}^{\xg_+}+d\tilde d_{K}^{\xg_++\xe}\asymp d\tilde d_{K}^{\xg_+}\quad\text{ in}\; K_{\xb_0}\cap\xO
\eal
satisfies
\begin{equation}
L_\xm\tilde\xf +\Lambda\tilde\xf\leq0 \, ,
\qquad \mbox{ in }K_{\xb_0}\cap\xO .
\label{eq:dimi}
\end{equation}
Set $\xf=C\xz_\xb^{-1}\tilde\xf =C(e^{Md} +\tilde{d}_K^{\epsilon})$,
where $C>0$ is a constant such that
$\xf\leq \frac{1}{2}v_\xm$ in $\partial K_{\xb_0}\cap\xO$.
By \eqref{eq:dimi} and because $v_{\mu}$ satisfies the Euler equation
for \eqref{eigenvalueap2}, we have
\bal
-\div(\xz^2_\xb\nabla(\xf-v_\xm))-(\xf-v_\xm)(\xz_\xb\xD\xz_\xb+\xm\frac{\xz^2_\xb}{d_K^2})+\xL\xz^2_\xb(\xf-v_\xm)\leq 0,
\quad\text{in}\;\;K_{\xb_0}\cap\xO .
\eal
By Theorem \ref{density}, we may take $g=(\xf-v_\xm)_+$  as test function  in the above inequality. Therefore,
\ba\label{202}
\int_{K_{\xb_0}\cap\xO}\xz^2_\xb|\nabla g|^2 dx
-\int_{K_{\xb_0}\cap\xO}g^2(\xz_\xb\xD\xz_\xb+\xm\frac{\xz^2_\xb}{d_K^2})dx+\xL \int_{K_{\xb_0}\cap\xO}g^2\xz_\xb^2 dx\leq0,
\ea
But, by \eqref{eigenvalueap2} we have
\bal
\int_{K_{\xb_0}\cap\xO}\xz^2_\xb|\nabla g|^2 dx
-\int_{K_{\xb_0}\cap\xO}g^2(\xz_\xb\xD\xz_\xb
+\xm\frac{\xz^2_\xb}{d_K^2})dx\geq\xl_\xm
\int_{K_{\xb_0}\cap\xO}g^2\xz_\xb^2 dx .
\eal
This, together with \eqref{202}, implies $g=0$ since $\xL>-\xl_\xm$.
Hence $v_\xm\geq c$ in $\xO$.

Next we will similarly prove that $v_\xm\leq c$ in $\xO$.
As in the proof of Lemma \ref{subsup}, for $\xe\in(0,1)$
there exists $\xb_0<\frac{\xb}{4}$ such that the function
\bal
\tilde \xz= e^{-M d}d\tilde d_{K}^{\xg_+}-d\tilde d_{K}^{\xg_++\xe}\asymp d\tilde d_{K}^{\xg_+}\quad\text{ in}\; K_{\xb_0}\cap\xO
\eal
satisfies $L_\xm\tilde\xz-\xl_\xm\tilde\xz\geq0$ in $K_{\xb_0}\cap\xO.$
Set $\xz=C\xz_\xb^{-1}\tilde\xz,$  where $C>0$ is a constant such that
\bal
\xz\geq 2v_\xm \,  , \qquad\text{in}\;\;\partial K_{\xb_0}\cap\xO.
\eal
This time we have
\bal
-\div\big(\xz^2_\xb\nabla(v_\xm-\xz)\big)
-(v_\xm-\xz)(\xz_\xb\xD\xz_\xb+\xm\frac{\xz^2_\xb}{d_K^2})
\leq\xl_\xm\xz^2_\xb(v_\xm-\xz) , \qquad\text{in}\;\;K_{\xb_0}\cap\xO .
\eal
Hence, we may take $g=(v_\xm-\xz)_+$  as test function  in the above inequality. Therefore,
\bal
\int_{K_{\xb_0}\cap\xO}\xz^2_\xb|\nabla g|^2 dx+\int_{K_{\xb_0}\cap\xO}g^2(\xz_\xb\xD\xz_\xb+\xm\frac{\xz^2_\xb}{d_K^2})dx\leq\xl_\xm \int_{K_{\xb_0}\cap\xO}g^2\xz_\xb^2 dx.
\eal
By \eqref{logfall2}, \eqref{a7}, \eqref{200}
and the above inequality we obtain
\bal
C \int_{K_{\xb_0}\cap\xO}\frac{d^2 g^2}{\tilde d_K^{2-2\xg_+}|\ln \tilde d_K|^2}dx\leq \xl_\xm \int_{K_{\xb_0}\cap\xO}g^2\xz_\xb^2 dx,
\eal
which implies that $g=0,$ provided $\xb_0$ is small enough. Hence, $v_\xm\leq  c$ in $\xO$ and the result follows.
\end{proof}

\medskip
\noindent \textbf{Acknowledgement.} K. T. Gkikas acknowledges support by the Hellenic Foundation for Research and Innovation
(H.F.R.I.) under the “2nd Call for H.F.R.I. Research Projects to support Post-Doctoral Researchers” (Project
Number: 59).


\begin{thebibliography}{99}



\bibitem[A]{An} A. Ancona, \emph{Negatively curved manifolds, elliptic operators and the Martin boundary}, Annals Math. 2nd Series {\bf 125} (1987), 495--536.


\bibitem[BEL]{BEL}A.A. Balinsky, W.D. Evans and R.T. Lewis,
{\em The analysis and geometry of Hardy's inequality.} Universitext. Springer, Cham, 2015. xv+263 pp.


\bibitem[BG]{BG} P. Baras and J. Goldstein, \emph{The heat equation with a singular potential.} Trans. Amer. Math.
Soc. 284, 121–139 (1984)


\bibitem[BFT1]{BFT} G. Barbatis, S. Filippas and A. Tertikas, \emph{A unified approach to improved $L^p$ Hardy inequalities with best constants.} Trans. Amer. Math. Soc. {\bf 356} (2004), 2169-2196.

\bibitem[BFT2]{BFT2} G. Barbatis, S. Filippas and A. Tertikas,
{\em Critical heat kernel estimates for Schr\"odinger operators via Hardy-Sobolev inequalities,} J. Funct. Anal. 208 (2004), 1-30.

\bibitem[BFT3]{BFT3} G. Barbatis, S. Filippas and A. Tertikas,
{\em Sharp Hardy and Hardy-Sobolev inequalities with point singularities on
the boundary,} J. Math. Pures Appl. {\bf 117} (2018), 146–184.

\bibitem[B-VV]{BVi} M.F. Bidaut-V\'eron and L. Vivier, {\em An elliptic
semilinear equation with source term involving boundary measures: the
subcritical case}, Rev. Mat. Iberoamericana {\bf 16} (2000), 477--513.

\bibitem[BM]{BM} H. Brezis and M. Marcus, \emph{Hardy's inequalities revisited}, Ann. Sc. Norm. Super. Pisa Cl. Sci. (5) 25  (1997), 217-237.

\bibitem[BMS]{BMS} H. Brezis, M. Marcus and I. Shafrir,
{\em Extremal functions for Hardy's inequality with weight}, J. Funct. Anal. 171 (2000), no. 1, 177–191.

\bibitem[CM]{CM} X. Cabr\'e, and Y. Martel,
{\em Existence versus instantaneous blowup for linear heat equations with singular potentials,}
C. R. Acad. Sci. Paris Sér. I Math. 329 (1999), no. 11, 973–978.

\bibitem[CFMS]{caffa} L. Caffarelli, E. Fabes, S. Mortola and S. Salsa, \emph{Boundary behavior of nonnegative solutions of elliptic operators in divergence form}, Indiana Univ. Math. J. {\bf 30} (1981), 621--640.


\bibitem[CV1]{CheVer} H. Chen and L. V\'eron, \emph{Weak solutions of semilinear elliptic equations with Leray–Hardy potential and
measure data.} Math. Eng. 1, 391–418 (2019)

\bibitem[CV2]{CV} H. Chen and L. V\'eron, Schr\"odinger operators with Leray Hardy potential singular on the boundary,  Journal of Differential Equations  269 (2020), 2091-2131.


\bibitem[C]{C}  F. C\^irstea, \emph{A complete classification of the
isolated singularities for nonlinear elliptic equations with inverse square
potentials},  Mem. Amer. Math. Soc. 227 (2014), no. 1068, vi+85 pp.


\bibitem[D1]{D1} E.B. Davies, \emph{Perturbations of ultracontractive semigroups.} Quart. J. Math. Oxford 2(37), 167-176 (1986)

\bibitem[D2]{D2} E.B. Davies, \emph{Heat kernels and spectral theory.} Cambridge: Cambridge University Press, 1989

\bibitem[DS]{DS1} E.B. Davies, E.B. and B. Simon,
\emph{Ultracontractivity and the heat kernels for Schr\"odinger operators
and Dirichlet Laplacians,} J. Funct. Anal. 59  (1984), 335–395

\bibitem[DD1]{DD1} J. D\'avila and L. Dupaigne, {\em Comparison results for PDEs with a singular potential}, Proc. Roy. Soc. Edinburgh Sect. A {\bf 133} (2003), 61--83.

\bibitem[DD2]{DD2} J. D\'avila and L. Dupaigne, {\em Hardy-type
inequalities}, J. Eur. Math. Soc. {\bf 6} (2004), 335--365.

\bibitem[DN]{DN} L. Dupaigne and G. Nedev, {\em Semilinear elliptic PDE’s with a singular potential}, Adv. Differential Equations {\bf 7} (2002), 973--1002.


\bibitem[FM]{FF} M. Fall and F. Mahmoudi,{\em Weighted Hardy inequality with higher dimensional singularity on the boundary.}
Calc. Var. Partial Differential Equations 50 (2014), no. 3-4, 779-798.


\bibitem[FT]{FT} S. Filippas, A. Tertikas,\emph{ Optimizing improved Hardy inequalities.} J. Funct. Anal. {\bf 192} (2002), 186-233.


\bibitem[FMaT]{FMT}  S. Filippas, V. Maz'ya and A. Tertikas,
\emph{Critical Hardy-Sobolev inequalities}, J. Math. Pures Appl. {\bf 87} (2007), 37--56.


\bibitem[FMT1]{FMT2} S. Filippas, L. Moschini and A. Tertikas, \emph{Sharp two-sided heat kernel estimates for critical Schr\"odinger operators on bounded domains}, Comm. Math. Phys. {\bf 273} (2007), 237--281.

\bibitem[FMT2]{FMoT} S. Filippas, L. Moschini and A. Tertikas, \emph{Improving $L^2$ estimates to Harnack inequalities}, Proc. Lond. Math. Soc. {\bf 99} (2009), 326--352.

\bibitem[GN1]{GkT1} K. T. Gkikas and P.-T. Nguyen, \emph{Semilinear
elliptic Schr\"{o}dinger equations with singular potentials and
source terms,} preprint.

\bibitem[GN2]{GkT2} K. T. Gkikas and P.-T. Nguyen, \emph{Semilinear
elliptic Schr\"odinger equations with singular potentials and
absorption terms,} preprint.

\bibitem[GN3]{gktai}K. T. Gkikas and P.-T. Nguyen, \emph{Martin kernel of
Schr\"odinger operators with singular potentials and applications to B.V.P.
for linear elliptic equations}, Calculus of Variations and PDE {\bf 61}, 1 (2022).

\bibitem[GV]{GkV} K. T. Gkikas and L. V\'eron, \emph{Boundary
singularities of
solutions of semilinear elliptic equations with critical Hardy potentials},
Nonlinear Anal.  121 (2015), 469--540.


\bibitem[GM]{GM}N. Ghoussoub and A. Moradifam,
{\em Functional inequalities: new perspectives and new applications.}
Mathematical Surveys and Monographs, 187. American Mathematical Society, Providence, RI, 2013. xxiv+299.


\bibitem[GS-C]{GS} A. Grigoryan and L. Saloff-Coste, \emph{Stability results for Harnack inequalities.} Ann. Inst. Fourier, Grenoble 55(3), 825–890 (2005)

\bibitem[HW]{hunt} R.A. Hunt and  R. L. Wheeden, \emph{Positive harmonic functions on Lipschitz domains}, Transactions of the American Mathematical Society {\bf 147} (1970), 507--527.

\bibitem[K]{Kuf} A. Kufner, \emph{Weighted Sobolev spaces},
Teubner-Textezur Mathematik 31 (Teubner, Leipzig, 1981).

\bibitem[LS]{LS}V. Liskevich and Z. Sobol,
{\em Estimates of integral kernels for semigroups associated with second-
order elliptic operators with singular coefficients,}
Potential Anal. 18 (2003), no. 4, 359–390.

\bibitem[M1]{M1} M. Marcus, \emph{Estimates of sub and super solutions of Schrödinger equations with very singular potentials,} (2019) arXiv:1912.01283

\bibitem[M2]{M2} M. Marcus, \emph{Estimates of Green and Martin
kernels for Schr\" odinger operators with singular potential in Lipschitz
domains}, Ann. Inst. H. Poincar\'e Anal. Non Lin\'eaire {\bf 36} (2019),
1183--1200.

\bibitem[MMP]{hardy-marcus} M. Marcus, V. J. Mizel and Y. Pinchover,
\emph{On the best constant for Hardy's inequality in $\mathbb{R}^n$},
Trans. Amer. Math. Soc. 350 (1998), 3237-3255.

\bibitem[MM]{MarMor} M. Marcus and V. Moroz, \emph{Moderate
solutions of semilinear elliptic equations with Hardy potential under
minimal restrictions on the potential}, Ann. Sc. Norm. Super. Pisa Cl. Sci.
{\bf 18} (2018), 39--64.

\bibitem[MN1]{MT} M. Marcus and P.-T. Nguyen, \emph{Schr\"odinger
equations with singular potentials: linear and nonlinear boundary value
problems.} Math. Ann. 374 (2019), no. 1-2, 361-394.

\bibitem[MN2]{MarNgu} M. Marcus and P.-T. Nguyen, \emph{Moderate
solutions of semilinear elliptic equations with Hardy potential}, Ann. Inst.
H. Poincar\'e Anal. Non Lin\'eaire {\bf 34} (2017), 69--88.

\bibitem[MV1]{MVbook} M. Marcus and L. V\'eron, {\em Nonlinear
second order elliptic equations involving measures}, De Gruyter Series in
Nonlinear Analysis and Applications, 2013.

\bibitem[MV2]{MV} M. Marcus and L. V\'{e}ron, \emph{ Boundary trace of
positive solutions of semilinear elliptic equations in Lipschitz domains: the
subcritical case},  Ann. Sc. Norm. Super. Pisa Cl. Sci.  {\bf 10} (2011),
913-984.

\bibitem[M]{Mar} R.S. Martin, {\em Minimal positive harmonic functions},
Trans. Amer. Math. Soc. {\bf 49} (1941), 137–172.

\bibitem[MS]{MS}P.D. Milman and Yu.A. Semenov,
{\em Heat kernel bounds and desingularizing weights,}
J. Funct. Anal. 202 (2003), 1–24.

\bibitem[MT1]{MoT} L. Moschini and A. Tesei, \emph{Harnack inequality
and heat kernel estimates for the Schr\"odinger operator with Hardy
potential,} Rend. Mat. Acc. Lincei s. 9 v. 16,  (2005) 171-180.

\bibitem[MT2]{MoT2} L. Moschini and A. Tesei,
\emph{Parabolic Harnack inequality for the heat equation with inverse-square potential,} Forum Math. 19 (2007), no. 3, 407–427.

\bibitem[S-C]{SC1}{L. Saloff-Coste, \emph{Aspects of Sobolev-Type
Inequalities,} Cambridge Univ. Press, Cambridge, 2002.}

\bibitem[VZ]{VZ} J.-L. V\'azquez and E. Zuazua,\emph{ The Hardy inequality and the asymptotic behaviour of the heat equation
with an inverse-square potential.} J. Funct. Anal. 173, 103–153 (2000)

\bibitem[V]{Vbook} L. V\'{e}ron, \emph{Singularities of Solutions of Second Order Quasilinear Equations}, Pitman Research Notes in Math. Series 353, (1996).

\bibitem[Z]{Zhang}Q. S. Zhang,
{\em The boundary behavior of heat kernels of Dirichlet Laplacians.}
J. Differential Equations 182 (2002), no. 2, 416-430.



\end{thebibliography}
\end{document}